\documentclass[11pt]{amsart}
\usepackage{geometry}
\usepackage{graphicx}
\usepackage{amssymb}
\usepackage{mathtools}
\usepackage{ tipa }
\usepackage{ tikz-cd}
\usepackage{rotating}
\usepackage{enumitem}
\usepackage{hyperref}
\usepackage{scalerel}

\newcommand{\anchor}{{{\scalerel*{\includegraphics{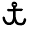}}{A}}}}
\newcommand{\boat}{{{\scalerel*{\includegraphics{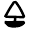}}{A}}}}

\usepackage{tikz-cd}

\usepackage{bm}

\newcommand{\CC}{\mathbb C}
\newcommand{\DD}{\mathbb D}

\newcommand{\FF}{\mathbb F}

\newcommand{\HH}{\mathbb H}
\newcommand{\II}{\mathbb I}

\newcommand{\RR}{\mathbb R}

\newcommand{\ZZ}{\mathbb Z}

\newcommand{\cA}{\mathcal{A}}

\newcommand{\cD}{\mathcal{D}}

\newcommand{\cG}{\mathcal{G}}

\newcommand{\cI}{\mathcal{I}}

\newcommand{\cM}{\mathcal{M}}
\newcommand{\cN}{\mathcal{N}}

\newcommand{\cT}{\mathcal{T}}

\let\emptyset\varnothing

	\newcommand{\cfh}{\widehat{CF}}

	\newcommand{\hfh}{\widehat{HF}}

	\newcommand{\cfkm}{CFK^-}

	\newcommand{\cflm}{CFL^-}

	\newcommand{\bsa}{{{BSA}}}

	\newcommand{\cfa}{CFA}
	\newcommand{\cfd}{{CFD}}
	\newcommand{\cfda}{{CFDA}}

\newcommand{\balpha}{{\pmb{\alpha}}}
\newcommand{\bbeta}{{\pmb{\beta}}}
\newcommand{\bgamma}{\pmb{\gamma}}
\newcommand{\bdelta}{\pmb{\delta}}

\newcommand{\blambda}{\pmb{\lambda}}

\newcommand{\btheta}{ \pmb{\theta} }

\newcommand{\ba}{{\bf a } }
\newcommand{\bb}{{\bf b } }

\newcommand{\bn}{{\bf n } }

\newcommand{\bp}{{\bf p } }
\newcommand{\bx}{{\bf x } }
\newcommand{\by}{{\bf y } }
\newcommand{\bz}{{ \bf z} }
\newcommand{\bw}{{ \bf w} }

\newcommand{\bC}{{ \pmb C} }
\newcommand{\bI}{{ \pmb I} }
\newcommand{\bJ}{{ \pmb J} }

\newcommand{\bbox}{\boxtimes}

\newcommand{\ind}{\mathrm{ind}}
\newcommand{\arrowrho}{\overrightarrow{\rho}}
\newcommand{\arrowdelta}{\overrightarrow{\delta}}
\newcommand{\arrowlambda}{\overrightarrow{\lambda}}

\newcommand{\conf}{\mathrm{Conf}}
\newcommand{\arrowP}{\overrightarrow{P}}
\newcommand{\freemod}{\cM_{\leftrightarrow}}

\newcommand{\arrowiota}{\overrightarrow{\iota}}

\newcommand{\inte}{\mathrm{Int}}

\newcommand{\handle}{\frak{H}}

\newcommand{\border}[1]{{\pmb #1}}

\newcommand{\glue}{\mbox{\textscriptg lue}}

\newcommand{\punc}{\mbox{\textwynn unc}}

\newcommand{\emb}{\mathrm{emb}}

\newcommand{\start}{\mathrm{start}}

\newcommand{\fin}{\mathrm{end}}
\newcommand{\jump}{\mathrm{jump}}

\newcommand{\tildemod}{\widetilde{\cM}}

\newcommand{\ev}{\mathrm{ev}}

\newcommand{\delbar}{\bar{\partial}}

\newcommand{\reg}{\mbox{reg}}

\newcommand{\dom}{\mbox{Dom}}

\newcommand{\poly}{\mathrm{Poly}}

\newcommand{\col}{\mathrm{Col}}

\newcommand{\ac}{\mathrm{ac}}
\newcommand{\ob}{\mathrm{ob}}

\theoremstyle{plain}
\newtheorem{thm}{Theorem}

\newtheorem{lem}[thm]{Lemma}
\newtheorem{prop}[thm]{Proposition}

\theoremstyle{definition}
\newtheorem{defn}[thm]{Definition}
\newtheorem{exmp}[thm]{Example}

\theoremstyle{plain}
\newtheorem{rmk}[thm]{Remark}
\newtheorem{convention}[thm]{Convention}

\newlist{Mlist}{enumerate}{1}
\setlist[Mlist]{label=(M-\arabic*):}

\newlist{BOR}{enumerate}{1}
\setlist[BOR]{label=(BOR-\arabic*):}

\newlist{FIL}{enumerate}{1}
\setlist[FIL]{label=(FIL-\arabic*):}

\title[The bordered Floer homology of link complements]{Holomorphic polygons and the bordered Heegaard Floer homology of link complements}
\author{Thomas Hockenhull}

\thanks{The author was supported by an EPSRC studentship from 2013-2017, and by Imperial College London mathematics department from 2017-2018. Some of this work was undertaken whilst visiting Princeton University, which was funded by the Doris Chen award.}

\begin{document}
\begin{abstract}In this paper, we describe the construction of an $\cA_\infty$ multi-module in terms of counts of holomorphic polygons in a series Heegaard multi-diagrams. We show that this is quasi-isomorphic to the type-A bordered-sutured invariant of a link complement with a view to calculating, in the sequel, these invariants in terms of the link Floer homology of the corresponding link.
\end{abstract}
\maketitle

\section{Introduction}
Suppose that $L \subset S^3$ is a link of $h$ components, and $\Lambda$ is a choice of framing for $L$. The purpose of this paper is to introduce an $\cA_\infty$ multi-module associated to $(L, \Lambda)$ over $\cT^h$, where $\cT$ is the algebra associated to a parameterised torus, as in \cite{LOT}: we call this $\poly(L, \Lambda)$. The $\cA_\infty$ quasi-isomorphism type of $\poly(L, \Lambda)$ is an invariant of the pair $(L, \Lambda)$, or, indeed, the three-manifold with parameterised boundary $(S^3 - L, \Lambda)$ --- and is closely related to the bordered-sutured Heegaard Floer invariants of this three-manifold.

$\poly(L, \Lambda)$ is defined in terms of \emph{holomorphic polygons}, and is closely related to the invariants $\cfh$ for closed three-manifolds, as defined in \cite{Ozsvath-Szabo:2001}. At its root, the chain complex $\cfh(Y)$ is defined as a chain complex arising from counting holomorphic discs in a symplectic manifold associated to a \emph{Heegaard diagram} corresponding with $Y$; that is, $\cfh(Y)$ is $\cfh(\cD)$ for some Heegaard diagram $\cD$ representing $Y$. More generally, one may consider a \emph{Heegaard $n$-diagram} for any $n \ge 2$, and count corresponding $n$-gons in an associated symplectic manifold. A Heegaard $m$-diagram corresponds with a sequence $Y^{0,1}, Y^{1,2} \dots, Y^{n-2,n-1}, Y^{n-1,0}$ of three-manifolds, and counting holomorphic $n$-gons yields maps
$$m_n: \cfh(Y^{0,1}) \otimes \cdots \otimes \cfh(Y^{n-2,n-1}) \rightarrow \cfh(Y^{n-1, 0}).$$

These are well-known to satisfy an $\cA_\infty$ relation, which is shown by counting the number of ends of one-dimensional moduli spaces and showing that they are of even parity. These moduli spaces decompose by \emph{homology classes}, and, in fact, the parity of the ends of a one-dimensional moduli space remains even when restricted to curves of a given homology class. We exploit this extra information to define $\poly(L, \Lambda)$.

Broadly speaking, this multi-module is defined by associating a series of $h$ sequences of elements of the torus algebra $\cT$ to each fixed homology class, in such a way that the $\cA_\infty$ relations for each homology class are compatible with the algebra --- yielding an $\cA_\infty$ multi-module over the algebra $\cT^h$.

The main theorem we show is:
\begin{thm}\label{mainthmlite}The invariant $\poly(L, \Lambda)$ is quasi-isomorphic to the bordered-sutured multi-module $\bsa(S^3 - L, \Lambda)$.
\end{thm}

The drawback of the invariants $\bsa$ is that they are often rather tricky to compute (cf. eg.  \cite{LOT, Hanselman:graph, LOT:mappingclasses, Levine:doubles}). The module $\poly(L, \Lambda)$, whilst somewhat involved in its definition, is built from objects which have in one guise or other been present from the advent of Heegaard Floer homology for closed three-manifolds. There is a close relationship between the types of polygon maps used to build $\poly(L, \Lambda)$ and the invariants $\cflm(L)$ --- see, for instance,  \cite{Ozsvath-Szabo:2003, Ozsvath-Szabo:integer, Ozsvath-Szabo:rational}, \cite{Rasmussen:03},  \cite{Manolescu:linksurgeries}. In the sequel to this paper \cite{me:satellite}, we will utilise this to show that the invariants $\bsa(S^3 - L, \Lambda)$ are determined by the invariant $\cflm(L)$, generalising the result of \cite{LOT}, which shows that for a knot the bordered invariants $\cfd(S^3 - K, \Lambda)$ are determined by the knot invariant $\cfkm(K)$.

Conceptually, our approach is perhaps advantageous to that of \cite{LOT} in that --- besides generalising directly to the case of general link complements --- its connections with pre-existing gluing type results in closed Heegaard Floer homology are clear. In particular, it allows one to easily write down formulae for gluings in terms of mapping cones of polygonal maps between (filtered) Heegaard Floer complexes. We will use this in the sequel to write down a mapping cone formula for the knot Floer homology of a satellite knot which depends upon the knot Floer homology of the companion knot and the link Floer homology of the pattern link, and bypasses the need to understand the machinery of bordered Heegaard Floer homology.

\subsection{Motivation}
The invariant $\poly(L, \Lambda)$ is rather daunting at first, but has many precursors in Heegaard Floer homology that make it a natural object to study.

In \cite[Section 4.2]{Ozsvath-Szabo:BDC}, a chain complex is associated to a framed $h$-component link. This complex is closely related to an iteration of the \emph{link surgeries mapping cone} complex $X(L, \Lambda)$ considered in \cite{Manolescu:linksurgeries}. As a set, this complex is given by $X = \bigoplus_{I \in \{m,l\}^h}\cfh(Y(I))$, where $Y(I)$ is the manifold formed by surgering $Y$ along $L$ according to the framing induced by $\Lambda$ and $I$. The differential $D$ counts a sum of polygon maps that respect a fixed ordering of the vertices of $I^h$.

The invariant $\poly(L, \Lambda)$ is defined in a very similar manner, but with two main differences. Firstly, polygon maps in $\poly(L, \Lambda)$ do not obey the ordering of the vertices: instead, we must consider all possible polygon maps --- i.e. the maps in $\poly(L, \Lambda)$ split as sets of maps, where each set respects a fixed ordering of the vertices of $I^h$. The other clear difference is that $\poly(L, \Lambda)$ is an $\cA_\infty$-module, whereas $X(L, \Lambda)$ is a chain complex. This can be reconciled by considering instead the differential graded module $(P, \partial) := \poly(L, \Lambda) \otimes \cT^*$ (here, we write $\cT^*$ to denote the tensor algebra associated to the torus algebra: see \cite[p.10]{LOT}). Both $P$ and $X(L, \Lambda)$ split as a direct sum over idempotents of the torus algebra, and the provision of an extra base-point for each component of $L$ in the corresponding Heegaard multi-diagram gives a filtration upon the differential of $X(L, \Lambda)$, which corresponds with the splitting of $\partial$ according to corresponding sequences of algebra elements. When $L$ is a knot, this should be compared to \cite{eftekhary:splicing}.

More recently, the complex $X(L, \Lambda)$ has been studied from the perspective of bordered Heegaard Floer homology, to understand the difference between two spectral sequences between the reduced Khovanov homology of the mirror of a link $L'$ and the Heegaard Floer homology $\hfh(\Sigma(L'))$, where $\Sigma(L')$ is the branched double cover of $L'$: the one given in \cite{Ozsvath-Szabo:BDC}, and another defined in terms of bordered Heegaard Floer homology \cite{LOT:SS}. This relationship follows naturally from Theorem \ref{mainthmlite}, but, perhaps more interestingly, the methods of \cite{LOT:SSII} can be applied, together with a model computation, to give a different proof of Theorem \ref{mainthmlite} to the one in the present paper.

We sketch the main idea. The main result of \cite{LOT:SSII} is a pairing theorem for holomorphic polygons in \emph{bordered Heegaard multi-diagrams} (cf. Section \ref{borderedpolygonmodulispaces} below). Approximately, this states that if $\cD_1$ and $\cD_2$ are bordered Heegaard $k$ and $l$-diagrams respectively (i.e. $\cD_1$ has $\bbeta^1, \dots, \bbeta^k$ sets of attaching circles and one set $\balpha$ of attaching arcs and circles, and similar for $\cD_2$), a pair of holomorphic polygons $u_1$ and $u_2$ with appropriate boundary conditions (encoded by elements of the torus algebra) corresponds with a holomorphic polygon in a diagram $\cD$. Here, $\cD$ corresponds roughly with the diagram formed by gluing $\cD_1$ and $\cD_2$ together along a common boundary: but in order to make sense of attaching circles, it is often necessary to add a series of perturbed `copies' of some $\bbeta^i$ to $\cD_1$ or $\cD_2$, to promote $\cD_1$ to a bordered Heegaard $l$-diagram (if $k<l$, say).

Consider a cube of framings with vertices given by elements $I \in \{m, l\}^2$. To this we may associate a type $D$ structure given by the sum of the invariants $\cfd(\cD(I))$, where each $\cD(I)$ is a bordered Heegaard multi-diagram for the torus, with attaching circles specified by $I$. We may promote this type $D$ structure to a type $DA$-bimodule by assigning to each domain $B $ for $\cD(I)$ a sequence of elements of the torus algebra $\arrowrho(B)$, and defining $m_{1,k}(\bx, \arrowrho)$ to be a map counting holomorphic curves in homology class $B$ with $\arrowrho(B) = \arrowrho$, such that summing these $m_{1,k}(\bx, \arrowrho)$ over all sequences of $\arrowrho$ gives the structure map $\delta^1$ for the invariant $\cfd(\cD(I))$. By direct calculation, one can see that this $DA$-bimodule is equivalent to the identity $DA$-bimodule for the torus, $\cfda(\II)$. As such, we have that $\cfa(E(L)) \cong \cfa(E(L)) \bbox \cfda(\II)$, but, by the pairing theorem and the above discussion, we know that this is equal to the sum $\bigoplus_I \cfa(E(L)) \bbox \cfd(\cD(I))$, which is a sum of polygon maps in a bordered heegaard multi-diagram with $|L|-1$ boundary components, splitting by the algebra element associated to the restriction of their domain to the portion $\cD(I)$. Iterating this for all boundary components of $E(L)$ gives an identification of $\cfa(E(L))$ with a module defined in terms of polygon maps in closed Heegaard diagrams, with extra algebra information stemming from how the associated domain behaves in the regions corresponding with $\cD(I)$ --- this is precisely the way the object $\poly(L, \Lambda)$ is defined.

This should be compared with the results of \cite[Section 6.2]{auroux:fukaya}; in particular, the sketched argument above is approximately equivalent to the sketch proof of Proposition 6.5  given in that paper, but re-cast in a setting and language that is perhaps more accessible to the reader with background in Heegaard Floer homology, rather than more broad symplectic geometry.\footnote{For full transparency, I learned of Auroux's results roughly two-thirds of the way through writing this paper.} 

\subsection{Organisation}
In Section \ref{algebra} we recall the algebraic definitions of $\cA_\infty$ multi-modules necessary for the rest of the paper, and some relevant constructions and properties of the torus algebra $\cT$. Then, in Section \ref{diagrams}, we consider the close relationship between bordered-sutured Heegaard diagrams which represent the complement of a link, and a corresponding family of `partially splayed' and `splayed' Heegaard diagrams. The invariant $\poly(L, \Lambda)$ is defined in terms of the splayed diagrams, which are a family of closed Heegaard multi-diagrams --- but we will need to consider a series of auxiliary modules $\poly_k(L, \Lambda)$ defined in terms of partially splayed diagrams, which are bordered-sutured $k$-multi-diagrams. These correspond roughly with the final $k$ algebra inputs into a structure map $m_n$ being contributed not by asymptotics at a puncture in the Heegaard diagram, but instead by algebra elements associated to the homology class of a holomorphic curve. In particular, when $k$ is zero, all algebra elements are contributed by asymptotics at a puncture, and so these recover the bordered sutured Floer homology of the link complement; when $k$ is large enough, these modules stabilise and there is no interesting behaviour near the puncture --- so these modules may be identified with $\poly(L, \Lambda)$.

In Section \ref{modulispaces}, we define a series of moduli spaces of holomorphic polygons which will be instrumental in our definition of the invariants $\poly(L, \Lambda)$, and the auxiliary modules $\poly_k(L, \Lambda)$. We carefully analyse the possible boundary of these moduli spaces, thus ensuring that the $\cA_\infty$ relations hold for these multi-modules. We also lay the foundations of a neck-stretching argument which will allow us, in the final section, to identify all of the modules $\poly_k(L, \Lambda)$ for each $k$, thus proving Theorem \ref{mainthmlite}.

\subsection{Acknowledgements}
I owe the Heegaard Floer community at large a debt of graditude for their support and encouragement. In particular, I would like to thank Marco Marengon for many useful and interesting conversations. I am also grateful to Peter Ozsv\'{a}th for hosting me in Princeton, and the helpful conversations that took place there.

Robert Lipshitz has gone far and beyond in his encouragement and support, and in his tirelessness in answering my questions about the `LOT' oeuvre --- he and the University of Oregon were also very gracious hosts to me in the spring of 2016. Richard Thomas has been invaluable in his support, encouragement and cajoling over the past year or so.

Finally, I have been constantly humbled by the generosity and kindness of my supervisor Jacob Rasmussen. Without his reassurance, encouragement and patience, I doubt any of this would have been possible.

\section{Algebraic background}\label{algebra}
We spend this section recalling the algebraic background of the bordered sutured invariants --- although these are defined in \cite[Section 8]{Zarev}, the algebraic structure is not discussed at length.

\subsection{The torus algebra}
We first define the algebras that we will work over for most of this paper. These are the \emph{torus algebra}, $\cT$, and the products $\cT^k : = \underbrace{\cT \times \cdots \times \cT}_{k \text{\ times}}$.

The torus algebra is generated by eight elements as an $\FF_2$-module, which are the two idempotents $\iota_m$ and $\iota_l$, and the six Reeb chords $\rho_1, \rho_2, \rho_3, \rho_{12}, \rho_{23}$ and $\rho_{123}$. The differential is trivial, and the idempotents satisfy $\iota_m \iota_l = \iota_l \iota_m = 0$ --- we also write $1$ for the element $\iota_m + \iota_l$. The Reeb chords satisfy $\rho_1 \rho_2 = \rho_{12}, \rho_{2}\rho_3 = \rho_{23}$ and $\rho_1 \rho_{23} = \rho_{12} \rho_3 = \rho_{123}$. The other non-zero products are:
$$\begin{array}{c c c}\iota_m \rho_1 = \rho_1 \iota_l = \rho_1, & \iota_l \rho_2 = \rho_2 \iota_m = \rho_2, & \iota_m \rho_3 = \rho_3 \iota_l = \rho_3,\\
\iota_m \rho_{12} = \rho_{12} \iota_m = \rho_{12}, & \iota_l \rho_{23} = \rho_{23}\iota_l = \rho_{23}, & \iota_m \rho_{123} = \rho_{123}\iota_l = \rho_{123}\end{array}.$$

These algebra elements may be interpreted in terms of \emph{Reeb chords}. Let $Z$ denote an oriented circle with four marked points $\ba = \{a_1, \dots, a_4\} \subset Z$, together with a basepoint $b \in Z \setminus \ba$ --- we suppose that $\ba$ is numbered in such a way that a positive path from $b$ encounters them in order. A \emph{Reeb chord} $\rho \in (Z \setminus z, \ba)$ is an immersed arc in $Z \setminus z$ with endpoints in $\ba$, and orientation agreeing with that on $Z$. To each Reeb chord $\rho$ we associate the algebra element $a(\rho) : = \rho_{[i,j]}$, where $\partial \rho = a_j - a_i$ and $[i,j]$ is the unique substring of $123$ beginning with $i$ and ending with $j$. We will often conflate a Reeb chord with its corresponding algebra element.

We say that two Reeb chords $\rho$ and $\rho'$ are \emph{composable} if the product $\rho \rho'$, which we call their \emph{composition}, is non-zero. More generally, we call sequence $(\rho^1, \dots, \rho^k)$ of Reeb chords \emph{composable} if for every $i = 1, \dots, k-1$, the chords $\rho^i$ and $\rho^{i+1}$ are composable. If $\arrowrho$ is a sequence of Reeb chords, then there is an associated sequence of Reeb chords $\arrowrho'$ formed by replacing any composable subsequences of Reeb chords with their composition, which we call the \emph{reduction} of $\arrowrho$. There is an equivalence relation upon sequences of Reeb chords, given by saying that $\arrowrho$ and $\arrowrho'$ are \emph{composable-equivalent} if and only if their reductions agree.

In a similar manner, we say that two chords $\rho, \rho'$ are \emph{$123$-composable} if they are composable and their composition is $\rho_{123}$, and make a similar definition of the \emph{$123$-reduction} of a sequence of Reeb chords. We shall say that a pair of sequences of Reeb chords are \emph{$123$-equivalent} if and only if their $123$-reductions agree.

\begin{rmk}Being $123$-composable depends only upon \emph{pairs} of adjacent chords. In particular, the sequence $(\rho_1, \rho_2, \rho_3)$ is $123$-reduced, whereas the sequence $\rho_1, \rho_{23}$ has $123$-reduction $\rho_{123}$.
\end{rmk}

One can see from the definition of the torus algebra that some Reeb chords behave rather differently to others with respect to products by the elementary idempotents. If $\iota_\delta \cdot \rho \ne 0$ we will say that $\iota_{\start}(\rho) = \iota_\delta$, and if $\rho \cdot \iota_\delta \ne 0$ we will say that $\iota_{\fin}(\rho) = \iota_\delta$.

\begin{defn}\label{defjump} Suppose that $\arrowrho = (\rho^1, \dots, \rho^m)$ is a sequence of elements of the torus algebra. To the sequence $\arrowrho$ we associate the sequence of its \emph{jumping chords}. That is, the subsequence $$\jump(\arrowrho) := (\rho^i \in \arrowrho: \iota_{\start}(\rho^i) \ne  \iota_{\fin}(\rho^i))$$
of all chords in the sequence that have different starting to ending idempotent. 
\end{defn}

We turn our attention to what we term a \emph{splicing} of a series of Reeb chords. These should be thought of as something between a set of elements $\vec{b}^1, \dots, \vec{b}^k$ where $\vec{b}^i \in \cT^{\otimes h}$, and a set of sequences of elements $\vec{a}_1, \dots, \vec{a}_h$ where each $\vec{a}_j \in \cT^{k_j}$.

\begin{defn}\label{splicing}Suppose that $\arrowrho_1, \dots, \arrowrho_h$ are a series of sequences of Reeb chords $\arrowrho_j : = (\rho^1_j, \dots, \rho^{i_j}_j)$ with corresponding jumping sequences $\jump(\arrowrho_1), \dots, \jump(\arrowrho_j)$.

An $m$-\emph{splicing} of $\arrowrho_1, \dots, \arrowrho_h$ is an ordered length $m$ partition $P_j : = (P^1_j, \dots, P^{m}_j)$ of each $\jump(\arrowrho_j)$, allowing for parts $P^i_j$ to be empty, and such that \begin{itemize}\item$|P^i_j| \le 1$ for every $i,j$ and \item for each $i$, the union $P^i_1 \cup \cdots \cup P^i_h$ is nonempty.\end{itemize}

Every splicing $\vec\sigma$ has an underlying set of unordered partitions, which we shall often denote by $\sigma$. We often write $|\vec\sigma| = m$ to indicate that $\vec\sigma$ is an $m$-splicing. We also write $\col(\vec\sigma)$ for the number of elements in the set
$$\{i : P^i_j \text{\ is nonempty for more than one\ } j\}.$$

We shall say that $\vec\sigma$ is an \emph{interleaving} of $\arrowrho_1, \dots, \arrowrho_h$, or that $\vec\sigma$ is \emph{interleaved} if $\col(\vec\sigma) = 0$: that is, for every $i$, precisely one $P^i_j$ is non-empty.
\end{defn}

If we are given two sets of sequences of Reeb chords $\arrowlambda = \arrowlambda_1, \dots, \arrowlambda_h$ and $\arrowdelta = \arrowdelta_1, \dots, \arrowdelta_h$, where each $\arrowlambda_j$ is of length $l_j$ and each $\arrowdelta_j$ is of length $k_j$, there is a corresponding set of sequences of Reeb chords $\arrowlambda \star \arrowdelta$ defined by $(\arrowlambda \star \arrowdelta)_j : = (\arrowlambda_j, \arrowdelta_j)$. Similarly, given splicings $\vec\sigma_1 = P_1, \dots, P_h$ and $\vec\sigma_2 = Q_1 \dots Q_h$ for $\arrowrho$ and $\arrowdelta$ respectively, then there is a corresponding splicing $\vec\sigma_1 \star \vec\sigma_2$ of $\arrowrho \star \arrowdelta$, formed by $(\vec\sigma_1 \star \vec\sigma_2) = (P_1, Q_1), \dots, (P_h, Q_h)$.

A splicing naturally gives rise to a sequence of sets of idempotents in the torus algebra:
\begin{defn}\label{splicingtoidemps}Let  $\arrowrho : = \arrowrho_1, \dots, \arrowrho_h$ be a series of sequences of algebra elements $\arrowrho_j$, and let $\vec\sigma = P_1, \dots, P_h$ be an $m$-splicing for $\vec{a}$.

For each $j$, let $P_j^{i_1}, \dots, P_j^{i_{|\jump(\arrowrho_j)|}}$ denote the parts of $P_j$ which are nonempty, and let $\rho_j^{i_k}$ denote the (unique) Reeb chord in each $P_j^{i_k}$. Each $P_j^{i_k}$ has an associated sequence of parts $$\arrowP_j^{i_k} : =(P_j^{i_{k}},P_j^{i_{k}+1} \dots, P_j^{i_{k+1}-1}),$$ its \emph{trailed} sequence, where only  $P_j^{i_k} \in \arrowP_j^{i_k}$ is non-empty. The sequence $P_j$ decomposes as a union 
$$P_j = \bigcup_{k = 0}^{|\jump(\arrowrho_j)|} \arrowP^{i_k},$$
where we define $\arrowP_j^{i_{0}}$ to be the union of the empty partitions $(P^1_j, \dots, P^{i_1 - 1}_j)$.

We define a set of idempotents by
$$\iota^i_j : = \left\{\begin{array}{c l}\iota_\start(\rho^{i_1}_j) &\text{ if } P^i_j \in \arrowP_j^{i_0},\\
\iota_{\fin}(\rho^{i_k}_j) & \text{ if } P^i_j \in \arrowP_j^{i_k}, k \ne 0 \end{array}\right. .$$

We collect these into sequences of idempotents $\arrowiota^0, \dots, \arrowiota^{m}$ by $\arrowiota^i = \iota^i_1, \dots, \iota^i_h$, and denote by $\iota_{\vec\sigma} : = (\arrowiota^1, \dots, \arrowiota^{m+1})$.
\end{defn}
We can also go in the other direction, associating a splicing to a pair of a set of sequences of Reeb chords and a set of sequences of idempotents:

\begin{defn}\label{reebidemptosplicing}
Suppose that we are provided with a set of sequences of Reeb chords $\arrowrho_1, \dots, \arrowrho_h$, and a sequence of sets of idempotents $\arrowiota = (\arrowiota^0, \dots, \arrowiota^{m})$ where each $\arrowiota^i = \iota^i_1, \dots, \iota^i_h$ satisfies that for $j = 1, \dots, h$, the sequence of idempotents $\arrowiota_j : = \iota^0_j, \dots, \iota^{m}_j$ has
$$|\{\iota^i_j \in \arrowiota_j: \iota^i_j \ne \iota^{i-1}_j\}| = |\jump(\arrowrho_j)|.$$
There is a corresponding splicing given by setting $P^i_j = \emptyset$ when $\iota^i_j = \iota^{i-1}_j$, and for each $\iota^{i_k}_j \ne \iota^{i_k-1}_j$ we set $P^i_j$ to be the part $\{\rho_j^{i_k}\}$, where $\rho_j^{i_k}$ is the $k$th chord in $\jump(\arrowrho_j)$ --- we call this $\vec\sigma(\arrowrho, \arrowiota)$.
\end{defn}
We will also need to consider splitting a sequence of Reeb chords in two --- roughly, in what follows, one half $\arrowrho_\boat$ corresponds with genuine asymptotic behaviour of a holomorphic curve, whereas $\arrowrho_\anchor$ will be extracted from the topology of the homology class of the curve.

\begin{defn}Suppose that $\arrowrho = \arrowrho_1, \dots, \arrowrho_h$ is a set of sequences of Reeb chords, and $\vec\sigma$ is an interleaved splicing of $\arrowrho$. We shall say that a \emph{$k$-shipping} of $\arrowrho$ is a pair $\arrowrho_\boat, \arrowrho_\anchor$ (where $\arrowrho_\boat = \arrowrho_{\boat, 1}, \dots, \arrowrho_{\boat, h}$ and $\arrowrho_\anchor = \arrowrho_{\anchor, 1}, \dots, \arrowrho_{\anchor, h}$) together with a pair of splicings $\vec\sigma_\boat, \vec\sigma_\anchor$ such that:
\begin{itemize}\item for some $j$, we have that $\arrowrho_{\anchor, j}$ begins with a jumping chord,
\item $\vec\sigma_\anchor$ is a $k'$-splicing of $\arrowrho_\anchor$, where
$$k' = \max\{|\jump(\arrowrho_{\anchor,1})| \times \cdots \times |\jump(\arrowrho_{\anchor, h})|, k\},$$
and
\item $(\arrowrho_\boat, \vec\sigma_\boat) \star (\arrowrho_\anchor, \vec\sigma_\anchor)=(\arrowrho, \vec\sigma)$.
\end{itemize}

If $(\arrowrho_\boat, \vec\sigma_\boat)$ and $(\arrowrho_\anchor, \vec\sigma_\anchor)$ form a $k$-shipping of $(\arrowrho, \vec\sigma)$, and $(\arrowrho'_\boat, \vec\sigma'_\boat)$ and $(\arrowrho'_\anchor, \arrowrho'_\anchor)$ form a $(k+1)$-shipping of $(\arrowrho, \vec\sigma)$ then we shall say that the pair $(\arrowrho'_\boat, \vec\sigma'_\boat)$ and $(\arrowrho'_\anchor, \arrowrho'_\anchor)$ \emph{splays} $(\arrowrho_\boat, \vec\sigma_\boat)$ and $(\arrowrho_\anchor, \vec\sigma_\anchor)$.
\end{defn}

\subsection{$\cA_\infty$ multi-modules}
We begin with $\cA_\infty$-modules as a warm-up. Let $\cA$ be a differential graded algebra over $\FF_2$, with ring of idempotents $\cI \subset \cA$. A (right) $\cA_\infty$ module over $\cA$ is a module over $\FF_2$ with right action of $\cI$ and maps
$$m_{k+1}: M \otimes_{\cI} \cA \otimes_{\cI} \otimes \cdots \otimes_\cI \cA \rightarrow M,$$
satisfying the $\cA_\infty$ relation
\begin{align*}0& = \sum_{i = 0}^n m_{n-i+1}(m_{i+1}(x, a_1, \dots, a_i) a_{i+1}, \dots, a_n)\\ &+ \sum_{i=1}^{n-1}m_n(x, a_1, \dots, \mu(a_i, a_{i+1}), a_{i+2}, \dots, a_n)\\ &+ \sum_{i=1}^n m_{n+1}(x, a_1, \dots, \partial a_i, \dots, a_n)\end{align*}
for every $x \in M$ and $a_1, \dots, a_n \in \cA$.

More general is the notion of an \emph{$\cA_\infty$ multi-module}. If $\cA_1, \dots, \cA_h$ are differential graded algebras over $\FF_2$, each with corresponding ring of idempotents $\cI_1, \dots, \cI_h$, then a (right) $\cA_\infty$ multi-module over $\cA_1, \dots, \cA_h$  is a module $M$ over $\FF_2$ with a right action of $\cI: = \cI_1 \otimes \cdots \otimes  \cI_h$, and equipped with maps
$$m_{i_1, \dots, i_h}: M \otimes_{\cI} \cA_1^{\otimes i_1} \otimes \cdots \otimes  \cA_h^{\otimes i_h} \rightarrow M,$$
where $\cA_i^{\otimes i_j}$ is the result of taking the tensor product of $\cA_i$ with itself $i_j$ times, over the ring $\cI_j$.

The relations these maps must satisfy are more difficult to write down; we roughly follow Hanselman's presentation in \cite[Section 2.1]{Hanselman:graph}, with some modifications. If $\vec{a_j} := (a_j^1, \dots, a_j^{i_j}) \in \cA_j^{\otimes i_j}$ is a sequence of elements of $\cA_j$, then for every $0 \le k \le i_j$ we define the truncations
\begin{align*}T_k(\vec{a_j}) &:= (a_j^1, \dots, a_j^{k}) \\
T^k(\vec{a_j}) &:= (a_j^k, \dots, a_j^{i_j}),
\end{align*}
and the maps
\begin{align*}\bar{\mu}^k(\vec{a_j}) &: = (a_j^1, \dots, a_j^{k-1},  \mu(a_j^k, a_j^{k+1}), a_j^{k+2}, \dots a_j^{i_j})\\
\bar{\partial}(\vec{a_j}) &: = \sum_{k=1}^{j_h - 1}(a_j^1, \dots, a_j^{k-1},  \partial a_j^k, a_j^{k+1}, \dots a_j^{i_j}).
\end{align*}
This allows us to define the sequence
$$\bar{\mu}_j^k(\vec{a_1}, \dots, \vec{a_h}) : = (\vec{a_1}, \dots, \bar{\mu}^k(\vec{a_j}), \dots, \vec{a_h}).$$

With this in mind, and omitting subscripts, the relations which the multiplication maps must satisfy is given by
\begin{align*}0 & = \sum_{j_1, \dots, j_h} m(m(x, T_{j_1}(\vec{a_1}), \dots, T_{j_h}(\vec{a_h})), T^{j_1}(\vec{a_1}), \dots, T^{j_h}(\vec{a_h}))\\
&+ \sum_{j=1}^h\sum_{k = 1}^{i_j -1} m(x, \bar{\mu}^k_j(\vec{a_1}, \dots, \vec{a_h})) \\
&+ \sum_{k=1}^h m(x, \vec{a_1}, \dots, \vec{a_{k-1}}, \bar{\partial}(\vec{a_k}), \vec{a_{k+1}}, \dots, \vec{a_h}).
\end{align*}

When $\cA_1, \dots, \cA_h$ are all copies of the torus algebra, we will have cause to consider decompositions of these maps according to the extra data of a splicing of $\vec a_1, \dots, \vec a_h$ --- the easiest way to define the $\cA_\infty$ structure maps we will be interested in later in the paper is in terms of a series of auxiliary maps which sum to give the structure maps, and themselves satisfy a relation that implies the $\cA_\infty$ relations for their sums. First, we need to discuss operations upon interleaved splicings which are compatible with the operations $T_*, T^*$ and $\bar{\mu}^i_j$.

Fix an interleaved splicing $\vec\sigma$ of a set of sequences of Reeb chords $\arrowrho = \arrowrho_1, \dots, \arrowrho_h$. We shall say a pair of indices $i, j$ is \emph{compatible} with $(\arrowrho, \vec\sigma)$ if either:
\begin{enumerate}\item $\rho^i_j \rho^{i+1}_j = 0$, or
\item $\rho^i_j \rho^{i+1}_j$ is nonzero and precisely one of $\rho^i_j, \rho^{i+1}_j$ is jumping (in part $P^{k}_j$, say), or
\item $\rho^i_j \rho^{i+1}_j$ is nonzero, both of $\rho^i_j \in P^{k_i}_j$ and $\rho^{i+1}_j \in P^{k_{i+1}}_j$ are jumping, and we have that they are in consecutive parts --- that is, $\rho^i_j \in P^{k_i}_j, \rho^{i+1}_j \in P^{k_i +1}_j$ for some $k_i$.
\end{enumerate}
If $i, j$ are compatible with $\arrowrho, \vec\sigma$, then there is an associated splicing of $\bar{\mu}^i_j(\arrowrho)$ which we call $\vec\sigma^i_j$: in case (2), we set $\vec\sigma^i_j$ to be the splicing of $\bar{\mu}^i_j(\arrowrho)$ obtained from $\vec\sigma$ by replacing $P^k_j$ with the part $\{\rho^i_j \rho^{i+1}_j\}$. In case (3), we set $\vec\sigma^i_j$ to be the splicing given by omitting, for each $j$, the parts $P^{k_i}_j$ and $P^{k_i+1}_j$ from $P_j$.

If $\vec\sigma$ is interleaved, we shall say that an index $0<k<m$ is \emph{collidable} for $(\arrowrho, \vec\sigma)$ if for every $j$ we have that only one of $P^k_j$ and $P^{k+1}_j$ is non-empty. In this circumstance, we can define an associated $m-1$-splicing of $\arrowrho$ associated to $\vec\sigma$ and $k$ which we call $\vec\sigma(k)$, the \emph{collision of $\vec\sigma$ at $k$}. $\vec\sigma(k)$ consists of the partitions $Q_j$, where each $Q_j$ is formed by
$$Q_j^i := \left\{ \begin{array}{c c}P_j^i  & \mbox{ if } i < k\\
P_j^i \cup P_j^{i+1} & \mbox{ if } i = k\\
P_j^{i+1}  & \mbox{ if } i > k + 1
\end{array}\right. .$$

In words, we amalgamate the $k$ and $k+1$-th parts of each partition. Colliding does not preserve being interleaved.

\begin{defn}Let  $\vec\sigma(i_1, \dots, i_h)$ denote the set of all splicings of all sets of sequences of Reeb chords in  $\cA_1^{\otimes i_1} \otimes \cdots \otimes  \cA_h^{\otimes i_h}$.
We shall say a sequence of maps $$n_{i_1, \dots, i_h}: M \otimes_{\cI} \cA_1^{\otimes i_1} \otimes \cdots \otimes  \cA_h^{\otimes i_h} \times \vec\sigma(i_1, \dots, i_h) \rightarrow M$$
satisfies a \emph{partial $\cA_\infty$ relation} if, for every $x \in M$, set of sequences of Reeb chords $\arrowrho$, and interleaving $\vec\sigma$ of $\arrowrho$, we have that

\begin{align*}\label{partialainf} 0 & = \sum_{\substack{\arrowlambda \star \arrowdelta = \arrowrho \\ \vec\sigma_1 \star \vec\sigma_2 = \vec\sigma}} n(n(x, \arrowlambda, \vec\sigma_1), \arrowdelta, \vec\sigma_2 )\\
&+ \sum_{i, j \text{\ compatible}} n(x,\bar{\mu}^i_j(\arrowrho), \vec\sigma^i_j) \\
&+ \sum_{k \text{\ collidable}}^{|\vec\sigma|-1} n(x, \arrowrho, \vec\sigma(k)).
\end{align*}
\end{defn}

\begin{prop}Let $n_{i_1, \dots, i_h}$ be a sequence of maps satisfying the partial $\cA_\infty$ relation. Then the maps
 $$m_{i_1, \dots, i_h}: M \otimes_{\cI} \cA_1^{\otimes i_1} \otimes \cdots \otimes  \cA_h^{\otimes i_h} \rightarrow M$$
 defined by
 $$m_{i_1, \dots, i_h}(x, \arrowrho) : = \sum_{\vec\sigma \text{\ an interleaving for\ }\arrowrho} n_{i_1, \dots, i_h}(x, \arrowrho, \vec\sigma)$$
 satisfy the $\cA_\infty$ relations.
\end{prop}
\begin{proof}This follows readily by fixing $x$ and $\arrowrho$, and summing the partial $\cA_\infty$ relations over all interleavings for $\arrowrho$. The first two terms in  the partial relation sums to the first two terms in the $\cA_\infty$ relations, and the final terms cancel in pairs where $\vec\sigma$ and $\vec\sigma'$ are the two unique interleavings which satisfy $\vec\sigma(k) = \vec\sigma'(k)$.
\end{proof}

\section{Diagrams}\label{diagrams}
In this section, we will be concerned with \emph{diagrams} of various sorts. In the most generality:

\begin{defn}A \emph{diagram} consists of the following data:
\begin{enumerate}\item A connected surface $\Sigma$ of genus $g$, with $h$ boundary components $\partial \Sigma_1, \dots, \partial \Sigma_h$.
\item An ordered tuple of \emph{attaching curves} $\bgamma = \{\bgamma^1, \dots, \bgamma^m\}$, where each $\bgamma^i$ is a set of $(g + h -1) - k_i$   pairwise disjoint embedded curves and $2 k_i$ pairwise disjoint embedded arcs in $\Sigma$. We require that for every $i$ and $j$, $\bgamma^i$ intersects $\bgamma^j$ transversely.
\item A set of marked points $\bb = \{b_1, \dots, b_h\}$, where $b_i \in \partial \Sigma_i$.
\item A set of distinguished points $\bp = \{ p_1, \dots, p_l\}$, where $p_i \in \inte(\Sigma)$.
\end{enumerate}
\end{defn}

\begin{defn}Let $\cD$ be a diagram, with two sets of attaching curves and arcs $\bgamma, \bgamma'$. A \emph{generator} for $(\bgamma, \bgamma')$ is a set of intersection points $\bx = \{x_1, \dots, x_g\}$ where the $x_i$ are members of $\bgamma \cap \bgamma'$, such that:
\begin{itemize}\item every closed curve in $\bgamma$ and $\bgamma'$ contains precisely one $x_i$, and
\item exactly half of the arcs in each of $\bgamma$ and $\bgamma'$ contain precisely one $x_i$.
\end{itemize}

We will denote the set of generators for $(\bgamma, \bgamma')$ by $\cG(\bgamma, \bgamma')$.

In more generality, we will say a \emph{sequence of generators} for a diagram with more than two sets of attaching curves $\bgamma^1, \dots, \bgamma^m$ is a sequence $\bx^{1,2}, \dots, \bx^{m,1}$ where $\bx^{i,i+1} \in \cG(\bgamma^i, \bgamma^{i+1})$ (here, the superscripts are taken modulo $m$).
\end{defn}

If $\cD$ is a diagram with $m$ sets of attaching curves, then for any subset $S \subset \{1, \dots, m\}$ of size $l$, there is an associated set of attaching circles $\bgamma(S) : = \{\bgamma^i: i \in S\}$, and the data $\cD|_{S} := (\Sigma; \bgamma(S); \bp; \bb)$ also constitutes a diagram, which we call an \emph{$l$-subdiagram of $\cD$}.

We will later be interested in perturbing some of the curves in a diagram. A way to do this is given in \cite[Lemma 11.8]{Lipshitz:cylindrical}. Namely, Lipshitz describes a Hamiltonian $H: \Sigma \rightarrow \RR$ which is supported within a tubular neighbourhood of a set of attaching curves, such that the corresponding Hamiltonian isotopy $H(x,t)$ behaves as in Figure \ref{fig:approx} ($\theta$ refers to the $S^1$ co-ordinate of the tubular neighbourhood, identified with $S^1 \times [0,1]$). 
\begin{figure}[h]
\centering{
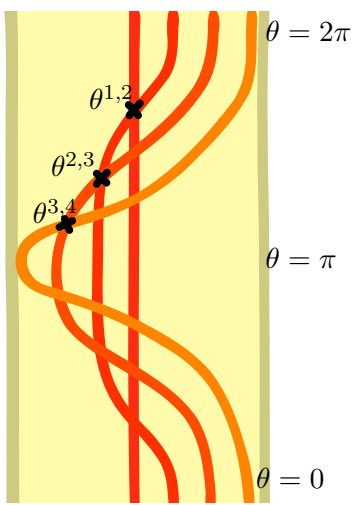
\caption{The behaviour of the hamiltonian isotopy $H(x,t)$, where $t = 0$ is shown in red and $t = \varepsilon$ in orange.}
\label{fig:approx}\phantomsection
}
\end{figure}

\begin{defn}Suppose $\bgamma$ is a collection of arcs and curves in some diagram. Then we define the $\varepsilon$-small $m$-fold approximation of $\bgamma$ to be the set of sets of arcs and curves $\bgamma(\varepsilon, m) = \{\bgamma^1(\varepsilon, m), \dots, \bgamma^j(\varepsilon,m)\}$, where $\bgamma^i(\varepsilon,m) = H(\bgamma,i \varepsilon / m)$. We shall sometimes write $\bgamma^{\varepsilon/m; i}$ for $\bgamma^i(\varepsilon,m)$ when it is more convenient.

If $\bgamma^1, \dots, \bgamma^m$ is a sequence of curves that may be written as $\bgamma(\varepsilon, m)$ for some $\bgamma, \varepsilon, m$, we will often say that $\bgamma^1, \dots, \bgamma^m$ are \emph{small approximations of one another}, and, when $m=2$, we shall write $\bgamma^1 \approx \bgamma^2$. If $\bgamma^1, \bgamma^2 \subset \Sigma$ are small approximations of one another for some $\varepsilon$, we will say that the subset of $\Sigma$ given by $\cup_{t \in [0, \varepsilon]}H(\bgamma^1, t)$ is the \emph{approximation region} for $\bgamma^1, \bgamma^2$. More generally, if $\bgamma^1, \dots, \bgamma^m$ are all small approximations of one another, we say the approximation region for $\bgamma^1, \dots, \bgamma^m$ is the union of the approximation regions for each $\bgamma^i \approx \bgamma^{i+1}$.
\end{defn}

Note that in an $\varepsilon$-small $m$-fold approximation, any consecutive set of arcs and curves $\bgamma^i, \bgamma^{i+1}$ constitute the $\varepsilon/m$-small 2-fold approximation of $\bgamma$. For any $2$-fold approximation $\bgamma^i, \bgamma^{i+1}$, we can choose a \emph{distinguished generator $\btheta^{i,i+1}$} : a choice of, for every $i = 1, \dots, g+k$, one of the two elements of $\bgamma \cap \bgamma'$. Hence, given an $\varepsilon$-small $m$-fold approximation, we may choose a distinguished generator for each consecutive pair $\bgamma^i, \bgamma^{i+1}$ to give a set of distinguished generators $\vec\btheta = \{\btheta^{1,2}, \dots, \btheta^{m-1,m}\}$.

We will be largely concerned with a canonically defined set of distinguished generators. Namely:
\begin{exmp}\label{thetagen}Consider the intersection points $\theta^{i, i+1}_{+,j} \in \gamma_j^i \cap \gamma_{j}^{i+1}$ as shown in Figure \ref{fig:approx}. The collection of $\btheta_+^{i, i+1} = \cup_j \theta^{i, i+1}_{+,j} $ for every $i$ gives a set of distinguished generators $\vec\btheta_+$.

Sometimes, we refer to the other choice of intersection point in $\gamma^i_j \cap \gamma^{i+1}_j$ as $\theta^{i, i+1}_{-, j}$.
\end{exmp}

\begin{defn}Suppose $\bgamma$ and $\bgamma'$ are two collections of arcs and curves in some diagram, and let $\bx \in \cG(\bgamma, \bgamma')$ be a generator.

If $\bgamma^{\varepsilon/m; 1}, \dots, \bgamma^{\varepsilon/m; m}$ is the $\varepsilon$-small $m$-fold approximation to $\bgamma$, then, provided $\varepsilon$ is sufficiently small, there are a series of canonical \emph{nearest point maps}, which are given by, for each $i$, the map $\varphi_{i}: \cG(\bgamma, \bgamma') \rightarrow \cG(\bgamma^{\varepsilon/m; i}, \bgamma')$ which sends  $x_j \in \gamma_j \cap \gamma'_{\sigma(j)}$ to the canonical nearest point $x_j^{\varepsilon/m;i} \in \gamma^{\varepsilon/m; i}_j \cap \gamma'_{\sigma(j)}$.

If we are given a set of distinguished generators $\vec\btheta$ for the approximation, then this allows us to associate to any pair of generators $\bx, \by \in \cG(\bgamma, \bgamma')$ a sequence of generators $$(\bx', \btheta^{1,2}, \dots, \btheta^{m-1,m}, \by')$$
in the set
$$\cG(\bgamma', \bgamma^{\varepsilon/m; 1}) \times \cG(\bgamma^{\varepsilon/m; 1}, \bgamma^{\varepsilon/m; 2}) \times \dots \times \cG(\bgamma^{\varepsilon/m; m-1}, \bgamma^{\varepsilon/m; m}) \times \cG( \bgamma^{\varepsilon/m; m}, \bgamma'),$$
where $\bx' = \bx$, and $\by' = \varphi_m(\by)$. We call this sequence the \emph{$\varepsilon$-small $m$-fold splaying} of $(\bx, \by)$, and shall often abuse notation and write $(\bx, \btheta^{1,2}, \dots, \btheta^{m-1m}, \by)$ for the relevant generator.
\end{defn}

\subsubsection{Toroidal bordered Heegaard diagrams}
One of the main instances of a diagram we shall consider is when $\cD$ specifies a three-manifold with boundary that is a disjoint union of tori. The definition we shall use is:
\begin{defn}\label{bordereddiagrams}Suppose that $\cD = (\Sigma; \bgamma; \bb; \bp)$ is a diagram. Then we shall say that $\cD$ is a \emph{toroidal bordered Heegaard diagram} if:
\begin{enumerate}\item $\Sigma$ is a genus $g$ surface with $h$ boundary components.
\item $\bgamma$ consists of only two sets of curves, $\balpha$ and $\bbeta$, such that $\bbeta$ consists  of simple closed curves, and $\balpha$ can be partitioned into two sets,  $\balpha = \balpha^a \cup \balpha^c$, where $\balpha^a$ consists  of $2h$ arcs and $\balpha^c$ consists of $g-h$ closed curves.

\item The set $\balpha^a$ can be written $\balpha^a = \{\alpha^m_1, \alpha^l_1, \dots, \alpha^m_{h}, \alpha^l_{h}\}$, where $\alpha^m_i \cup \alpha^l_i$ intersects $\partial \Sigma_i$ relative to the point $b_i \in \partial \Sigma_i$ as shown in the figure,  and is disjoint from $\partial \Sigma_j$ for $i \ne j$.

\item Every region of the surfaces $\Sigma - \balpha$ and $\Sigma - \bbeta$ meets some base-point $b_i \in \bb$ (we sometimes call this \emph{homological linear independence}). 

\item  The set of points $\bp$ is empty (we therefore usually omit this from the notation).

\end{enumerate}
\end{defn}

In the case where $\cD$ is a toroidal bordered Heegaard diagram with $|\partial \Sigma| = 1$, this constitutes a bordered Heegaard diagram compatible with a manifold with torus boundary in the sense of \cite{LOT}. When $|\partial \Sigma| > 1$, the data corresponds with a \emph{bordered sutured Heegaard diagram} in the sense of \cite{Zarev}. The data of $\bb = \{b_1, \dots, b_{k} \}$ corresponds with a series of arcs $\Gamma_\bb \subset \partial \Sigma$, so that $\balpha^a$ and $\Gamma_\bb$ specify an embedding of the disjoint union of $g-k$ arc diagrams for the solid torus into $\Sigma$: that is, $\cD$ is equivalent to a bordered sutured Heegaard diagram for a three-manifold with $k$ torus boundary components. We will tend to ignore the distinction between $\cD$ and the corresponding bordered sutured diagram.

\begin{convention}To avoid proliferation of adjectives, we will often refer to a toroidal bordered Heegaard diagram as simply a \emph{bordered diagram}.
\end{convention}

Note that when $\cD$ is a bordered diagram, a generator for $\cD$ is a set of intersection points $x_1, \dots, x_g$, where $x_1, \dots, x_{h} \in \balpha^a \cap \bbeta$ and $x_{h+1}, \dots, x_g \in \balpha^c \cap \bbeta$. Without loss of generality, for $i = 1, \dots, h$, we have that $x_i \in \alpha_i^{\delta_i} \cap \beta_{\sigma(i)}$, where $\delta_i \in \{m, l\}$ and $\sigma$ is some permutation of $\{1, \dots, g\}$. If $\delta = (\delta_1, \dots, \delta_h)$ is a sequence of elements of $\{m, l\}$, then we can associate the idempotent $\iota(\delta) = (\iota_{\delta_1}, \dots, \iota_{\delta_h})$ in $\cT^h$. Accordingly, we may associate to any generator $\bx$ the corresponding idempotent $\iota(\bx) := \iota(\delta)$ --- this is called $I_A(\bx)$ in \cite{LOT}.

It will be of use to us later to note that although the set $\bp$ is empty if $\cD$ is a bordered diagram, there is a set of points $\bp(\bb)\subset \inte(\Sigma)$ given by pushing each $b_i\in \bb$  slightly off of the boundary $\partial \Sigma_i$ into the region of $\Sigma - \balpha -\bbeta$ adjacent to $b_i$.

\subsubsection{Heegaard multi-diagrams}

Another important instance of a diagram is a \emph{multi-pointed Heegaard multi-diagram}. We follow approximately the treatment given in \cite{LOT:SSII}.

\begin{defn}Let $\cD = (\Sigma; \bgamma; \bb; \bp)$ be a diagram. Then we say $\cD$ is a \emph{multipointed Heegaard $m$-diagram} if:
\begin{enumerate}\item $\Sigma$ is a closed surface of genus $g$.
\item $\bgamma$ consists of $m$ sets of curves $\bgamma^1, \dots, \bgamma^m$, where each $\bgamma^i$ is a set of $g+h-1$ simple closed curves (i.e. none of them are arcs) which span a $g$-dimensional sublattice of $H_1(\Sigma)$. Equivalently, for each $i$, the surface $\Sigma - \bgamma^i$ has  precisely $h$ components $A^i_1, \dots, A^i_h$.
\item The set $\bb$ is empty (this is implied by $\Sigma$ being closed, and we will often omit it from the notation).
\item The set $\bp$ can be ordered as $\bp = (w_1, z_1, \dots, w_h, z_h)$, such that $w_j, z_j \in \cap_{i=1}^h A^i_j$. We will often write $\bp = \bw \cup \bz$ for the partition of $\bp$ into $(w_1, \dots, w_h)$ and $(z_1, \dots, z_h)$ respectively.
\end{enumerate}

In the case where $|\bgamma| = 2$, we will say that $\cD$ is a \emph{basic Heegaard diagram}.
\end{defn}

\subsubsection{Bordered multi-diagrams}
A common generalisation of the two types of diagram already considered is a \emph{toroidal bordered multi-diagram}. This is a diagram $$\cD = (\Sigma; \bgamma^0, \bgamma^1, \dots, \bgamma^{m-1}; \bb; \bp)$$ such that $(\Sigma, \bgamma^0, \bgamma^1, \bb)$ is a toroidal bordered Heegaard diagram, where $\bgamma^0$ are closed curves; $\bgamma^1 = \bgamma^{1,a} \cup \bgamma^{1,c}$ splits as a series of arcs and closed curves; and for $i \ge 2$ each of the sets of curves $\bgamma^i$ is a set of $g+h-1$ simple closed curves.

Given a toroidal bordered multi-diagram, there are two associated surfaces which we will have cause to mention later. The first is the associated closed surface $\bar{\Sigma}$, which comes with associated arcs $\bar{\bgamma}^1$ --- we may view $\Sigma$ as $\bar{\Sigma} - \partial\bar{\Sigma}$. The second is the surface $\Sigma_{\bar{e}}$ which is the result of collapsing each boundary circle $\partial \bar{\Sigma}_i$ of $\bar{\Sigma}$ down to a point $e^\infty_j$ --- thus we can also view $\Sigma$ as a punctured surface $\Sigma_{\bar{e}} \setminus \{e^\infty_1, \dots, e^\infty_h\}$. This comes endowed with natural closed curves $\bgamma_{\bar{e}}$ which are the closures of $\bgamma$ under this identification, and base-points $\bp_{\bar{e}} = \bp \cup \bp(\bb)$ so that $\cD_{\bar{e}} = (\Sigma_{\bar{e}}; \bgamma_{\bar{e}}; \bp_{\bar{e}})$ is a diagram.

\subsubsection{Converting between diagrams}
Most of the content of this paper is in relating counts of certain moduli spaces related to bordered diagrams with those related to multi-pointed Heegaard multi-diagrams. As such, we will need to have a dictionary between concepts for the above two diagrams.

We will work with a special class of diagrams in both cases. We first need to introduce the notion of a \emph{handle}.

\begin{defn}For $n\in \ZZ$, the \emph{closed $n$-twisted handle}  is a tuple $$\handle(n) = (H; \alpha_\handle(m), \beta_\handle, \alpha_\handle(l); \bp_\handle),$$ where:
\begin{itemize}\item $H$ is the tube $S^1 \times [-1,1]$, with left and right ends $\handle_L := S^1 \times \{-1\} $ and $\handle_R := S^1 \times \{1\}$ respectively;
\item $\alpha_\handle(m)$ is the curve $S^1 \times \{0\}$;
\item $\beta_\handle$ is the arc $\{0\} \times [-1,1]$;
\item $\alpha_\handle(l)$ is the arc  $\{(e^{n i \pi t},t): t \in [-1, 1]\}$.
\item $\bp_\handle$ is the pair of basepoints $(w_\handle, z_\handle)$ as defined below.
\end{itemize}

For the definition of $\bp_\handle$, we choose a disc neighbourhood of the point $\alpha_\handle \cap \gamma_\handle$ which is sufficiently small that the triple $(D, \alpha_\handle \cap D, \gamma_\handle \cap D)$ is modelled on the unit disc $\DD \subset \CC$, together with the imaginary and real axes respectively. Under this identification, we put $w_\handle = \frac{1}{2}e^{\frac{3 \pi i}{4}}$ and $z_\handle = \frac{1}{2}e^{\frac{-3 \pi i}{4}}$.

The  \emph{bordered $n$-twisted handle} is a tuple $\border{\handle}(n) = (\border{H}; \alpha_\border{\handle} (m), \beta_\border{\handle},  \alpha_\border{\handle}(l); b)$, where $\border{H}, \alpha^m_\handle, \beta_\handle, \alpha^l_\handle$ are the punctured tube and arcs induced from ${H}, \alpha_\handle, \beta_\handle, \gamma_\handle$ by removing $D$ from $H$. Under the identification of $\DD$ with $D$, the boundary of $\border{H}$ consists of three components: $\border\handle_L, \border\handle_R$ and a new circle $\border\handle_\partial$ modelled upon $S^1 = \partial \DD$. We define $b$ to be the point in $\eta_\partial$  that corresponds with $e^{-\frac{\pi i}{4}}$.

It will be useful for us to note here that there is a clear bijection between the sets of curves $\alpha_\handle(m), \beta_\handle, \alpha_\handle(l)$ and $\alpha_\border{\handle} (m), \beta_\border{\handle},  \alpha_\border{\handle}(l)$ respectively, given by the map $\cap_D(\gamma) = \gamma \cap (\Sigma \setminus D)$.
\end{defn}

Please see Figure \ref{fig:closedhandlelong} for a picture of the closed $-6$-twisted handle with the curves $\beta_\handle$ and $\alpha_\handle(l)$ pictured; and Figure \ref{fig:closedhandlemerid} for a picture of a closed handle with the curves $\beta_\handle$ and $\alpha_\handle(m)$ shown. A bordered $-6$-twisted handle is also shown in Figure \ref{fig:borderedhandle}.

We will shortly construct some diagrams by gluing handles to a multi-pointed Heegaard multidiagram in a way specified by some fixed data. First, we consider sequences of idempotents:

\begin{defn}An \emph{$h$-splaying sequence of idempotents} is a sequence $\arrowiota$ of $m+1$ sequences of elements in the subring of idempotents of the $m$-torus algebra: $\arrowiota = \arrowiota^0, \dots, \arrowiota^{m}$, where $\arrowiota^i = (\iota^i_1, \dots, \iota^i_h)$ and each $\iota^i_j$ is one of $\iota_m$ or $\iota_l$.
\end{defn}
Note that if $\arrowrho = \arrowrho_1, \dots, \arrowrho_h$ is a set of sequences of Reeb chords and $\vec\sigma$ is a splicing for $\arrowrho$, the associated sequence $\arrowiota_\sigma$ is an $h$-splaying sequence of idempotents.
In general, the data we shall require is as follows.
\begin{defn}Suppose $\cD$ is a basic Heegaard diagram $(\Sigma; \balpha, \bbeta; \bp)$, where $|\bp| = 2h$. By definition, we can number the regions $A$ of $\Sigma - \balpha$ and $B$ of $\Sigma - \bbeta$ by the numbers $1, \dots, h$ so that the points $\bp = \{w_1, z_1, \dots, w_h, z_h\}$ satisfy $(w_i, z_i) \in A_i \cap B_i$. 
 
A set of \emph{stabilisation data for $\cD$} is a tuple $d = (\blambda_{w,z}, \blambda_{z,w}, \bn, \varepsilon, \arrowiota)$, where:
\begin{enumerate}
\item $\blambda_{w,z} = \{\lambda_{w_1, z_1}, \dots,  \lambda_{w_k, z_k}\}$ and $\blambda_{z,w} = \{\lambda_{z_1, w_1}, \dots,  \lambda_{z_k, w_k}\}$ are oriented arcs in $\Sigma$ chosen such that:\begin{itemize}\item $\partial \lambda_{w_i, z_i} = z_i - w_i$, \item $\partial \lambda_{z_i, w_i} = w_i - z_i$, \item $\blambda_{w,z} \cap \balpha = \emptyset$, and  \item$\blambda_{z,w} \cap \bbeta = \emptyset$.\end{itemize}
\item $\bn = (n_1, \dots, n_h)$ is a sequence of integers;
\item $\varepsilon$ is a small approximation parameter; and
\item $\arrowiota$ is an $h$-splaying sequence of idempotents.\end{enumerate}
\end{defn}

This is demonstrated in Figure \ref{fig:stabdata} for a specific choice of basic Heegaard diagram. (This is a  $4$-pointed Heegaard diagram for the Hopf link in $S^3$.)
\begin{figure}[h]
\centering{
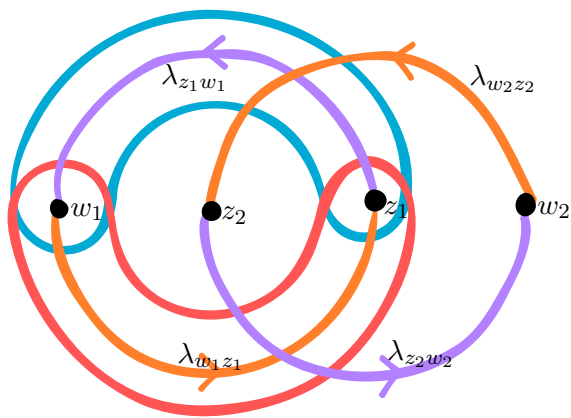
\caption{Stabilisation data for a $4$-pointed basic Heegaard diagram. Here, the $\balpha$ curves are shown in red and the $\bbeta$ in blue.}
\label{fig:stabdata}
}
\end{figure}

We now describe the process of gluing handles to a basic Heegaard diagram. The reader is strongly encouraged to refer to Figure \ref{fig:borderedhandle} either whilst or instead of reading the following definition.

\begin{defn}Let $\cD = (\Sigma; \balpha, \bbeta; \bp)$ be a basic Heegaard diagram, where $|\bp| = 2h$, and let $d = (\blambda_{w,z}, \blambda_{z,w}, \bn, \varepsilon, \arrowiota)$ be a set of stabilisation data for $\cD$. Suppose that  $(\border{\handle}_1, \dots, \border{\handle}_h)$ is a set of bordered handles.

We label the elements of  $\bp$ by $(w_1, z_1, \dots, w_h, z_h)$, where $w_i$ and $z_i$ both belong to the same component of $\Sigma - \balpha - \bbeta$. Choose a pair of small neighbourhoods $N^w_i$ and $N^z_i$ around $w_i$ and $z_i$ which are holomorphically identified with the unit disc $\DD \subset \CC$ in such a way that the baspoint $z_i$ is identified with the origin and the arcs $\lambda_{z_i, w_i} \cap N^w_i$ and $\lambda_{z_i, w_i} \cap N^w_i$ are identified with the positive and negative sides of the real axis respectively; and that $\lambda_{w_i, z_i} \cap N^z_i$ and $\lambda_{w_i, z_i} \cap N^z_i$  are identified with the positive and negative sides of the real axis respectively.

Now choose homeomorphisms $h_i^L$ (resp. $h_i^R$) from $\border{\handle}_i^L$ (resp. $\border\handle_i^R$) to $\partial N^w_i$  (resp. $\partial N^z_i$), in such a way that $h_i^L(\alpha_{\border\handle_{i}}(l) \cap \border\handle_i^L) = \lambda_{z_i, w_i}  \cap N^w_i$ and $h_i^L(\beta_{\border\handle_{i}} \cap \border\handle_i^L) = \lambda_{w_i, z_i}  \cap N^w_i$ (resp. $h_i^R(\alpha_{\border\handle_{i}}(l) \cap \border\handle_i^L) = \lambda_{z_i, w_i}  \cap N^w_i$ and $h_i^R(\beta_{\border\handle_{i}} \cap \border\handle_i^L) = \lambda_{w_i, z_i}  \cap N^w_i$).

Form the surface $$\border\Sigma(d) : = (  \Sigma - N_1^w - N_2^z - \cdots - N_h^w - N_h^z )\cup_{h^L, h^R} (H_1, \dots, H_h)$$  by gluing the handles to $\Sigma$ via the homeomorphisms $h_i^L$ and $h_i^R$. We denote by $\border\glue$ the composition of the projection map from $\Sigma$ to $(\Sigma - N_1^w - N_2^z - \cdots - N_h^w - N_h^z)$ with the gluing map; that is, it is a map
$$\border\glue: \Sigma \sqcup H_1 \sqcup \dots \sqcup H_h \rightarrow \border\Sigma(d).$$
There is a set of $h$ curves $\bbeta_\border\handle; = (\lambda_{w_1, z_1} \cup_{h_i^L, h_i^R} \beta_{\border\handle_1}, \dots, \lambda_{w_h, z_h} \cup_{h_h^L, h_h^R} \beta_{\border\handle_h})$ contained in $\Sigma(d)$, and also a sequence of arcs $\balpha^a_{\border\handle}$ given by $\alpha^a(m_i) := \border\glue(\alpha_{\border\handle_i}(m))$ and $\alpha^a(l_i) : = \alpha_{\border\handle_i}(l) \cup_{h_h^L, h_h^R} \lambda_{z_1, w_1}$, for $i = 1, \dots, h$.

These allow us to form the bordered diagram $ \border\cD(d) = (\border\Sigma(d), \border\balpha(d)^a, \border\balpha(d)^c, \border\bbeta(d), \bb(d)))$, where:\begin{itemize}\item $\Sigma(d)$ is as above;\item $\balpha(d)^a$ are the arcs $\balpha^a_\border\handle$; \item$\balpha(d)^c$ are the images $\border\glue(\balpha)$;  \item$\bbeta(d)$ is the union of $\bbeta_\border\handle$ and the images $\border\glue(\bbeta)$; and \item $\bb(d)$ is the image $\border\glue(\bb(\border\handle_1) \cup \dots \cup \bb(\border\handle_h))$.
\end{itemize}
We call this diagram the \emph{bordering} of $\cD$; any bordered diagram that arises as $\cD(d)$ for some $(\cD, d)$ we will call \emph{stable}.
\end{defn}

\begin{figure}[h]
\centering{
\makebox[\textwidth]{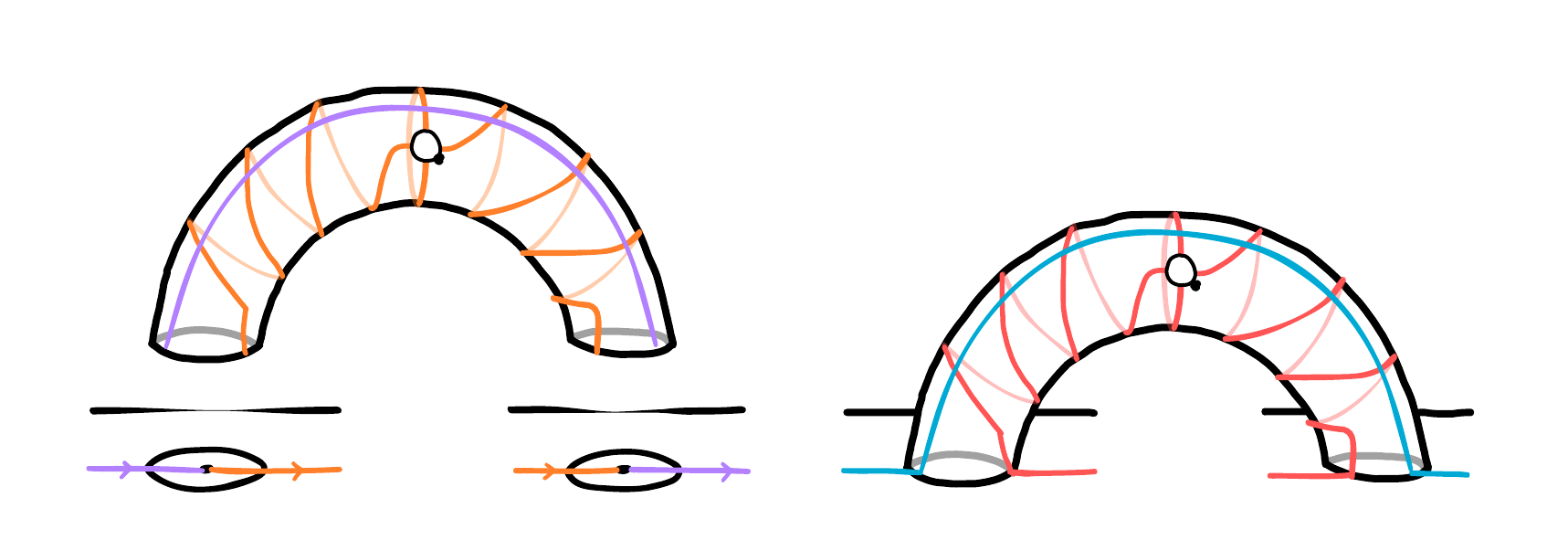}
\caption{The bordered stabilisation of a basic Heegaard diagram at one pair of basepoints}
\label{fig:borderedhandle}
}
\end{figure}

\begin{figure}[h]
\centering{
\makebox[\textwidth]{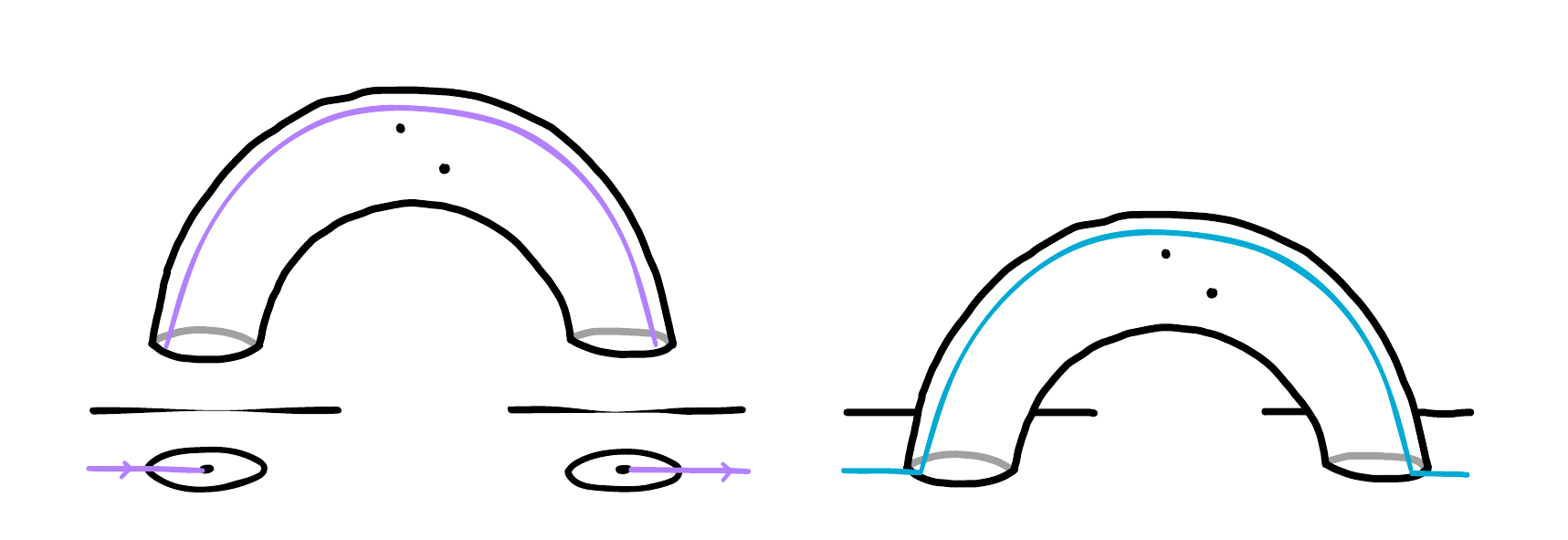}
\caption{The closed stabilisation of a basic Heegaard diagram at one pair of basepoints $w_i, z_i$, where the idempotent $\iota = \iota_{\delta_1} \otimes \cdots \otimes \iota_{\delta_m}$ satisfies $\delta_i = m$}
\label{fig:closedhandlemerid}
}
\end{figure}

Note that in this construction, the elements of the approximation data $\varepsilon$ and $\arrowiota$ were redundant. A small adaptation of the above also gives us a multi-pointed Heegaard multi-diagram (illustrated in Figure \ref{fig:closedhandlelong} and Figure \ref{fig:closedhandlemerid}):

\begin{defn}\label{defunsplayed} Let $\cD = (\Sigma; \balpha, \bbeta; \bp)$ be a basic Heegaard diagram, where $|\bp| = 2h$, and let $d = (\blambda_{w,z}, \blambda_{z,w}, \bn, \varepsilon, \arrowiota)$ be a set of stabilisation data, such that $\arrowiota$ consists of a single idempotent of $ \cT^h$, $\iota = (\iota_{\delta_1}, \cdots, \iota_{\delta_h})$  where each $\delta_i \in \{m, l\}$. Suppose further that  $(\handle_1, \dots, \handle_h)$ is a set of closed handles.

One can imitate the construction of the surface $\border\Sigma(d)$ from the previous definition, but with $(\handle_1, \dots, \handle_h)$ in place of $(\border\handle_1, \dots, \border\handle_h)$ to yield a surface $\Sigma(d, \iota) : = \Sigma \cup_{h^L, h^R} (H_1, \dots, H_h)$  by gluing the handles to $\Sigma$ via homeomorphisms $h_i^L$ and $h_i^R$. We denote by $\glue$ the corresponding gluing map $\glue: \Sigma \sqcup H_1 \sqcup \dots \sqcup H_h \rightarrow \Sigma(d, \iota)$.

Again, there is a set of $h$ curves $\bbeta_\handle; = (\lambda_{w_1, z_1} \cup_{h_i^L, h_i^R} \beta_{\handle_1}, \dots, \lambda_{w_h, z_h} \cup_{h_h^L, h_h^R} \beta_{\handle_h})$ contained in $\Sigma(d, \iota)$. There is also a sequence of curves contingent on $\iota$, given by the set $\balpha_\handle(d, \iota) = \{\gamma_i(\delta_i)\}$, where $$\gamma_i(\delta_i) = \left\{\begin{array}{c c}  \glue(\alpha_{\handle_i}(m)) & \mbox{ if } \mu_i = m\\
\alpha_{\handle_i}(l) \cup_{h_i^L, h_i^R} \lambda_{z_i, w_i} & \mbox{ if } \mu_i = l \end{array}\right. .$$

These allow us to form the diagram $\cD(d, \iota) = (\Sigma(d, \iota), \balpha(d, \iota), \bbeta(d, \iota), \bp(d, \iota))$, where:\begin{itemize}\item $\Sigma(d, \iota)$ is as above; \item$\balpha(d,\iota)$ is the union of $\balpha_\handle(d, \iota)$ and  the images $\glue(\balpha)$;\item $\bbeta(d, \iota)$ is the union of $\bbeta_\handle$ and the images $\glue(\bbeta)$; and\item $\bp(d, \iota)$ is the image $\glue(\bp_{\handle_1} \cup \dots \cup \bp_{\handle_h})$.\end{itemize}
\end{defn}

\begin{figure}[h]
\centering{
\makebox[\textwidth]{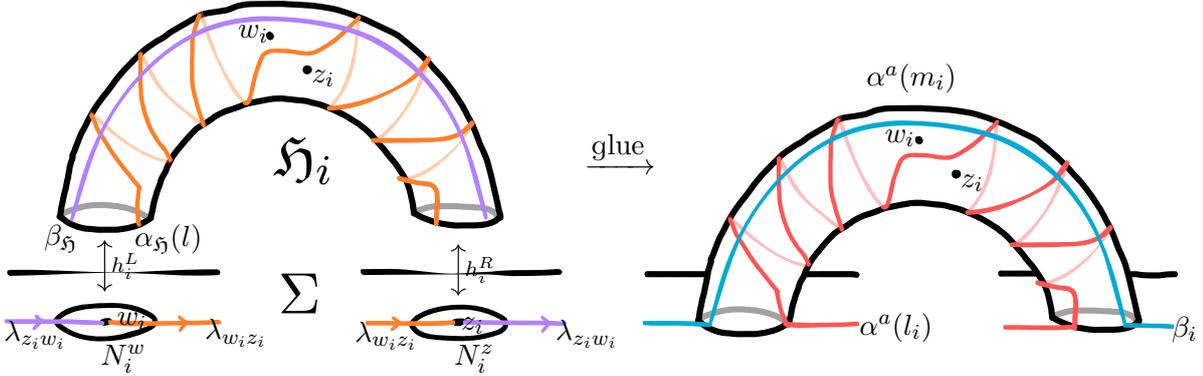}
\caption{The closed stabilisation of a basic Heegaard diagram at one pair of basepoints $w_i, z_i$, where the idempotent $\iota = \iota_{\delta_1} \otimes \cdots \otimes \iota_{\delta_m}$ satisfies $\delta_i = l$.}
\label{fig:closedhandlelong}
}
\end{figure}

Note that the map $\cap_D: \handle \rightarrow \border\handle$ extends to a map $\cap_D: \cD \rightarrow \border\cD(d, \iota)$, which yields a bijection $\cap_D: \bbeta(d, \iota) \rightarrow \bbeta(d)$. It no longer yields a bijection of the sets $\balpha(d, \iota)$; rather, it gives an injective map sending $\glue(\balpha) \mapsto \border\glue(\balpha)$ bijectively, and each element $\gamma_i(\delta_i) \in \balpha_\handle(d, \iota)$ to the corresponding arc $\cap_D(\gamma_i(\delta_i)) = \alpha^a(\delta_i)$. It follows that to every generator $\bx \in \cG(\border\cD)$ we may associate a generator $\cap_D^{-1}(\bx) \in \cD(d, \iota(\bx))$ (we often abuse notation for this generator and also denote it by $\bx$).

The diagram $\cD(d, \iota)$ as defined above constitues a multi-pointed Heegaard multi-diagram (in fact a simple Heegaard diagram), but does not use the information $\varepsilon$ from the approximation data $d$, and is restricted to the case where $\arrowiota$ is a single set of idempotents. In more generality, we make a similar construction:

\begin{defn}Let $\cD = (\Sigma; \balpha, \bbeta; \bp)$ be a basic Heegaard diagram, where $|\bp| = 2h$, and let $d = (\blambda_{w,z}, \blambda_{z,w}, \bn, \varepsilon, \arrowiota)$ be a set of stabilisation data for $\cD$ such that $\arrowiota =  \iota^1, \dots, \iota^m$.

The \emph{splayed diagram} associated to the pair $(\cD, d)$ is the tuple $$\cD(d):= (\Sigma(d), \balpha^1(d), \dots, \balpha^m(d), \bbeta(d), \bp(d))$$ given by:
\begin{itemize}\item $\Sigma(d) = \Sigma(d, \iota^1)$ (this is the same as $\Sigma(d, \iota^i)$ for any $i$).
\item $\balpha^i(d)$ is given by the $\varepsilon$-small approximation $H(\balpha(d, \iota(\iota_i), i\varepsilon / m)$ to the set of curves $\balpha(d, \iota(\iota_i))$ for each $i = 1, \dots, m$, as described in Definition \ref{defunsplayed}.

Here, for each $i = 1, \dots, h$, if we denote by $\alpha(d, \iota, \handle_i)$ curve of $\balpha(d, \iota)$ which intersects $\handle_i$, then we chose the identification of the tubular neighbourhood $N(\alpha, i)$ of $\alpha(d, \iota, \handle_i)$ with $S^1 \times [0,1]$ so that the intersection of the complementary curve $\alpha(d, \iota^{c}, \handle_i)$ with $N(\alpha, i)$ coincides with the arc $\pi \times [0,1]$. (See Figure \ref{fig:approxexample}.)

\begin{figure}[h]
\centering{
\scalebox{0.7}{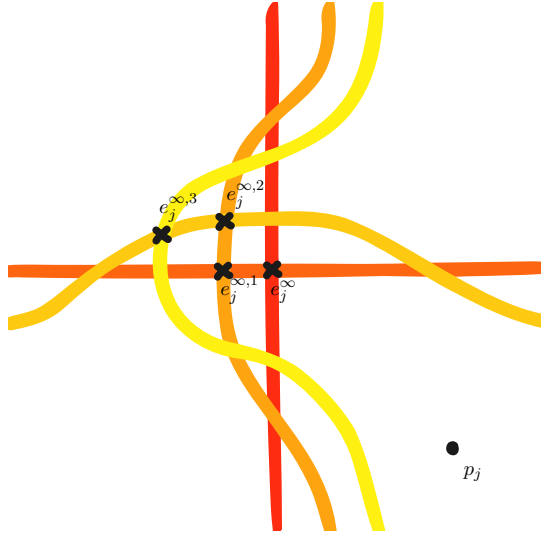}
\caption{A splayed diagram near $e^\infty_j$.}
\label{fig:approxexample}
}
\end{figure}

\item $\bbeta(d)$ is given by the set $\bbeta(d, \iota_1)$ (this, again, is really independent of $\iota_1$)
\item $\bp(d)$ is given by the set $\cup_{i=1}^m\bp_{\handle_i}$.
\end{itemize}
\end{defn}

If $\border\cD$ is a stable bordered diagram $\border\cD = \border\cD(d)$ for some $d$ and $\cD$, then we will refer to the corresponding splayed diagram $\cD(d)$ as the \emph{splaying} of $\border\cD$ . We will often denote this simply by $\cD(\varepsilon, \arrowiota)$, as these are the only parts of the splaying data $d$ not determined by the bordered diagram $\border\cD(d)$.

There is one further variant which we shall need to consider:
\begin{defn}Suppose that $\cD$ is a basic Heegaard diagram, $d$ is a set of stabilisation data for $\cD$. Consider the corresponding splayed diagram $$\cD(d, \arrowiota):= (\Sigma(d), \balpha^1(d), \dots, \balpha^m(d), \bbeta(d), \bp(d)).$$ 
We note that each $\balpha^i(d)$ is disjoint from the points $e^\infty$. It follows that there is an induced set of sets of curves in the surface $\border\Sigma(d)$, given by the images of the $\balpha^i(d)$ under the map $\cap_D$. As such, there is a corresponding bordered multi-diagram $\cD^\partial(d, \arrowiota)$ given by
$$\cD^\partial(d, \arrowiota) : = (\Sigma(d, \arrowiota), \balpha^a(d, \arrowiota) \cup \balpha^c(d, \arrowiota), \balpha^1(d,\arrowiota), \dots, \balpha^m(d, \arrowiota), \bbeta(d, \arrowiota), \bp(d, \arrowiota)).$$
(We often relabel so that $\Sigma(d, \arrowiota) = \Sigma(d)$, etc.)

We call this the \emph{partial splaying of $\cD$} associated to the data $d, \arrowiota$
\end{defn}

Later, we will need to perturb subdiagrams into a more standard form:
\begin{defn}Suppose that $\cD$ is a splayed or partially splayed diagram, and consider a subdiagram $\cD'$ with attaching curves $\bgamma = \bgamma^{i_1}, \dots, \bgamma^{i_k}$. Let $\bgamma' = \bgamma^{i'_1},\dots, \bgamma^{i'_{k'}}$ denote the subsequence of these attaching curves obtained by omitting any $\bgamma^{i_j}$ such that $\bgamma^{i_{j-1}} \equiv \bgamma^{i_{j}}$.

Each set of curves $\bgamma^{i'_j}$ is given either by $H(\balpha(d, \arrowiota^{i'_j}, i'_j \varepsilon/m)$ for some sequence of idempotents $\arrowiota^{i'_j}$ or is equal to $\balpha$ or $\bbeta$ --- say there are $m'$ curves of the first sort.  We let $\bgamma_{\text{reg}}$ denote the sequence of attaching curves obtained from $\bgamma'$ by replacing each $\bgamma^{i'_j}$ of the first sort with the perturbation $\bgamma^{i'_j}_{\text{reg}} : = H(\balpha(d, \arrowiota^{i'_j}, j \varepsilon/m'))$, and $\cD'_{\text{reg}}$ be the diagram obtained from $\cD'$ by replacing the attaching curves used with $\bgamma_{\text{reg}}$.

We shall call this diagram the \emph{regularisation} of $\cD'$.

\end{defn}

The bordered stabilisation and splaying of a Heegaard diagram exhibit a number of similarities.

Let $\cD : = (\Sigma; \balpha, \bbeta; \bp)$ be a basic Heegaard diagram, and $d = (\blambda_{w,z}, \blambda_{z,w}, \bn, \varepsilon, \arrowiota)$ some stabilisation data for $\cD$. Recall that associated to the curves $\balpha$ are the $m$-fold $\varepsilon$ approximations $\balpha^1, \dots, \balpha^m$.

The image of $\balpha^i$ under the map $\glue$ coincides with a set of $g-1$ of the closed stabilisation curves $\balpha^i(d) \subset \cD(d) = (\Sigma(d); \balpha^1(d), \dots, \balpha^m(d); \bbeta(d); \bp(d)) $, or similarly in the partial splaying of $\cD$. We number the $\balpha^i(d)$ so that these $g-h$ curves are given by the set $\{\alpha^i_{h+1}(d), \dots, \alpha_{g+h-1}^i(d)\} \subset \balpha^i(d)$. It follows that to any distinguished generators $\vec\btheta_{\mathrm{approx}} = \btheta_{\mathrm{approx}}^{1,2}, \dots, \btheta_{\mathrm{approx}}^{m-1,m}$ for the $m$-fold $\varepsilon$-approximation of $\balpha$ we may associate a corresponding set of partial generators $\glue(\vec\btheta_{\mathrm{approx}})$ for the splayed or partially splayed diagram $\cD(d)$, which we also denote by $\vec\btheta_{\mathrm{approx}}$.

The curves in each $\balpha^i(d)$ not in the image of $\glue(\balpha^i)$, which we assume are labelled by $\alpha^i_1, \dots, \alpha^i_h$ such that, for each $i$ either:
\begin{itemize}\item $\alpha^i_j \cap \alpha^{i+1}_j$ is a single point $e^{\infty,i}_j$ (this corresponds with the idempotents $\iota_j^i$ and $\iota_{j}^{i+1}$ being different).
\item $\alpha^i_j$ and $\alpha^{i+1}_j$ are Hamiltonian isotopic (in fact, they are $\varepsilon/m$-small approximations of one another).
\end{itemize}

In the second case, there are two points in $\alpha^i_j \cap \alpha^{i+1}_j$. For every such $i, j$, we choose an element $\theta^{i,i+1}_{j} \in \alpha^i_j \cap \alpha^{i+1}_j$. We term such a choice a \emph{splaying extension} of the distinguished generators $\vec\btheta_\mathrm{approx}$. For every fixed $i$, we collect the intersection points $e^{\infty, i, i+1}_j$ and $\theta^{i,i+1}_j$ into a set $\btheta^{i,i+1}_{\mathrm{ext}}$, and set $\btheta^{i,i+1} = \btheta^{i,i+1}_{\mathrm{ext}} \cup \btheta^{i,i+1}_{\mathrm{approx}}$ --- this is a bona fide set of generators dependent upon both the choice of distinguished generators and the choice of splaying extension.

\begin{exmp}\label{thetagen}Suppose that $\vec\btheta_{\mathrm{approx}}$ is the distinguished set of generators $\vec\btheta_+$ described in Example \ref{thetagen}. Then we shall call the corresponding set of $\btheta^{i,i+1}$ by the same name $\btheta^{i,i+1}_+$ --- and hope that this does not cause confusion.

If we wish to draw attention to the number of generators in $\vec\btheta_+$, we will write $\vec\btheta_+(k)$, where $k = |\vec\btheta_+(k)|$.
\end{exmp}
Recall there are corresponding generators $(\bx', \by')$ in the splayings $\cD(0, \iota(\bx))$ and $\cD(0,\iota (\by))$, given by the map $\cap_D^{-1}$ (see  the discussion following Definition \ref{defunsplayed}). The nearest point map gives a generator $\bar{\by}$ corresponding with $\by'$ in the set of generators $\cG(\balpha^m, \bbeta)$; it follows that the tuple $(\bx, \btheta^{1,2}, \dots, \btheta^{m-1, m}, \bar{\by})$ constitutes a sequence of generators for the splayed or partially splayed diagram $\cD(\varepsilon, \arrowiota)$. We call this sequence the \emph{splaying} of $(\bx, \by)$.

\begin{rmk}\label{nbhdrmk}For sufficiently small $\varepsilon$, for each $j$ all of the points $e_j^{\infty, i}$ may all be assumed to be contained in a small disc neighbourhood $N_j$ of $e^\infty_j$ which is disjoint from the curves $\bbeta(d)$.
\end{rmk}

\begin{exmp}\label{splayingexmp} Figures \ref{fig:splay1}, \ref{fig:splay2} and \ref{fig:splay3} show the bordered stabilisation of the diagram in Figure \ref{fig:stabdata} by the data  $d =(\blambda_{w,z}, \blambda_{z,w}, \bn, \varepsilon, \arrowiota)$, where  $\blambda_{w,z}$ and $\blambda_{z,w}$ are as given in Figure \ref{fig:stabdata}, $\bn = (0, 0)$, $\varepsilon$ isn't pictured, and $\arrowiota$ is stated below each figure. Also demonstrated are a series of generators, and their respective splayings.

\begin{figure}[h]
\centering{
\makebox[\textwidth]{\scalebox{0.5}{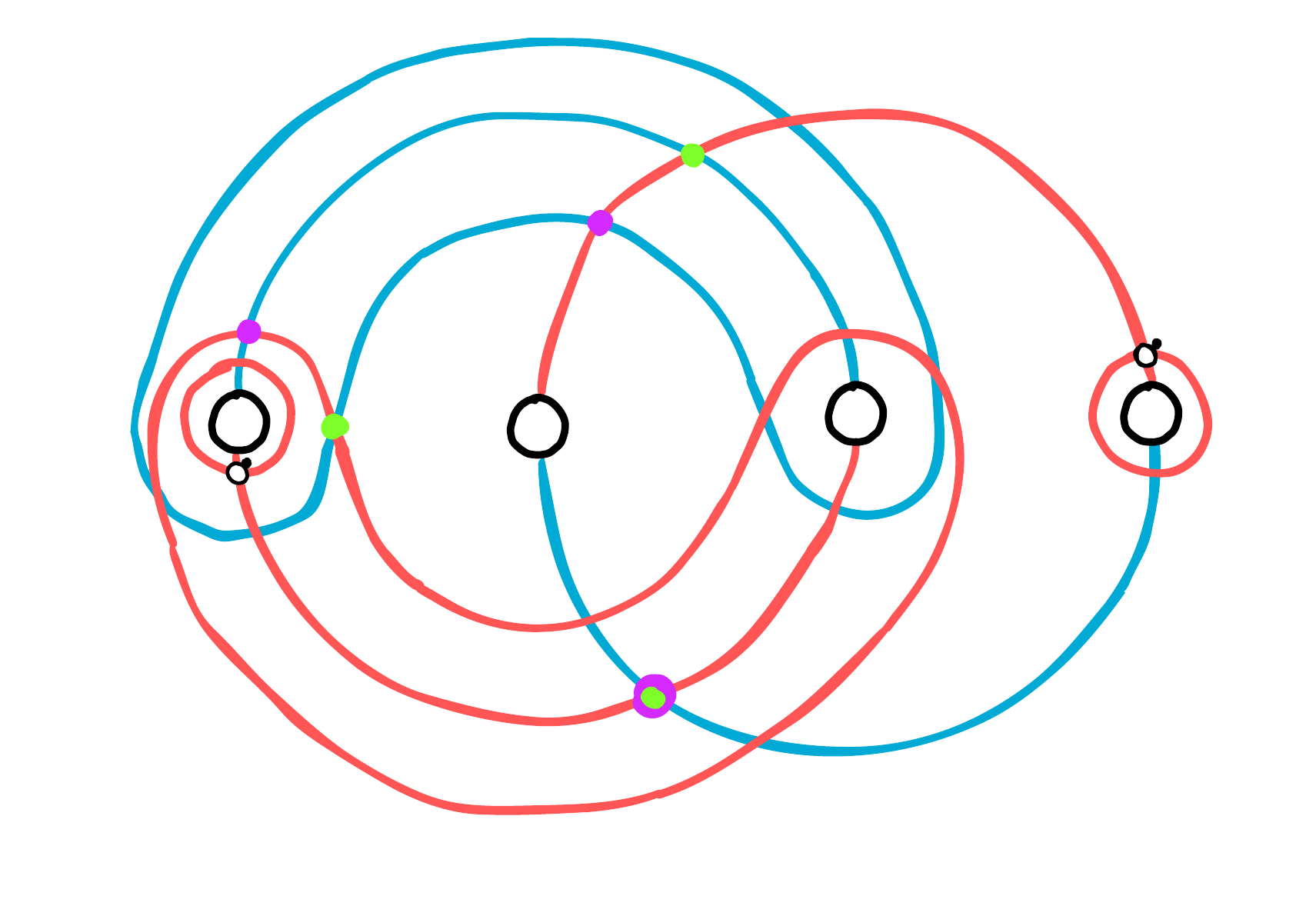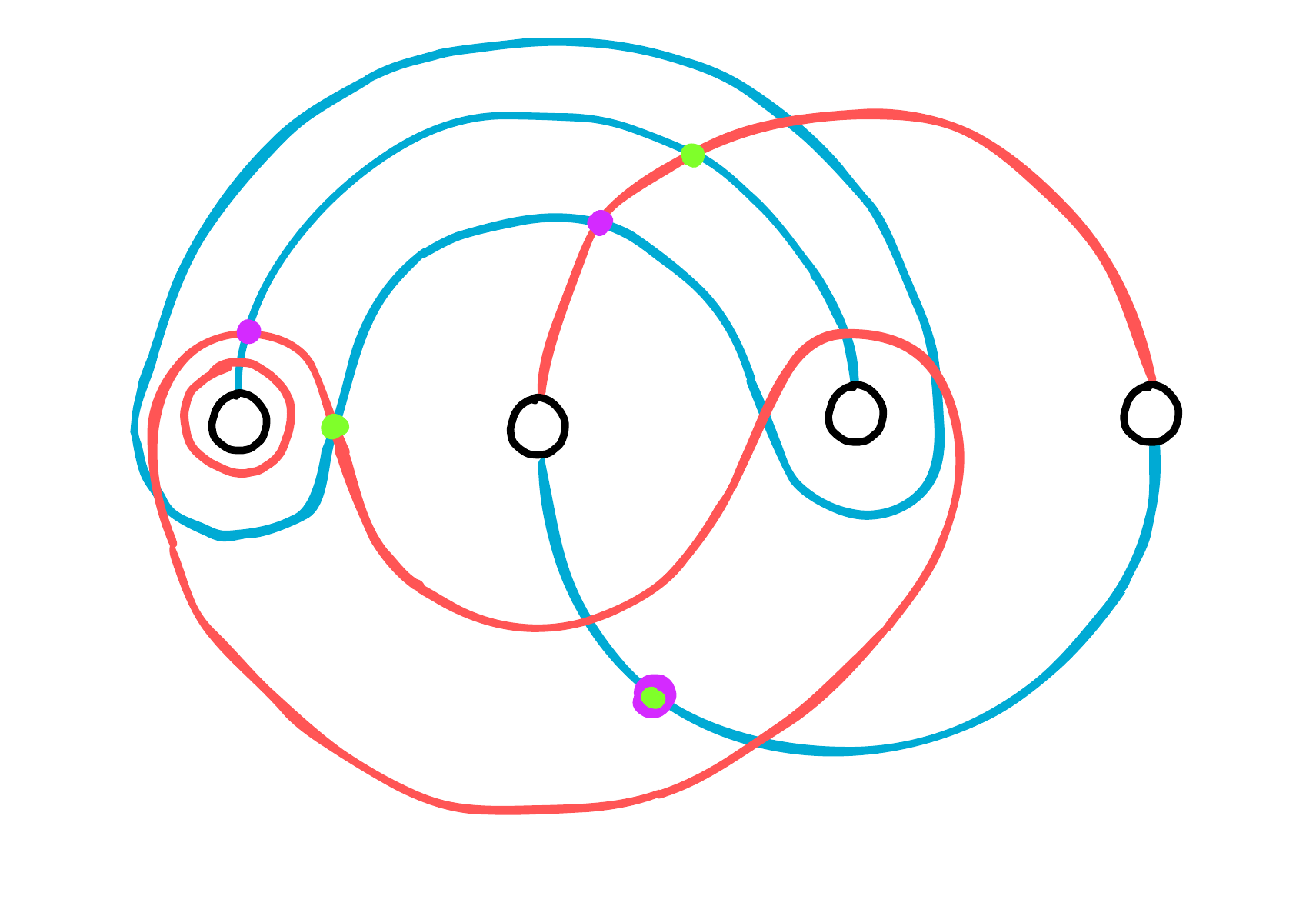}}
\caption{The splayings of the diagram in Figure \ref{fig:stabdata} with splaying data $d$ as stated in Example \ref{splayingexmp}. On the left is shown the bordered stabilisation with a pair of generators $\bx$ (violet) and $\by$ (lime green); on the right is shown the splaying by the idempotents $(\iota_l, \iota_l)$ together with the corresponding splayed generator.}
\label{fig:splay1}
}
\end{figure}

\begin{figure}[h]
\centering{
\makebox[\textwidth]{\scalebox{0.5}{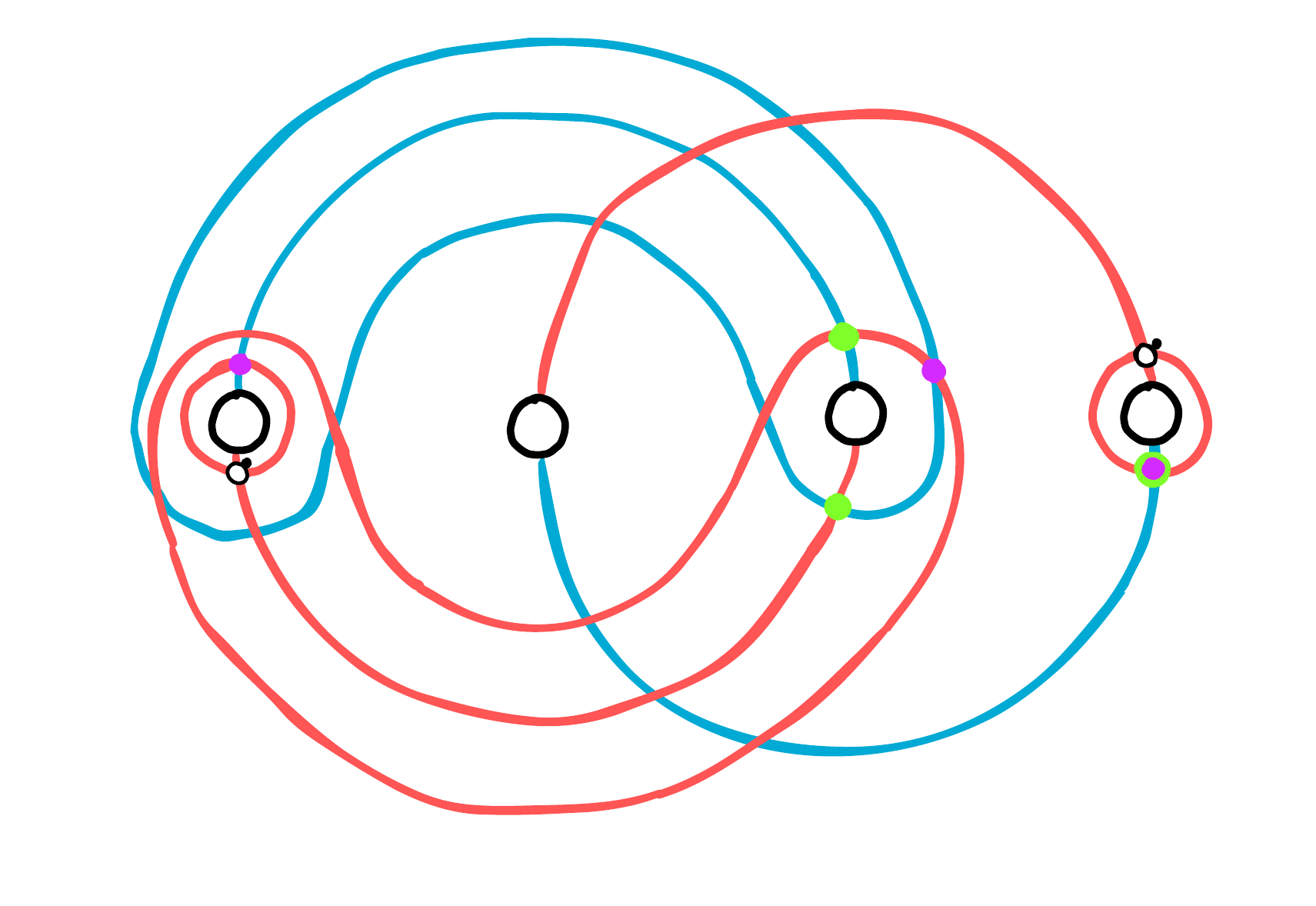 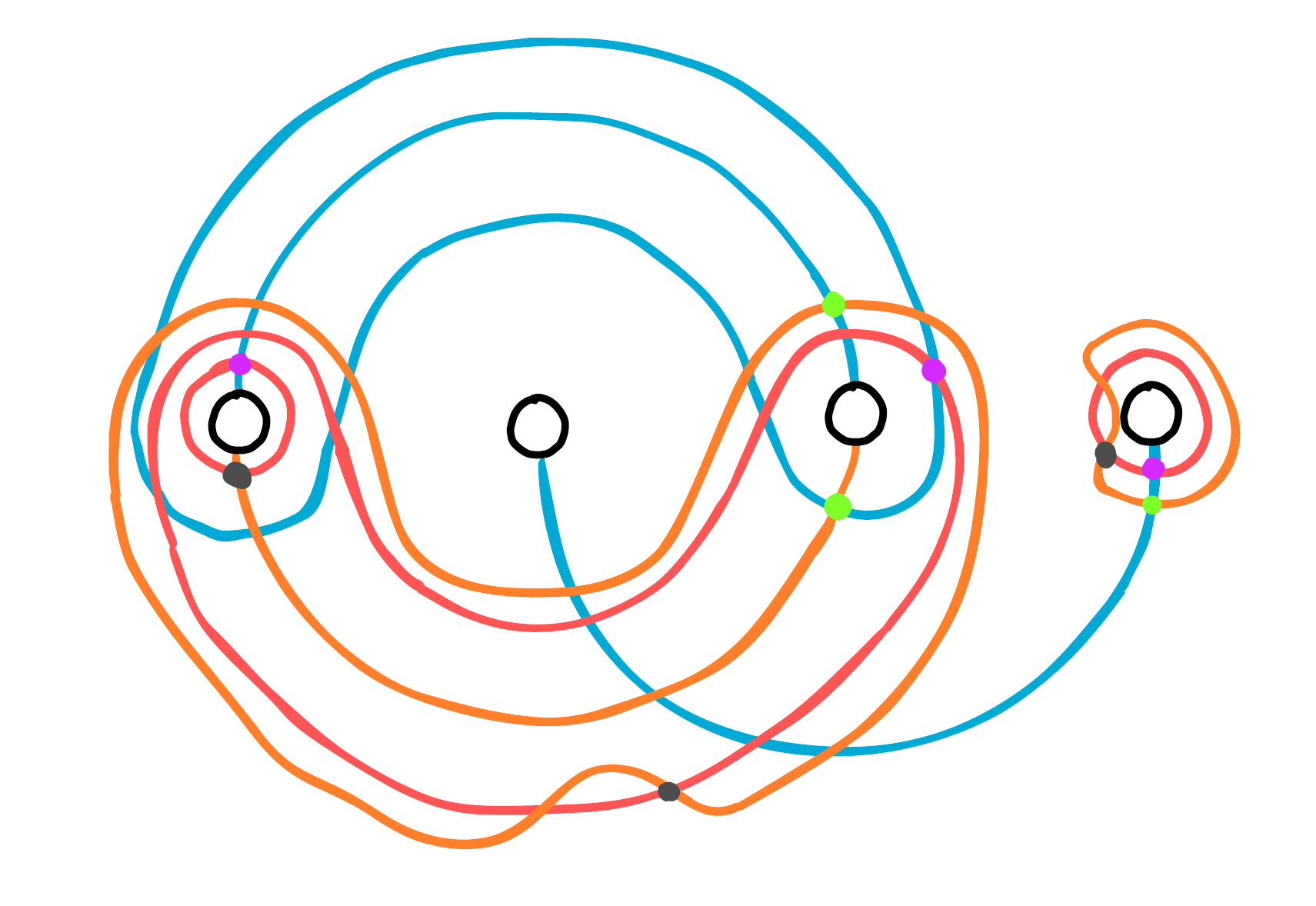}}
\caption{The splayings of the diagram in Figure \ref{fig:stabdata} with splaying data $d$ as stated in Example \ref{splayingexmp}. On the left is shown the bordered stabilisation with a pair of generators $\bx$ (violet) and $\by$ (lime green); on the right is shown the splaying by idempotents $((\iota_m, \iota_m), (\iota_l, \iota_m))$, together with the corresponding splayed generators. Here, the distinguished generator is shown in dark grey.}
\label{fig:splay2}
}
\end{figure}

\begin{figure}[h]
\centering{
\makebox[\textwidth]{\scalebox{0.5}{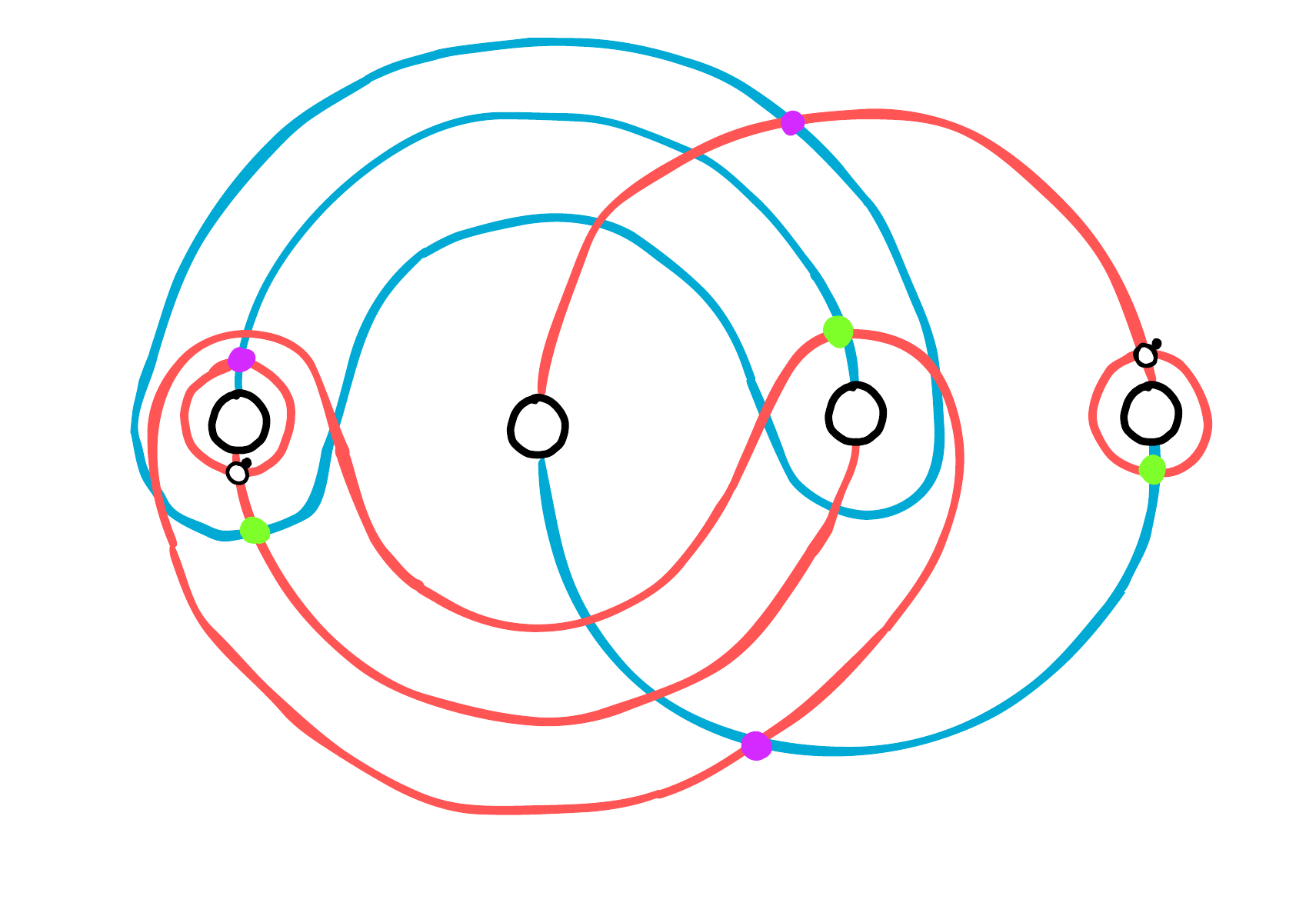 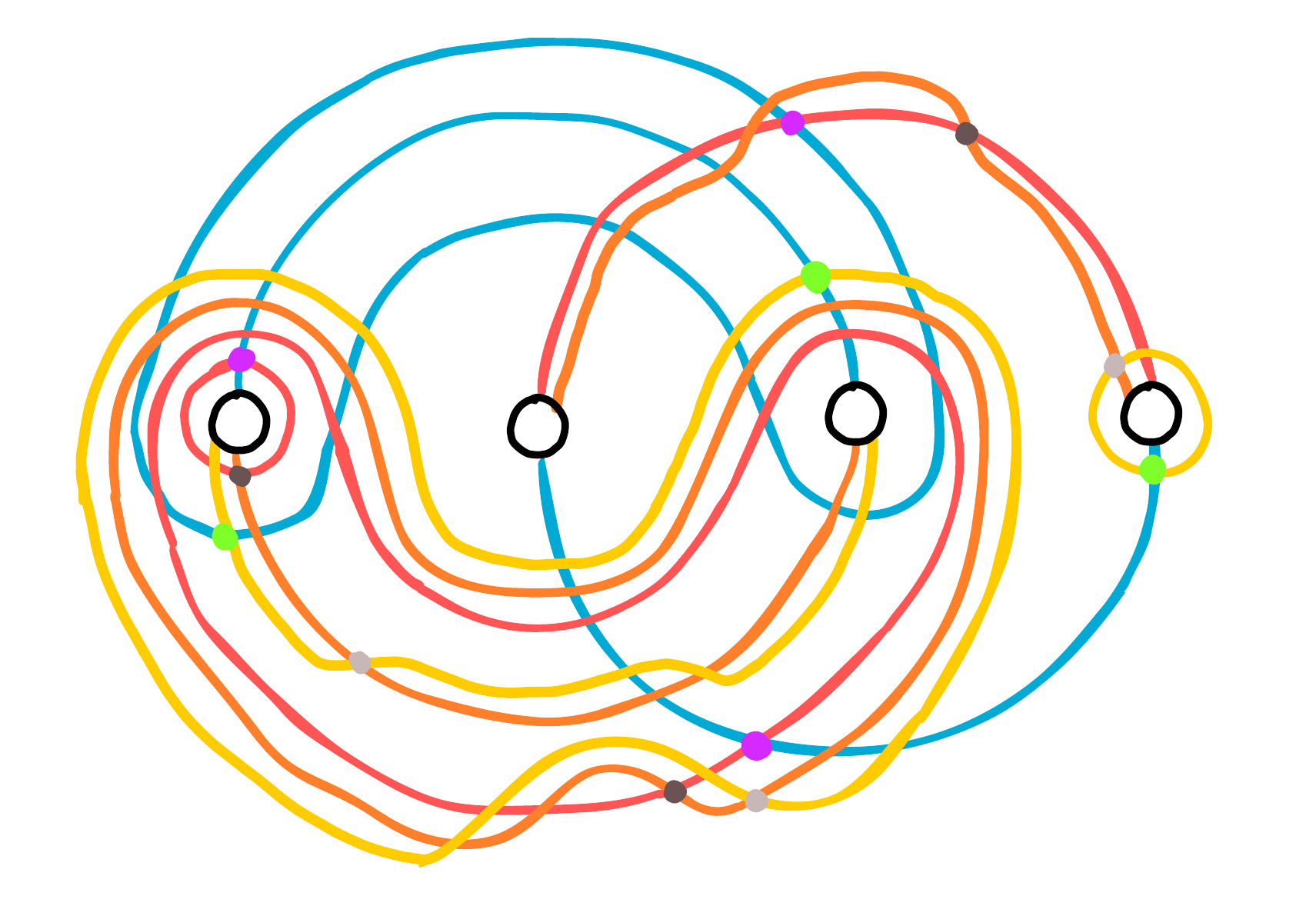}}
\caption{The splayings of the diagram in Figure \ref{fig:stabdata} with splaying data $d$ as stated in Example \ref{splayingexmp}. On the left is shown the bordered stabilisation with a pair of generators $\bx$ (violet) and $\by$ (lime green); on the right is shown the splaying by the sequence of sets of idempotents $((\iota_m, \iota_l), (\iota_l, \iota_m), (\iota_l, \iota_l))$ together with the corresponding splayed generators. Here, the distinguished generators are shown in dark grey and light grey respectively.}
\label{fig:splay3}
}
\end{figure}

\end{exmp}

\begin{convention}From here on, we will label the attaching curves of a splayed or partially splayed diagram $\cD(\varepsilon, \arrowiota)$ or $\cD^\partial(\varepsilon, \arrowiota)$ by $\bgamma^0, \dots, \bgamma^{m}$, where $\bgamma^0 = \bbeta$, and the curves $\bgamma^1, \dots, \bgamma^m$ are given by the splayed curves $\bgamma^i = \balpha^i(d)$, or, for a partially splayed diagram, $\bgamma^1 = \balpha = \balpha^a \cup \balpha^c$ and the subsequent curves are the closed curves $\bgamma^i = \balpha^i(d)$.
\end{convention}
\subsection{Domains}

Suppose that $\cD = (\Sigma; \bgamma^0, \dots, \bgamma^m; \bb; \bp)$ is a diagram. In the next chapter, we will be interested in maps
$$u: (S, \partial S) \rightarrow (\Sigma \times D_{m+1}, \bgamma^0 \times e^0 \cup \cdots \cup \bgamma^m \times e^m)$$
where $S$ is a Riemann surface with boundary and punctures upon its boundary, and $D_{m+1}$ is an $m+1$-gon --- that is, a disc with $m+1$ boundary punctures $v^{0,1}, \dots, v^{m,0}$ (see Section \ref{polygons}) --- with sides $e^0, \dots, e^{m}$. Given any such $u$, one can post-compose with the projection maps $\pi_\Sigma: \Sigma \times D_{m+1} \rightarrow \Sigma$ and $\pi_D: \Sigma \times D_{m+1} \rightarrow D_{m+1}$ to yield corresponding maps $u_\Sigma$ and $u_D$.

Near any boundary puncture $q$ of $S$ that is mapped to $v^{i, i+1}$ under $u_D$, the maps $u$ which we study will be asymptotic to a point $I^{i,i+1}(q)$ of the form
$$I^{i,i+1}(q) : = x^{i,i+1}(q) \times v^{i,i+1},$$
for some point $x^{i,i+1} \in \bgamma^i \cap \bgamma^{i+1}$, such that the set
$$\bx^{i,i+1} : = \bigcup_{q \in u_D^{-1}(v^{i,i+1})}x^{i,i+1}(q)$$
constitutes a generator for $(\Sigma, \bgamma^{i}, \bgamma^{i+1})$ --- that is, the sequence $\cG(u) : = (\bx^{0,1}, \dots, \bx^{m,0})$ is a sequence of generators for $\cD$.

We label the following objects:
\begin{align*}\bC^{i} &: = \bgamma^{i} \times e^i;\\ \bI^{i, i+1} &:= \bigcup_{q \in u_D^{-1}(v^{i, i+1})}I^{i, i+1}(q);\\
S^{\partial} &: = (\partial \Sigma \setminus \bb) \times e^0 .\end{align*}

It will follow from our assumptions on $u$ that there is an associated map $\bar{u}$ which represents some class in the relative homology group
$$H_2(\bar{\Sigma} \times I_s \times I_t, (\bC^0 \cup \cdots \cup \bC^m) \cup (\bI^{0,1}(\bx^{0,1} \cup \cdots \cup \bI^{m,0}(\bx^{m,0})) \cup S^\partial).$$
We denote by $\pi_2(\bx^{0,1}, \dots, \bx^{m,0})$ the set of all elements of the above group which map to the relative fundamental class of $\bI^{0,1}(\bx^{0,1}) \cup \cdots \cup \bI^{m,0}(\bx^{m,0})$ under the composition of the boundary homomorphism and collapsing the remainder of the boundary.

A related notion is that of a \emph{domain}. Namely, if $\cD$ is a diagram, then a \emph{domain} is an element of the homology group $$\dom(\cD) : = H_2(\bar{\Sigma}, \bgamma^1 \cup \cdots \cup \bgamma^m \cup \partial \bar{\Sigma}).$$ 
This homology group is generated by the regions $R_i$ of $\Sigma - \bgamma$, and thus any domain is given by a formal linear combination $\sum r_i R_i$.

A \emph{periodic} domain is a domain with zero multiplicity (coefficient) at all regions $R_i$ where $R_i \cap (\bz \cup \bb) \ne \emptyset$ and with boundary equal to a linear combination of (full) curves and arcs from $\bgamma$. A \emph{provincial periodic domain} is a domain which further has zero multiplicity at all regions $R_i$ such that $R_i \cap \partial \Sigma \ne \emptyset$.

We say a diagram $\cD$ is \emph{admissible} if every periodic domain in $\cD$ has both positive and negative coefficients; we say it is \emph{provincially admissible} if every provincial periodic domain in $\cD$ has both positive and negative coefficients.

Any homology class $B$ in the group $\pi_2(\bx^{0,1}, \dots, \bx^{m,0})$ can be projected under $\pi_\Sigma$ to an element of $\dom(\cD)$, which we call its \emph{domain}. One can see that a homology class is uniquely determined by the corresponding domain; thus we often conflate the two notions.

Given a sequence of Reeb chords $(\rho^1, \dots, \rho^k)$ of the torus algebra, there is an associated homology class $[\arrowrho] = [\rho^1]+[\rho^2] + \dots + [\rho^k]$, where $[\rho^i]$ is the image of the Reeb chord $\arrowrho^i$ in the relative homology group $H_1(S^1, \ba)$. If $\cD$ is a toroidal Heegaard multi-diagram with $h$ boundary components, each of the relative homology groups $H_1(\partial \Sigma_j, \ba_j)$ are identified with $H_1(S^1, \ba)$.

\begin{defn}\label{compatible}Suppose that $B \in \pi_2(\bx^{0,1}, \dots, \bx^{m, 0})$ is a domain. We shall say that $B$ is \emph{$\arrowrho$-compatible} for $\arrowrho = (\arrowrho_1, \dots, \arrowrho_h)$ if, for each $j$, we have that the image of $B$ in $H_1(\partial \Sigma_j, \ba_j)$ is equal to $[\arrowrho_j]$.
\end{defn}

\begin{convention}When $\cD$ is the splaying of a diagram $\border \cD$, if $\bx, \by \in \cG(\border \cD)$ are generators, we shall often abbreviate the corresponding group
$$\pi_2(\bx', \btheta^{1,2}_+, \dots, \btheta^{m-1,m}_+, \by')$$
as $\pi_2^{\btheta_+}(\bx,\by) \subset \dom(\cD(d))$.
\end{convention}

We end this subsection with a brief discussion of generators and domains in east-compactified diagrams, which will come in useful in Section \ref{eastcompactifying}. Recall that to any diagram $\cD$ is associated a corresponding east-compactified diagram $\cD_{\bar{e}}$. Generators for $\cD$ have obvious counterparts in $\cD_{\bar{e}}$, and we generally abuse notation and call these by the same label; any homology class $B$ for $\cD$ has a corresponding east-compactified class $B_{\bar{e}}$  in the relative homology group
$$H_2( \Sigma_{\bar{e}} \times D_n, ((\bgamma^1_{\bar{e}} \times e^1) \cup \cdots \cup (\bgamma^m \times e^m) ) \cup (\bI^{0,1} \times \cdots \times \bI^{m-1, 0})).$$

\subsubsection{Splaying domains}\label{splayingdomains}

We would like to understand the relationship between domains in a bordered diagram $\cD = (\Sigma; \balpha, \bbeta; \bb)$ and a corresponding (partial) splaying $\cD^\partial(\varepsilon, \arrowiota)$. First, we set up some notation for describing different qualities of sequences of Reeb chords.

Fix a stable bordered Heegaard diagram $\border\cD$, generators $\bx, \by \in \cG(\cD)$, a set of sequences of Reeb chords $\arrowrho_\boat = \arrowrho_{\boat, 1}, \dots, \arrowrho_{\boat, h}$, and a sequence of idempotents $\arrowiota$. Consider the partially splayed diagram $\cD^{\partial}(\varepsilon, \arrowiota)$ and the splayed diagram $\cD(\varepsilon, \arrowiota)$. We shall associate to any $\arrowrho_\boat$-compatible domain $B \in \pi_2^{\btheta_+}(\bx, \by) \subset \dom(\cD^{\partial}(\varepsilon, \arrowiota))$, or any domain $B \in \pi_2^{\btheta_+}(\bx, \by) \subset \dom(\cD(\varepsilon, \arrowiota))$ a sequence of Reeb chords $\arrowrho(B)$.

\begin{convention}Recall from the beginning of the section that any homology class in $\dom(\cD)$ corresponds with a homology class in $\dom(\cD_{\bar{e}})$ --- we make this identification for the rest of this section, and consider domains in diagrams where the underlying surface is closed.
\end{convention}

Consider a small neighbourhood $N_j$ of $e^\infty_j \in \Sigma_{\bar{e}}$, which is disjoint from $\bbeta$ and such that for every $i$, $e^{\infty,i}_j$ is contained in $N_j$, as in Remark \ref{nbhdrmk}. Suppose that $B \in \pi_2^{\btheta_+}(\bx, \by)$ is a domain, and that $u$ is a map from a surface with boundary and punctures upon the boundary to $\Sigma \times D_m$ or $\Sigma_{\bar{e}} \times D_m$, with homology class $B$. If $S_j$ denotes the restricted source $(\pi_\Sigma \circ u)^{-1}(N_j)$, then the map $u'_j = u|_{S_j}$ takes the form of a map
$$u'_j: (S_j, \partial S_j) \rightarrow (N_j \times D_{m+1}, \partial N_j \cup \gamma_j^1 \times e^1 \cup \cdots \cup \gamma_j^m \times e^m).$$

Maps of this form are in bijection with maps
$$u_j: (S_j, \partial S_j) \rightarrow (N_j \times D_{m+1}, \partial N_j \cup \delta_j^1 \times e^1 \cup \cdots \cup \delta_j^{m'} \times e^{m'}),$$
where $\delta_j^1, \dots, \delta_j^{m'}$ is the subsequence of $\gamma_j^1, \dots, \gamma_j^m$ formed by deleting every arc $\gamma_j^i$ such that $\gamma_j^i$ approximates $\gamma_j^{i-1}$. On the level of domains, the map $u_j$ has an associated domain
$$B_j \in H_2(N_j, \partial N_j \cup (\delta_j^1 \cup \cdots \cup \delta_j^{m'}))$$
which can be described by choosing a point $x_{R}$ for every region of $N_j - (\gamma_j^1 \cup \cdots \cup \gamma_j^{m})$ not contained in the approximation regions between any two adjacent curves $\gamma_j^i, \gamma_j^{i+1}$ which approximate one another: $B_j$ is the domain given by requiring that the multiplicity of $B_j$ at each $x_R$ is the same as the multiplicity of $B|_{N_j}$ at each $x_R$.

In what follows, it will help us to have an understanding of the homology groups $H_2(N_j, \partial N_j \cup \delta^i_j)$ and $H_2(N_j, \partial N_j \cup \delta^i_j \cup \delta^{i+1}_j)$.

The pair $(N_j, \delta^i_j)$ is modelled upon either the pair $(\DD, \Re)$ or $(\DD, \Im)$ depending upon whether $\delta^i_j$ is meridional or longitudinal. In these cases, the groups $H_2(\DD, \Re)$ and $H_2(\DD, \Im)$ are generated by the half-discs $\DD - \Re$ and $\DD-\Im$ respectively, only one of which has multiplicity zero at the base-point $p_j$ in each case. We label these homology classes by $D(\rho_{12})$ and $D(\rho_{23})$ respectively, and the corresponding domain under the identification with $(N_j, \delta^i_j)$ by $D_j^i(\rho_{12})$ and $D_j^i(\rho_{23})$ respectively (see Figure \ref{fig:nonjumplocal}). We generalise this to length $n_i$ repeated sequences of these chords, by setting $D_j^i(\rho_{12}, \dots, \rho_{12}) := n_i \cdot D_j^i(\rho_{12})$, and similar for $\rho_{23}$ --- in both instances, we allow the degenerate length 0 sequence with associated domain equal to zero.
\begin{figure}[h]
\centering{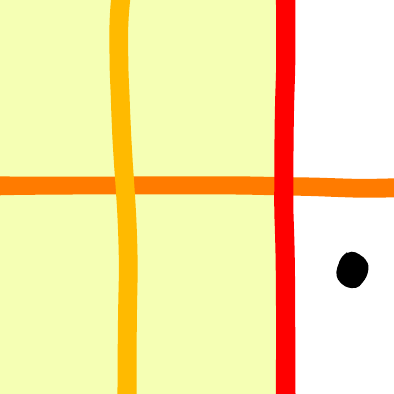\quad \quad 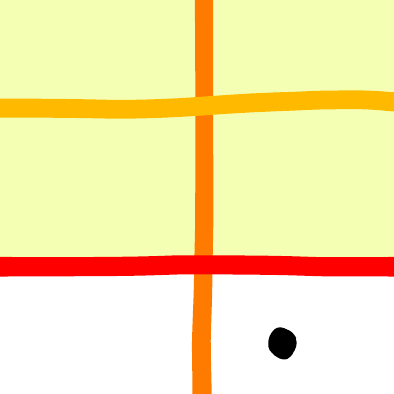
\caption{A series of examples of various domains (shaded green) in the neighbourhood $N_j$ of $e^\infty_j$. On the left is an example of $B^1(\rho_{12})$, and on the right is $B^1(\rho_{23})$.}
\label{fig:nonjumplocal}
}
\end{figure}
If $\delta^i_j$ and $\delta^{i+1}_j$ are not small approximations of one another, there are two models for the ordered triple $(N_j, \delta_j^i , \delta_j^{i+1})$ depending upon the configuration of the curves:
\begin{itemize}\item $(\DD, \Im, \Re)$ if $\delta^i_j$ is meridional and $\delta^{i+1}_j$ longitudinal;
\item $(\DD, \Re, \Im)$ if $\delta^i_j$ is longitudinal and $\delta^{i+1}_j$ is meridional;
\end{itemize}

$H_2(\DD, \Re \cup \Im)$ is generated by four quadrants $B(\rho_0), B(\rho_1), B(\rho_2), B(\rho_3)$, where $B(\rho_0)$ is the quadrant of $(\DD, \Re \cup \Im)$ containing the point $p_j$, and the others are given in the order they are encountered clockwise from $B(\rho_0)$ (see Figure \ref{fig:reebdomains}). 
\begin{figure}[h]
\centering{
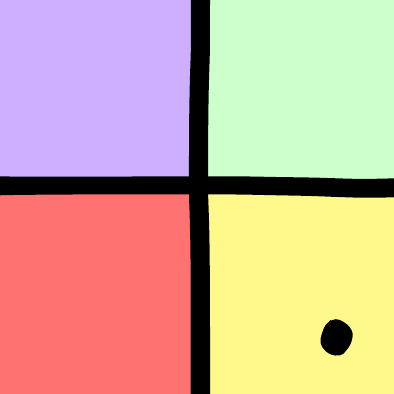
\caption{The four quadrants of $\DD$.}
\label{fig:reebdomains}
}
\end{figure}There is a natural identification $\varphi: H_1(S^1, \ba_j) \rightarrow H_2(\DD, \Re \cup \Im)$, and we denote $B(\rho) := \varphi(\rho)$. We denote by $B_j^i(\rho)$ the corresponding domain in $H_2(N_j, \partial N_j \cup (\delta_j^i \cup \delta_j^{i+1}))$. (Please see Figure \ref{fig:localdomainsjumping} for some examples.)

\begin{figure}[h]
\centering{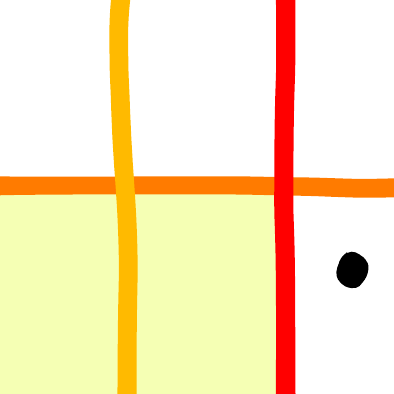\quad\quad \quad 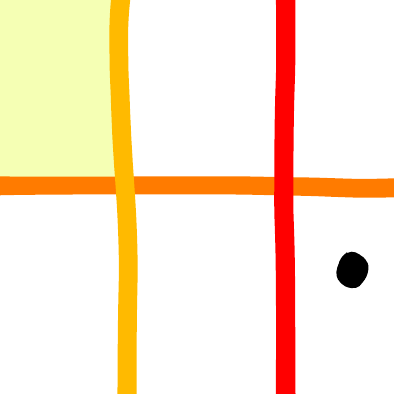
\caption{A series of examples of various domains (shaded green) in the neighbourhood $N_j$ of $e^\infty_j$. On the left is an example of $B^1(\rho_{1})$, and on the right is $B^2(\rho_{2})$. We note how the sum of these domains has an obtuse corner at the intersection of $\gamma^1_j$ and $\gamma^2_j$.}
\label{fig:localdomainsjumping}
}
\end{figure}
We also extend this assignment to the case where $\cD$ is partially splayed and $\bgamma^1 \cap N_j = (\gamma^{1,m}_j \cup \gamma^{1,l}_j) \cap N_j$. Noting that $(N_j, \bgamma^1 \cap N_j)$ may be identified with $(\DD, \Im \cup \Re)$, we may associate to a sequence of Reeb chords $\arrowrho$ the corresponding domain $B(\arrowrho)$ identified with $\varphi(\arrowrho)$, in the homology group $H_2(N_j, \partial N_j \cup (\bgamma^1 \cap N_j))$.

Recall that two sequences of Reeb multi-chords $\arrowrho, \arrowrho'$ are 123-equivalent if their corresponding 123-composed sequences agree. We introduce a related notion for domains:
\begin{defn}We shall say that a sequence of domains
$$S = (D_j^1,B_j^1,D_j^2,B_j^2, \dots,D_j^k, B_j^k , D_j^{k+1}),$$
is a \emph{sequence of Reeb domains} if each $D_j^i$ is equal to $D_j^i(\arrowrho^i)$ for some repeated sequence $\arrowrho^i : =(\rho^i, \rho^i, \dots, \rho^i)$ of a non-jumping chord $\rho^i \in \cT$, and each $B_j^i$ is equal to $B_j^i(\underline{\rho}^i)$ for some jumping chord $\underline{\rho}^i \in \cT$. For such a sequence, denote by $\arrowrho(S)$ the sequence of Reeb chords
$$(\arrowrho^1 \cup \underline{\rho}^1 \cup \arrowrho^2 \cup \dots \cup \underline{\rho}^k \cup \arrowrho^k).$$

If $S'$ is another sequence of Reeb domains, we shall say that $S$ and $S'$ are \emph{composable-equivalent} if the homology classes in $H_2(N_j, \partial N_j \cup \delta^1_j \cup \cdots \cup \delta^{m'}_j)$ formed by summing $S$ and $S'$ are the same.
\end{defn}

As the name suggests, this relation is closely related to the composable-equivalence of corresponding Reeb chords:
\begin{prop}Suppose $S$ and $T$ are sequences of Reeb domains. Then they are composable-equivalent if and only if the sequences $\arrowrho(S)$ and $\arrowrho(T)$ are $123$-equivalent.
\end{prop}
\begin{proof}Suppose that $S$ is a sequence of Reeb domains and consider the corresponding sequence of Reeb chords $\arrowrho(S)$. One can see there is a unique sequence $U(S)$ of Reeb domains that is composable-equivalent to $S$, and such that $\arrowrho(U(S))$ is the $123$-reduction of $\arrowrho(S)$.

Indeed, $\arrowrho(S)$ is not $123$-reduced if and only if there are two adjacent chords $\rho_k, \rho_{k+1} \in \arrowrho(S)$ that are either equal to the torus algebra chords $\rho_1, \rho_{23}$ or $\rho_{12}, \rho_3$. In the first case, if $B^i_j = B^i(\rho_1)$ and $D^{i+1}_j = n_{i+1}B^{i+1}(\rho_{23})$ are the corresponding domains in $S$ then one can see that the sequence $S'$ given by replacing $B^i_j$ by $B^i(\rho_{123}) = B^i(\rho_{1}) + B^i(\rho_{23})$ and $D^{i+1}_j $ by $(n_{i+1}-1)B^{i+1}(\rho_{23})$ is composable-equivalent to $S$, and the corresponding $\arrowrho(S')$ differs from $\arrowrho(S)$ by replacing $\rho_k, \rho_{k+1}$ with $\rho_k\cdot \rho_{k+1}$. Similar holds when $(\rho_k, \rho_{k+1}) = (\rho_{12}, \rho_3)$. Repeating this process for any $123$-composable chords yields the $U(S)$ required.

It follows that $S$ and $T$ are composable-equivalent if and only if $U(S) = U(T)$, which holds if and only if the $123$-reductions of $\arrowrho(S)$ and $\arrowrho(T)$ agree, and so the result follows.
\end{proof}

\begin{prop}\label{constructdomains}Suppose that $B \in \pi_2^{\btheta_+}(\bx, \by)$  is a domain which carries a holomorphic representative, and is $\arrowrho_\boat$-compatible. Then for every $j = 1, \dots, h$, the domain $B_j - B(\arrowrho_{\boat,j})$  may be written uniquely, up to composable-equivalence, as a sequence of Reeb domains.\end{prop}
\begin{proof}
We construct an explicit sequence of Reeb domains $S$ inductively. If there is only one set of $\bgamma$-curves, $B_j$ is an element of the group $H_2(N_j, \partial N_j \cup \delta^1_j)$, which, as discussed, is equal to $n_1 B_j^1(\rho)$ for some non-jumping $\rho \in \cT$.

Consider a small disc neighbourhood $M^1_j$ of $e^{\infty,1}_j$, disjoint from all other $e^{\infty,i}_j$ and $e^{\infty}_j$. The restriction of the domain of $B$ to $M^1_j$ is modelled upon $n_1 B(\rho_1) + n_2 B(\rho_2) + n_3 B(\rho_3)$ for some $n_1, n_2, n_3 \ge 0$. There are three cases:
\begin{itemize}\item If $\gamma^2_j$ is meridional, write $D^1_j : = n_3 B^1_j(\rho_{23})$ and $B^1_j : = B^1_j(\rho_2)$, else
\item if $\gamma^2_j$ is longitudinal and $n_1 > n_2$, write $D^1_j : = n_2B^1_j(\rho_{12})$ and $B^1_j : = B^1_j(\rho_1)$, else
\item if $\gamma^2_j$ is longitudinal and $n_1 \le n_2$, write $D^1_j : n_1B^1_j(\rho_{12})$ and $B^1_j : = B^1_j(\rho_3)$.
\end{itemize}
It follows that the domain $B_j' = B_j - (D^1_j + B^1_j)$ has the same multiplicity upon either side of $\gamma^2_j$. As such, $B_j'$ is an element of $H_2(N^j, \partial N^j \cup (\delta^3_j \cup \dots \cup \delta^k_j))$, and the result follows by induction.
\end{proof}

With this in mind, we may make a definition:
\begin{defn}Let $B \in \pi_2^{\btheta_+}(\bx, \by)$ be a domain which carries a holomorphic representative $u$, and let $S_j$ be a decomposition of each $B_j$ into Reeb domains given by the construction in Proposition \ref{constructdomains}. We denote by $[\arrowrho(B)_j]$ the $123$-equivalence class of Reeb multi-chords defined by $[\arrowrho(B)_j] := [\arrowrho(S_j)]$. We often collect these together, writing $[\arrowrho(B)] = ([\arrowrho(B)_1], \dots, [\arrowrho(B)_h])$.
\end{defn}

By construction, the sequence $[\arrowrho(B)_1], \dots, [\arrowrho(B)_h]$ satisfies that for any set of representatives $\arrowrho(B)_1, \dots, \arrowrho(B)_h$ of each corresponding $123$-equivalence class, the sequence of idempotents $\arrowiota_j : = \iota^1_j, \dots, \iota^{m+1}_j$ has
$$|\{\iota^i_j \in \arrowiota_j: \iota^i_j \ne \iota^{i-1}_j\}| = |\jump(\arrowrho_j)|,$$
as in Definition \ref{reebidemptosplicing}. As such:

\begin{defn}If $B \in \dom(\cD(\varepsilon, \arrowiota))$ is a domain, then there is a well-defined splicing of $\arrowrho(B)_1, \dots, \arrowrho(B)_h$ given by $\vec\sigma(B) : = \vec\sigma(\arrowrho, \arrowiota)$.
\end{defn}

We shall later be interested in domains which look similar between splayed diagrams. A useful construction to bear in mind (cf. \cite{LOT:SSII}[Definition 3.34]) is:
\begin{defn}Fix an $h$-splaying sequence of idempotents $\arrowiota$, and consider an additional tuple of idempotents $\iota = (\iota_1, \dots, \iota_h)$. Let $\arrowiota' = (\iota, \arrowiota)$ and consider the partially splayed diagrams $\cD^\partial(\varepsilon, \arrowiota)$ and $\cD^\partial(\varepsilon, \arrowiota')$: we label the curves in each by $\bgamma^0, \dots, \bgamma^m$ and $\bdelta^0, \dots, \bdelta^{m+1}$ respectively. For every $i > 1$, the curves $\bgamma^i$ and $\bdelta^{i+1}$ are small approximations of one another (with associated approximation region $R^{i}$, and the curves $\bgamma^2$ are small approximations of some subset of the curves $\bgamma_{\bar{e}}^1$ in the east-compactification $\cD^\partial(\varepsilon, \arrowiota)|_{\bar{e}}$ (with associated approximation region $R^1$).

For every curve $\gamma^i_j \in \bgamma^i$, choose a pair of base-points on either side of $\gamma^i_j$ such that the base-points are disjoint from each of the approximation regions $R^i$. For any $$B' \in \pi_2(\bx, \btheta^{1,2}, \dots, \btheta^{m-1,m}, \by) \subset \dom(\cD^\partial(\varepsilon, \arrowiota')),$$ we let $\phi(B')$ denote the unique domain
$$B \in \pi_2(\bx, \btheta^{2,3}, \dots, \btheta^{m-1,m}, \by) \subset \dom(\cD^\partial(\varepsilon, \arrowiota))$$
with the same multiplicities at each of these base-points as $B'$ (here we abuse notation and write the same symbols for generators which are identified by the nearest-point map). We sometimes call $\phi(B')$ the domain obtained from $B'$ by \emph{taking multiplicities away from the approximation region}, and shall say that $B'$ \emph{splays} $B$ if $\phi(B')= B$.
\end{defn}

\section{Moduli spaces of holomorphic curves and polygons}\label{modulispaces}

\subsection{Complex structures on polygons}\label{polygons}

Before we describe the moduli spaces that the results of the paper will be based upon, we need to fix notation for moduli spaces of polygons. Most of this section follows \cite[Chapter 9]{Seidel:Fukaya}.

Let $D_m$ denote a disc, with $n$ punctures upon the boundary (we often view these as marked vertices). Choose a distinguished edge $e^0 \subset \partial D_m$, and label the remaining edges $e^1, \dots, e^{m-1}$ anticlockwise from $e^0$; also label the vertices by $p^{i, i+1}$, so that the vertex $p^{i,i+1}$ lies between $e^i, e^{i+1}$ (here, superscripts are taken modulo $m$).

Let $\conf(D_m)$ denote the moduli space of equivalence classes of positively oriented complex structures upon the disc $D_m$, up to biholomorphism that preserves the labelling of the edges (and vertices). The moduli space $\conf(D_m)$ has a natural compactification by trees of holomorphic polygons which we summarise now.

\begin{defn}An \emph{$m$-leaved tree} is a properly embedded planar tree $T \subset \RR^2$ with $m+1$ semi-infinite edges, one of which is distinguished and called the \emph{root} of $T$, and the others we call the \emph{leaves}. We denote by $V(T)$ the set of vertices of $T$, and $\mbox{Ed}^{\text{int}}$ the set of interior (finite) edges of $T$.

A \emph{flag} in $T$ is an oriented pair of a vertex and adjacent edge, $(V,e)$, oriented so that the edge $e$ begins at $V$. We orient $T$ so that every path from the root to a leaf is positively oriented. To every vertex $v$ of $T$ we can associate a series of flags, given by the pair $(V, e)$ and oriented accordingly. One of these flags is \emph{negative} with respect to $V$, in the sense that the orientation on $e$ induced by $T$ differs from that stemming from the flag (i.e. the edge points `into' $V$), and the others are \emph{positive}. It follows that we can order the flags originating at any given vertex $V$ by saying $f_0(V)$ is the unique negative flag at $V$, and ordering the remainder $f_1(V), \dots, f_{|V| - 1}(V)$ so that they are encountered in order from $f_0(V)$, anticlockwise. We call this the \emph{root-ordering} of the flags.
\end{defn}

Suppose we are given an $m$-leaved tree $T$, and a collection of Riemann surfaces $S = \{S(V)\}_{V\in V(T)}$ with boundary and punctures on the boundary, such that the boundary punctures of $S(V)$ are in one-to-one correspondence with the flags $f(V)$. To each interior edge $e$ in $T$ connecting $V_1, V_2$ there are thus two related boundary punctures: one in $S(V_1)$ corresponding to the flag $(V_1, e)$, and one in $S(V_2)$ corresponding to the flag $(V_2, e)$. There is an operation upon the surface $S$ given by deleting a small open neighbourhood of each of these punctures, and identifying the two new resultant boundary arcs to yield a new surface. Performing this operation iteratively over all interior edges gives a connected (topological) surface, $S(T)$. 

For disks, we require more predictable behaviour near the punctures.
\begin{defn}Let $D$ be a representative for some equivalence class in $\conf(D_n)$. A choice of \emph{strip-like ends} for $D$ is a choice of open neighbourhood $N^{i,i+1}$ of each boundary puncture $p^{i,i+1}$ --- so that $N^{i, i+1} \cap N^{j, j+1} = \emptyset$ for $i \ne j$ --- together with a choice of a biholomorphism between each $N^{i, i+1}$ and the strip $\RR_+ \times [0,1]$ for every $i \ne 0$, and between $N^{0, 1}$ and the strip $\RR_- \times [0,1]$.
\end{defn}

Suppose $T$ is a $d$-leaved tree, and that we are given a collection of Riemann surfaces $S = \{S(V)\}_{V\in V(T)}$, where each $S(V)$ represents some equivalence class in $\conf(D_{|V|})$ and is endowed with a choice of strip-like ends (i.e. a tree of holomorphic discs with strip-like ends). We shall say that a map $\rho: \mbox{Ed}^{\text{int}}(T) \rightarrow (0, \infty)$ is a \emph{length assignment} for $T$ --- thus giving $T$ the structure of a metric graph. We often omit $\rho$ from the notation.

We refine the construction of $\cup_TS$. There is a natural correspondence between the flags of each vertex $V$ and the punctures of $D_{|V|}$, between the puncture $p^{i, i+1} \in D_{|V|}$ and the flag $f_i(V)$. When we delete neighbourhoods of punctures corresponding with a negative flag $(V_1, e)$ and positive flag $(V_2, e)$, we take these to be the open neighbourhoods identified by the choices of strip-like ends with $(-\infty, -\log(\rho(e)) \times [0,1] \subset \RR_- \times [0,1]$ and $(\log(\rho(e), \infty) \times [0,1] \subset \RR_+ \times [0,1]$ respectively.

This construction yields an element $D_T$ in some equivalence class of $\conf(D_m)$ --- an $m$-pointed disc with Riemann surface structure, with vertices labelled $p^{0,1}, \dots, p^{m-1, 0}$ starting from the one corresponding with the root of $T$ (and edges labelled accordingly). Moreover, there is a natural choice of strip-like ends upon $D_T$ induced by those upon the punctures corresponding with the semi-infinite flags of $T$.

Allowing edge lengths to degenerate and take the value $0$ allows us to describe the compactification of the moduli space $\conf(D_m)$. If a given edge length is zero, we omit this edge from the cutting and pasting procedure described above --- effectively cutting the edge into two semi-infinite edges, and thus the tree into more components, upon which we perform the gluing process individually. The resulting collection of discs is described by a tree $\bar{T}$, which is given by collapsing all internal edges of nonzero length, and assigning to each vertex $\bar{v}$ of $\bar{T}$ a $|\bar{v}|$-pointed disc in $\conf(D_{|V|})$.

There is a canonical way to label the edges of $\bar{T}$ in a way compatible with those of $T$. We first label the semi-infinite flags of $\bar{T}$ by $(0,1), (1,2), \dots, (m,0)$ anti-clockwise from the root of $\bar{T}$. The task now is to label the interior flags of $\bar{T}$ by labels compatible with these: in such a way that the edges of the corresponding disks $D_V$ can be labelled by $0, 1, \dots, m$, so that an edge between punctures $(i,k)$ and $(k,j)$ is labelled by $k$. The way to do this is to first label the vertices which just have semi-infinite flags and one in-coming flag --- in this case, the semi-infinite flags are already labelled anticlockwise by some contiguous subsequence $(j, j+1), (j+1, j+2), \dots, (j+k-1,j+k)$, so the root flag is forced to be labelled by $(j+k, j)$. We can then delete this vertex from $\bar{T}$ and replace the root flag with a semi-infinite flag labelled by $(j+k, j)$, and repeat the process until all edges of $\bar{T}$ are labelled.

\begin{defn}Suppose that $\cD$ is a diagram with $m$ sets of attaching curves $\bgamma^1, \dots, \bgamma^m$, and let $T$ be a (possibly degenerate) $m$-leaved metric tree. Then a \emph{generator assignment} for $T$ is an assignment of a generator $\bx^{i,j} \in \cG(\bgamma^i, \bgamma^j)$ to each edge of $\bar{T}$ labelled by $(i,j)$. If $\vec\bx$ is a generator assignment for $T$, we refer to the sequence of generators corresponding to the edges adjacent to a vertex $V \in \bar{T}$, ordered by root-ordering, by $\vec\bx(V)$. Note that a generator assignment for $T$ has a well-defined overall sequence of generators $\vec\bx(T)\in \cG(\bgamma^1, \dots, \bgamma^m)$, given by the set of generators labelling the exterior edges of $T$.

Given a generator assignment $\vec\bx$, a \emph{family of compatible domains} is a set $\{B_V\}_{V \in V(\bar{T})}$ of domains such that that $B_V$ is a member of the group $\pi_2(\vec\bx(V))$. The juxtaposition $B(T) = \sum_{V \in V(\bar{T})} B_V$ is, accordingly, a member of the group $\pi_2(\vec\bx(T))$.
\end{defn}

We would like to work with bona fide representatives of each equivalence class for $\conf(D_m)$, with strip-like ends. Once these are chosen, the compactification of the moduli space $\conf(D_m)$ is then described by the set $\cup_{T, \rho} D_T$, indexed over all $d$-leaved trees and all (possibly degenerate) length assignments $\rho$. As edge lengths in this space degenerate to zero, the surface $D_T$ degenerates to a disjoint union of surfaces $\{D_{V}\}_{V \in V(\bar{T})}$. 

To make this make sense, one must choose representatives $D$ and choices of strip-like ends for every $d \le m$ and each equivalence class of $\conf(D_d)$, so that at these boundary strata the strip-like ends chosen agree with those induced by those chosen upon $D_d$. This is possible: see, for instance, \cite[Section 9g]{Seidel:Fukaya}.

From this point, we fix such a set of choices of representatives for equivalence classes with strip-like ends such that this works (i.e. we have consistency on the boundary strata), and abuse notation by identifying elements of $\conf(D_m)$ with these representatives. We also fix a choice of symplectic forms $\{\omega_j\}_{j \in \conf(D_m)}$ where $\omega_j$ is compatible with $j$, so that if $(j_n)_{n = 1}^\infty \in \conf(D_m)$ is a sequence of complex structures converging to the collection $\{j_{\bar{v}} \}_{\bar{v} \in V(\bar{T})} \in \partial \overline{\conf(D_m)}$, then $\omega_{j_n} \rightarrow \{\omega_{j_{\bar{v}}}\}_{\bar{v} \in V(\bar{T})}$.
  
We will also have need to fix identifications $\eta^i$ of a small neighbourhood $N(e^i)$ of each edge $e^i$ in the polygons $(D_m, j) \in \conf(D_m)$ with $\RR \times [0, \varepsilon)$, so that on each of these $N(e^i)$ the symplectic form $\omega_j$ agrees with the pullback of the standard symplectic form on $\RR^2$. To make this make sense, we require that these identifications are continuous, consistent with the choice of strip-like ends, and consistent across boundary strata of $\conf(D_m)$ (see \cite[Section 3.3]{LOT:SSII} and \cite[Section 10e]{Seidel:Fukaya}, for instance).

We will later need to single out some vertices of our trees:

\begin{defn}\label{specialvertices}Let $T$ be a tree. Then the set of \emph{comb vertices} of $\bar{T}$ is the set of vertices of $\bar{T}$ visited by the path in $\bar{T}$ connecting the exterior edge labelled by $(0,1)$ (i.e. the root of $\bar{T}$) with the exterior edge labelled by $(1,2)$. The set of \emph{spinal vertices} of $\bar{T}$ is the set of vertices of $\bar{T}$ which have a flag $f_V$ labelled by $(i,j)$ where one of $i$ or $j$ is equal to $0$. The set of \emph{easterly vertices} of $\bar{T}$ is the set of non-spinal vertices of $\bar{T}$.

If $V$ is spinal, we associate to it its \emph{storied tree} $T' \subset T$, given by the maximal connected subtree of $T$ which contains $V$ and no other spinal vertices (so, it contains $V$ and a series of easterly vertices).
\end{defn}

\subsection{Sources}
In this section, we define what shall be the sources of the various types of holomorphic curves considered in the remainder of this chapter. We first fix positive integers $m$, $h$, and a disc $D_m\in \conf(D_m)$.
\begin{defn}A \emph{topological $(h,m)$-source} is a  surface $S$, with boundary and punctures $\punc(S)$ together with a labelling $\lambda: \punc(S) \rightarrow \punc(D_m) \cup \cT$ of the boundary punctures of $S$ by either a puncture in $\partial D_m$ or a Reeb chord in $\cT$. We will write $E(S)$ for the subset of `east' punctures $\lambda^{-1}(\cT)$, and also require that $E(S)$ is partitioned into $h$ subsets $E_1(S), \dots, E_h(S)$.

An \emph{$(h,m)$-source} is a topological $(h,m)$-source together with a choice of complex structure $j_S$ which gives $S$ the structure of a smooth (not nodal) Riemann surface.
\end{defn}

It will be useful to have some more adjectives to describe various types of source.

\begin{defn}We give names to the following specific varieties of $(h,m)$ source:
\begin{itemize}
\item A (topological) \emph{bordered} source is a (topological)  $(h,2)$-source.

\item A (topological) \emph{polygonal}  $m$-source is a (topological) $(0,m)$-source such that $E(S) = \emptyset$.

\item A (topological)  \emph{tooth} source is a (topological) $(h,m)$-source together with a labelling of the punctures $E(S)$ by either `$e$' or `$w$'. We will usually denote these by $T$, and shall write $E_j(S)$ (resp. $W_j(S)$) for the set of punctures labelled by `$e$' (resp. `$w$') in part $E_j(S)$.

\item In the degenerate case where $m = 1$, i.e. $D_m = \HH$ is the upper-half plane, with one boundary puncture at $\infty$, we call a (topological) $(h, 1)$ source an \emph{anchor source}.
\end{itemize}
\end{defn}

Later, we will need to fix the behaviour of holomorphic curves near a certain subset of non-east punctures. The relevant kind of source is:

\begin{defn}\label{splayedsource}A (topological) \emph{splayed} $(h,m)$-source is a (topological) $(h,m)$ source together with a distinguished set $\rho_{123}(S) \subset \punc(S) - E(S)$, and a partition $Q$ of $\rho_{123}(S)$ into three sets, $Q = (Q_{123}, Q_{1,23}, Q_{12,3})$.
\end{defn}

If we are given a metric tree $T$, then we can associate to every vertex $V$ of the corresponding tree $\bar{T}$ a (topological) $(h,|V|)$ source $S_V$. If, for each $V$ and $V'$ connected by an edge of $\bar{T}$, we are given a bijection between the punctures of $S_{V}$ and $S_{V'}$ corresponding with this edge, we can associate to $T$ a corresponding (topological) $(h, m)$ source $S_T$, where $m$ is the number of non-interior edges of $\bar{T}$. This is formed by deleting small strip neighbourhoods of punctures of $S$ corresponding with an edge of $T$, and gluing the resultant ends together according to the chosen bijection. There are natural decorations of $S_T$ induced by those upon $S_V$: each east puncture in a source $S_V$ is preserved by the gluing process, and thus has a corresponding puncture in $S_T$: we label the resultant punctures by the corresponding element of $\cT$, and partition them according to $E_j(S_T) = \cup_{V \in V(T)}E_j(S_V)$. If some sources $S_V$ are also endowed with the extra data of a distinguished set $\rho_{123}(S_V)$ and partition $Q_V$ as in Definition \ref{splayedsource}, the resultant source has a natural distinguished set $\rho_{123}(S_T)$ given by the union of all of the $\rho_{123}(S_V)$ after the gluing procedure --- where we omit any elements of $\rho_{123}(S_V)$ which are glued along, and regard a (topological) $(h,m)$ source as a splayed source with empty distinguished set.

If $\bar{T}$ only has two vertices $V_1, V_2$, then we sometimes denote the corresponding source by $S_{V_1} \natural S_{V_2}$. A similar construction can also be made for tooth sources: if $S$ is a bordered source and $T$ is a tooth source, with a bijection between $E(S)$ and $W(T)$, we can form a glued source $S \natural T$ given by deleting small neighbourhoods of $E(S)$ and $W(T)$ and identifying the resultant `stubs' via the bijection.

\subsection{Moduli spaces of holomorphic curves}

We will now define a series of moduli spaces of holomorphic curves. Suppose that $\cD = (\Sigma; \bgamma; \bb; \bp)$ is a diagram, where $\bgamma = \{\bgamma^{1} , \dots, \bgamma^{m}\}$.  We shall be concerned with holomorphic maps from an $m$-source into the four-manifold $X_{\cD,m}:= \Sigma \times D_{m}$. We assume that $\Sigma$ is endowed with a symplectic form $\omega_\Sigma$, and choose a complex structure $j_\Sigma$ upon $\Sigma$ which is compatible with $\omega_\Sigma$.

We also assume that $j_\Sigma$ is cylindrical near the boundary, in that we fix small neighbourhoods $N_j \subset \Sigma_{\bar{e}}$ of each $e^\infty_j$, together with a symplectic identification of each $U_j = N_j - e^\infty_j$ with $S^1 \times (0,1)$. Under this identification we assume that the complex structures $(j_\Sigma)|_{U_j}$  and arcs $\bgamma \cap U_j$ are also invariant in the $(0,1)$-direction. It follows from this assumption that the complex structure $j_\Sigma$ can be extended to a complex structure $j_{\bar{e}}$ upon $\Sigma_{\bar{e}}$.

We denote by
\begin{align*}\pi_\Sigma: \Sigma \times D_m &\rightarrow \Sigma, \mbox{ and }\\
\pi_D: \Sigma \times D_m &\rightarrow D_m\end{align*}
the relevant projection maps.

\begin{defn}An \emph{admissible} family of almost-complex structures for $\cD$ is a collection $\bJ = \cup_{m \ge 1} \bJ_m$ of families of almost-complex structures $\bJ_m =\{J_j\}_{j \in \conf(D_m)}$ upon $X_{\cD, m}$ satisfying:
\begin{enumerate}\item The projection map $\pi_D: \Sigma \times D_m \rightarrow D_m$ is $(J_j, j)$-holomorphic for each $j$.
\item The fibres of the maps $\pi_D$ are $J_j$-holomorphic for every $j$.
\item Each $J_j$ is adjusted to the split symplectic form $\omega_\Sigma \oplus \omega_j$ on $\Sigma \times D_m$.
\item Near the punctures $e^\infty_1 \times D_m, \dots, e^\infty_h \times D_m$, the almost complex structures $J_j$ are split.
\end{enumerate}
\end{defn}

For fixed $m$, we shall study moduli spaces of maps into $X_{\cD,m}$ which are $J \in \bJ_m$-holomorphic. We will require some form of consistency for the family $\bJ$, so that for any family of elements of $\conf(D_m)$ converging to some point in $\partial \overline{\conf(D_m)}$, the corresponding family of almost-complex structures upon $X_{\cD, m}$ converges to the correct tree of almost-complex structures. Namely:

\begin{defn}Suppose $\bJ$ is a family of admissible almost-complex structures. Recall that for any $m$ and sequence of $m$-leaved metric trees $(T_n)_{n \ge 0}$ converging to some degenerate tree $T$, the chosen family of equivalence classes $\conf(D_m)$ satisfies that the corresponding almost-complex structures $j_n : = j_{T_n} \in \conf(D_m)$ converge to a tree of complex structures $\{j_V\}_{V \in V(T)}$, where each $j_V \in \conf(D_{|V|})$.

We say that $\bJ$ is \emph{consistent} if the family $J_{j_n}$ converges to the corresponding tree $\{J_{j_V}\}_{V \in V(T)}$ --- where each $J_{j_V}$ is an almost-complex structure on $X_{\cD, |V|}$.
\end{defn}

Given a diagram $\cD$, there are corresponding Lagrangian `cylinders' inside the four-manifold $X_\cD$. At their most basic, these are given by the sets of cylinders and half-cylinders $C^{i}:=\bgamma^i \times e_i \subset \Sigma \times e_i$. In Section \ref{sub:anchorsperturbations}, we will perturb these cylinders slightly, replacing $C^i$ with cylinders $C^i(\varepsilon)$ associated to some Hamiltonian perturbation of $\Sigma$. In this case, the corresponding definition is to be made by replacing $C^i$ with $C^i(\varepsilon)$ below. 

With these definitions in mind, we specify the types of holomorphic curves we shall be interested in.
\begin{defn}Let $\cD = (\Sigma; \bgamma; \bb; \bp)$ be a diagram with attaching curves $\bgamma = \bgamma^1, \dots, \bgamma^m$, and write $\bgamma^i = \bgamma^{i, a} \cup \bgamma^{i, c}$ for the partition of each $\bgamma^i$ into arcs and closed curves, respectively. Suppose also that we are given a partition of the points $\bp = \bw \cup \bz$, where either part may be empty.

Fix a consistent, admissible family $\bJ$ of almost complex structures for $\cD$, and a source $S$, endowed with a complex structure $j_S$. We consider pairs $(u, j)$, where $j \in \conf(D_m)$ and $$u: (S, \partial S) \rightarrow (X_\cD, C^1 \cup \cdots \cup C^n),$$
is a map that satisfies the following conditions:

\begin{Mlist}\item $u$ is $(j_S,J_j)$-holomorphic.
\item  $u: S \rightarrow \Sigma \times D_n$ is proper.
\item $u$ has finite energy
\item The image of $u$ is disjoint from the sets $\bb \times D_n$ and $\bz \times D_n$.
\item $u$ extends to a proper map $u_{\bar{e}}: S_{\bar{e}} \rightarrow \Sigma_{\bar{e}} \times D_n$. (This is only meaningful when $\Sigma_\cD$ has boundary.)
\item $\pi_{\DD} \circ u_{\bar{e}}$ is a $g$-fold branched cover.
\item $u$ is an embedding.
\item At each puncture $q$ of $S$ labelled by $p^{i, i+1}$, $\lim_{z \rightarrow q}(\pi_D\circ u) (z) = p^{i,i+1}$.
\end{Mlist}

When $\Sigma$ has boundary, we require three further conditions:

\begin{BOR}\item At each puncture $q\in E(S)_j$, $\pi_\Sigma \circ u$ is asymptotic to the Reeb chord labelling $q$ upon the boundary component $\partial\Sigma_j$.

\item For every $i = 1, \dots, n$ and $j = 1, \dots, g+h-1$, and every $t \in e^i$, we require that $u^{-1}(\gamma^{i,c}_j \times \{t\})$ consists of a single point.

\item (Strong boundary monotonicity) For every $i = 1, \dots, n$ and $j = 1, \dots, 2k_i$, and every $t \in e^i$, we require that $u^{-1}(\gamma^{i,a}_j \times \{t\})$ consists of at most one point.
\end{BOR}
\end{defn}

From these conditions, it is clear that near each puncture $q$ of $S$ which is labelled by $p^{i,i+1}$, the map $u$ is asymptotic to $$I^{i, i+1}(q): = x^{i, i+1}(q)\times p^{i,i+1},$$ where $x^{i,i+1}(q)$ is a point in $H^i(\gamma^i, \infty) \cap H^{i+1}(\gamma^{i+1}, -\infty)$.
We collect these into sets of points $\bx^{i,i+1} : = \{x^{i,i+1}(q) : \lambda(q) = p_{i,i+1}\}$, for each $i$, and collect the points $I^{i, i+1}(q)$ into corresponding sets $I^{i, i+1}(\bx^{i, i+1})$.

We will be concerned with a few specific instances of this construction, which we spend the next sections giving more thorough names to.

\subsection{For bordered diagrams}\label{borderedmoduli}This section briefly recounts \cite[Section 5.2]{LOT}.

Suppose that $\cD$ is a toroidal bordered diagram, $\cD = (\Sigma; \balpha, \bbeta; \bb)$. We relabel $C^1 := C^\balpha, C^2 := C^\bbeta$ for clarity's sake, and consider the case when $S$ is a bordered source.

Given generators $\bx$ and $\by$, a choice of almost-complex structure $J$ upon $X_{\cD,2}$, a fixed decorated source $S$ and a homology class $B \in \pi_2(\bx, \by)$, we shall write $\freemod^B(\bx, \by; S)$ for the moduli space of curves $u$ with source $S$ which satisfy (M-1)--(M-8) and (BOR-1)--(BOR-3) with respect to $J$, in the homology class $B$, with asymptotics $(\bx, \by)$. Sometimes, to make explicit which almost-complex structure we are using, we will write $\freemod^B(\bx, \by; S; J)$.

If $u \in \freemod^B(\bx, \by; S)$ is a holomorphic curve, then there is a correpsonding \emph{evaluation map} for each puncture $q$ of $S$ labelled by a Reeb chord, $\ev_q(u) : = t \circ u_{\bar{e}}(q)$. Sometimes, if $u$ is implied, we refer to $\ev_q(u)$ as the \emph{height} of $q$. There are evaluation maps $$\ev_j = \prod_{q \in E_j(S)} \ev_q: \freemod^B(\bx, \by; S; J) \rightarrow \RR^{|E_j(S)|},$$
which fit together into a single map $\ev: \freemod^B(\bx, \by; S; J) \rightarrow \RR^{|E(S)|}$ given by $\ev:= \ev_1 \times \cdots \times \ev_h$.

If we fix a partition $P_j$ of $E_j(S)$, there is a corresponding subspace $\Delta_P\subset \RR^{|E_j(S)|}$ given by $\{x_p = x_q : P^i_j \in P_j, p, q \in P^i_j\}$. This allows us to further split the moduli space above by requiring that the heights of punctures in each part coincide, namely:

\begin{defn}Let $\bx, \by$ be generators, $B \in \pi_2(\bx,\by)$, $S$ be a decorated source, and $P: = (P_1, \dots, P_h)$ partitions of the set $E_1(S), \dots, E_h(S)$. Then we write
$$\freemod^B(\bx, \by;S;P_1, \dots, P_h) : = \ev^{-1}(\Delta_{P_1} \times \cdots \times \Delta_{P_h}) \subset \freemod^B(\bx, \by; S).$$
\end{defn}

Again, we sometimes write $\freemod^B(\bx, \by;S;P_1, \dots, P_h;J)$ to make explicit which almost-complex structure we are using, and often abbreviate to $\freemod^B(\bx, \by;S;P)$ for brevity.

A further refinement is gotten by endowing each partition $P_j$ with an ordering of its parts, $\arrowP_j$. We can then write $\freemod^B(\bx, \by;S;\arrowP)$ for the subset of the set of curves $u \in \freemod^B(\bx, \by;S;P_1, \dots, P_h)$ for which the ordering of each $\arrowP_j$ agrees with that given by dictating $P^i_j < P^{i'}_j$ if $\ev_{q^i}(u) < \ev_{q^{i'}}(u)$ for all (any) $q^i \in P_j^i, q^{i'} \in P_j^{i'}$. Often we abbreviate $\arrowP : = (\arrowP_1, \dots, \arrowP_h)$.

If we are given a set of sequences of  Reeb chords $\arrowrho_{\boat} = (\arrowrho_{\boat,1}, \dots, \arrowrho_{\boat, 1})$, then we shall say the pair $(\arrowP, S)$ is \emph{compatible} with $\arrowrho_{\boat}$ if the sequence of sets of Reeb chords gotten from $\arrowP_j$ by replacing each part $P^i_j$ with the set of Reeb chords labelling the punctures $q \in P^i_j \subset E_j(S)$ is equal to $\rho_{\boat, j}^i$, and we shall often write $\freemod^B(\bx, \by;S;\arrowrho_{\boat})$ for the moduli space $\freemod^B(\bx, \by;S;\arrowP)$, where $(\arrowP, S)$ are compatible with $\arrowrho_{\boat}$. If $(\arrowP, S)$ are $\arrowrho_\boat$-compatible, for any $\rho \in \arrowrho_{\boat, j} \in \arrowrho_\boat$ we shall write $q(\rho)$ for the corresponding puncture of $S$.

\begin{rmk}
The moduli space $\freemod^B(\bx, \by;S;\arrowrho_{\boat})$ has an action by $\RR$, induced by translating in the $\RR$-direction of $D_2 = [0,1] \times \RR$. Unless $S$ is a disjoint union of $g$ strips, each with two boundary punctures, and the map $u$ has trivial homology class, this translation is free --- and we say $u$ is \emph{stable}. We will have cause to consider the quotient moduli space $\freemod^B(\bx, \by;S;\arrowrho_{\boat})/ \RR$ which we term $\tildemod^B(\bx, \by;S;\arrowP)$.

In \cite{LOT}, what we call $\freemod$ is called $\tildemod$, and the moduli space we call here by $\tildemod$ is called simply $\cM$. We reserve the latter for certain moduli spaces of bordered polygon with constraints on behaviour near certain corners: however, in the case where the bordered polygon is in fact a bi-gon, this does not present a further constraint and thus $\cM = \tildemod$, agreeing with the previous notation in the literature.
\end{rmk}

In \cite{LOT}, a number of properties of these moduli spaces are explored. We presently re-state those that shall be of use to us.

\begin{prop}[{\cite[Proposition 5.6]{LOT}}] \label{borderedtransversality} For each $\bx, \by, S$, there is a dense set of admissible $J$ such the moduli spaces $\freemod^B(\bx, \by;S;J)$ are transversally cut out by the $\delbar$-equations. Moreover, for any countable set $\{M_i\}$ of submanifolds of $\RR^{|E(S)|}$, there is a dense set of admissible $J$ which satisfy that $\ev: \freemod^B(\bx, \by;S; J) \rightarrow \RR^{|E(S)|}$ is transverse to all of the $M_i$: i.e. there is a dense set of admissible $J$ such that the moduli spaces $\freemod^B(\bx, \by;S;\arrowrho_{\boat}; J)$ are also transversely cut out.

In particular, the set of \emph{sufficiently generic} admissible almost complex structures --- those for which the moduli spaces $\freemod^B(\bx, \by;S;\arrowrho_{\boat};J)$ are transversally cut out for \emph{all} choices of $\bx, \by, S, \arrowrho_{\boat}$ --- is nonempty.
\end{prop}

By the implicit function theorem, it follows that for sufficiently generic $J$ the moduli space of $u \in \freemod^B(\bx, \by;S; J; \arrowrho_{\boat})$ satisfying $\ind(D\delbar (u)) = k$ is a $k$-dimensional manifold. The index at any given $u$ satisfies a formula in terms of the topology of $S$ and $B$, contingent on the \emph{Euler measure} of a domain:
\begin{defn}\label{eulermeasure}Let $B \in \pi_2(\bx^{0,1}, \dots, \bx^{k,0})$ be a domain. Then its \emph{Euler measure}, $e(B)$, is given by
$$e(B) : = \chi(B) - \frac{\mbox{ac}(B)}{4} + \frac{\mbox{ob}(B)}{4},$$
where $\chi(B)$ denotes the Euler characteristic, and $\mbox{ac}(B)$ and $\mbox{ob}(B)$ are the number of acute and obtuse corners in $B$ respectively.
\end{defn}

The relevant index formula is:

\begin{prop}[{\cite[Proposition 5.8]{LOT}}] \label{indbord} The expected dimension of $\freemod^B(\bx, \by;S; \arrowrho_{\boat};J)$ is given by
$$\ind(B, S, P) = g - \chi(S) + 2 e(B) + |\arrowrho_{\boat, 1}| + \cdots + |\arrowrho_{\boat, h}| ,$$
where $|\arrowrho_{\boat,j}|$ is the number of Reeb chords in $\arrowrho_{\boat, j}$, and $e(B)$ is the Euler measure of $B$.
\end{prop}

It follows from this that for fixed $B$ and $S$, the dimension of the moduli space $\freemod^B(\bx, \by;S; \arrowrho_{\boat};J)$ is well-defined. There is a further interpretation of this formula solely in terms of data about the domain $[u]$, though we do not pursue this here; the important related fact for us will be:
\begin{prop}\cite[Proposition 5.69]{LOT} \label{embeddedeuler}Whether or not a holomorphic curve $u \in \freemod(\bx, \by; S; \arrowrho_{\boat};J)$ is embedded is determined solely by the Euler characteristic of the source $S$: the curve $u$ is embedded if and only if $\chi(S) = \chi_\emb(S)$ for some well-defined quantity $\chi_\emb(S)$.
\end{prop}

Though the moduli spaces $\tildemod^B(\bx, \by;S ; \arrowrho_{\boat};J)$ form manifolds, they are  not compact except when the moduli space $\tildemod^B(\bx, \by;S; \arrowrho_{\boat};J)$ is zero-dimensional. They admit certain compactifications in terms of \emph{holomorphic combs}, which we consider in Section \ref{combs}.

\subsubsection{Splicings and moduli spaces}\label{splicingbordered}
Suppose that $S$ is a bordered source and $\arrowP$ is a sequence of partitions $\arrowP_1, \dots, \arrowP_h$, such that $S, \arrowP$ is compatible with a sequence of sets of Reeb chords $\arrowrho = (\arrowrho_1, \dots, \arrowrho_h)$.

The definitions of the moduli spaces  $\freemod^B(\bx, \by;S;P)$ and $\freemod^B(\bx, \by;S;\arrowP)$ stipulate relations on the relative heights of punctures from each set $E_j(S)$, as compatible with the sequence $\arrowrho$. More generally, we could consider constraints upon relative heights of punctures which belong to $E_j(S)$ and $E_{j'}(S)$ for $j \ne j'$.

Suppose we are given an $m$-splicing $\vec\sigma = \vec\sigma_1, \dots, \vec\sigma_h$ of $\arrowrho$, and that $(S, \arrowP)$ are $\arrowrho$-compatible. There are $m$ corresponding sets
$$Q^{i} : = \{q(\rho) : \rho \in \sigma_1^i \cup \cdots \cup \sigma_h^i\}$$
of punctures. In turn, these yield corresponding submanifolds of $\RR^{|E(S)|}$ specified by $\Delta_{\sigma}^i: = \{x_p = x_q:  p, q \in Q^i\}.$

This allows us to define moduli spaces $\tildemod^B(\bx, \by;S;P, \sigma)$ and $\tildemod^B(\bx, \by;S;\arrowP, \sigma)$, given by $\ev^{-1}(\Delta_{\sigma}^1, \dots, \Delta_{\sigma}^m)$ for the evaluation maps from $\tildemod^B(\bx, \by;S;P)$ and $\tildemod^B(\bx, \by;S;\arrowP)$ to $\RR^{|E(S)|}$ respectively. These are transversally cut out for sufficiently generic $J$ by Proposition \ref{borderedtransversality}.

The corresponding index formula is given by
$$\ind(B, S, P, \sigma) = g - \chi(S) + 2 e(B) + |\arrowrho_{\boat, 1}| + \cdots + |\arrowrho_{\boat, h}| - \col(\sigma),$$
and is proved in a similar manner as \cite[Proposition 5.8]{LOT} --- each height constraint coming from $\sigma$ drops the expected dimension by one.

These moduli spaces only depend upon the unordered partitions in $\sigma$. The extra information of the order of the partitions in $\vec\sigma$ allows us to define even more restricted moduli spaces $\tildemod^B(\bx, \by;S;\arrowP, \vec\sigma)$ as the set of $u \in \tildemod^B(\bx, \by;S;\arrowP, \sigma)$ such that the ordering of the parts in $\sigma$ according to $\vec\sigma$ agrees with the ordering of the parts in $\sigma$ by dictating that $Q^i < Q^{i'}$ if $\ev_{q_i}(u) < \ev_{q_{i'}}(u)$ for every (any) $q^i \in Q^i$ and $q^{i'} \in Q^{i'}$.

\subsection{Multipointed Heegaard multidiagrams}\label{polygonmodulispaces} We now turn our attention to the relevant moduli spaces when $\cD$ is a multipointed Heegaard $m$-diagram, $\cD = (\Sigma, \bgamma^0, \dots, \bgamma^{m-1}, \bp)$, where $\bp$ splits as $\bw \cup \bz$.

Fix a collection of admissible compatible almost-complex structures $\bJ$. As $\Sigma$ is closed, the conditions (BOR-1)--(BOR-3) do not apply, and we are concerned with pairs $(u, j)$ where $j \in \conf(D_m)$ and $u$ is a $(j_S, J_j)$-holomorphic curve satisfying (M-1)--(M-8) disjoint from the set $\bz \times D_m$. 

The conditions imply that $\bx^{i, i+1}$ is a generator for the pair $(\bgamma^i, \bgamma^{i+1})$. In particular, the sequence $\vec\bx = (\bx^{0, 1}, \dots, \bx^{m-1, 0})$ constitutes a sequence of generators for $\cD$.

For $m > 2$, we collect the holomorphic curves $(u, j)$ with source $S$ and given asymptotics $\vec\bx$ in homology class $B \in \pi_2(\vec\bx)$  satisfying (M-1) -- (M-8) into moduli spaces $\tildemod^B(\vec\bx; S)$. If $m = 2$, as in the previous section the relevant moduli space has a free $\RR$-action, and so we denote by $\cM^B(\vec\bx; S)$ the result of taking the quotient by this. We call an element of one of these moduli spaces a \emph{holomorphic $m$-gon}.
As in the previous section, we sometimes include the family $\bJ$ in the notation to draw attention to it: writing things like $\cM^B(\vec\bx; S; \bJ)$, for instance.

For any $S \subset \{0, \dots, m-1\}$ with $|S| = l$, we can apply the above construction to $\cD|_S$ to define corresponding moduli spaces $\cM^B(\bx^{i_0,i_1}, \dots, \bx^{i_{l-1}, i_0}; S; J)$.

The relevant transversality result for holomorphic $m$-gons is:

\begin{prop}[{\cite[Proposition 3.8]{LOT:SSII}}] There is an admissible, compatible family of almost-complex structures $\bJ$ for $\cD$, such that for every source, every $l$-subdiagram of $\cD$, every sequence of generators for such a subdiagram, and every almost-complex structure $J \in \bJ_l$, the corresponding moduli spaces are all transversally cut out.
\end{prop}

There is a corresponding index formula:

\begin{prop}\label{indpoly}Suppose that $(u, j) \in \cM^B(\bx^{0, 1}, \dots, \bx^{k-1, k}, \bx^{k, 0}; S)$ is a holomorphic polygon. Then the index of the $D\delbar$ operator at $(u, J)$ is given by
$$\ind(B, S) = \frac{3-k}{2} g - \chi(S) + 2e(B).$$
\end{prop}

Oncemore, it follows from this and the implicit function theorem that these moduli spaces are manifolds of dimension $\ind(u) + k - 2$ (the $k-2$ stems from the space $\conf(D_k)$ having moduli), though in general non-compact.

\subsection{Bordered polygons}\label{borderedpolygonmodulispaces}
In this section, we construct relevant moduli spaces of polygons for bordered multi-diagrams, generalising both of the previous two sections. For our purposes, it will be enough to consider bordered polygons that only have east asymptotics during the first edge of $D_n$, which we have called $e^1$ above.

We fix a toroidal bordered Heegaard $m$-diagram $\cD$. Supposing that $\bJ$ is a collection of consistent, admissible almost-complex structures, we again consider pairs $(u, j)$ of $J\in\conf(D_m)$ and $(j_s, J_j)$-holomorphic curve $u$ with source $S$ satisfying (M-1)--(M-8) and (BOR-1)--(BOR-3), mapping into $$(\Sigma \times D_m, C^0 \cup \cdots \cup C^{m-1}).$$

The conditions imply that the sequence of points $\vec\bx = (\bx^{0,1}, \dots, \bx^{m-1,0})$ form a sequence of generators for $\cD$. As per usual, such $(u,j)$ have a homology class $B$ in the group $\pi_2(\vec\bx)$, and so we collect these maps into moduli spaces $\freemod^B(\vec\bx; S).$

Once again, if we wish to draw attention to the family $\bJ$ in question, we include it in the notation: e.g. $\freemod^B(\vec\bx; S; \bJ)$. We often refer to a holomorphic curve $(u, j)$ by solely $u$, and hope that this does not cause undue confusion.

As in Section \ref{borderedmoduli}, there is an evaluation map at each east puncture $q$ of $S$, $\ev_q(u): \freemod^B(\vec\bx; S) \rightarrow e^1 \subset \partial D_m$. These fit together to form evaluation maps $$\ev_j: \freemod^B(\vec\bx; S) \rightarrow (e^1)^{|E_j(S)|}$$ for each $j = 1, \dots, h$, which in turn fit together to give a map
$$\ev: \freemod^B(\vec\bx; S) \rightarrow (e^1)^{|E(S)|}.$$
As such, for any partitions $P : = (P_1, \dots, P_h)$ of $E_1(S), \dots, E_h(S)$ there is once more a corresponding moduli space
$$\freemod^B(\vec\bx; S; P) : = \ev^{-1}(\triangle_{P_1} \times \cdots \times \triangle_{P_h}) \subset \freemod^B(\vec\bx; S).$$
Furthermore, ordering the parts of the $P_j$ as $\arrowP_j = (P_j^1, \dots, P_j^{k_j})$ yields a corresponding moduli space
$$\freemod^B(\vec\bx; S; \arrowP),$$
where we further require that the orderings of the punctures induced by the evaluation maps $\ev_j$ agree with that of $\arrowP_j$.

Now consider the set of sequences of Reeb chords $\arrowrho_{\boat} = \arrowrho_1, \dots, \arrowrho_k$. We will say that a map $u$ is \emph{compatible} with $\arrowrho_{\boat}$ if for each $i$ and $j$, we have that $\{\rho(q)\}_{q \in P_j^i} = \rho^i_j$. Sometimes we shall write
$$\freemod^B(\vec\bx; S; \arrowrho_{\boat})$$
for the moduli space of curves in $\freemod^B(\bx; S; \arrowP)$ which is compatible with $\arrowrho_{\boat}$.

As in subsection \ref{splicingbordered}, if we are given an $m$-splicing $\vec\sigma = \vec\sigma_1, \dots, \vec\sigma_h$, there are corresponding submanifolds $\Delta_{\vec\sigma}^1, \dots, \Delta_{\vec\sigma}^m$ of $\RR^{|E(S)|}$, and thus we may define the moduli spaces  $\freemod^B(\vec\bx ;S;P; \sigma)$ and $\freemod^B(\vec\bx, ;S;\arrowP; \sigma)$ by $\ev^{-1}(\Delta_{\sigma}^1, \dots, \Delta_{\sigma}^h)$ for the evaluation maps from $\freemod^B(\vec\bx;S;P)$ and $\freemod^B(\vec\bx;S;\arrowP)$ to $\RR^{|E(S)|}$ respectively. We also write $\freemod^B(\vec\bx, ;S;\arrowP; \vec\sigma)$ for the subset of $\freemod^B(\vec\bx, ;S;\arrowP; \sigma)$ such that the ordering upon jumping punctures induced by the evaluation maps agrees with the ordering dictated by $\vec\sigma$.

As the group of M\"{o}bius transformations acts three-transitively upon boundary points of the disc, we can assume that the greatest height of an east puncture of $S$ labelled by a jumping chord of $\cT$ is $0$. When $m \ne 2$, this does not alter the moduli spaces we study; when $m = 2$, these moduli spaces are precisely the bordered moduli spaces $\freemod^B(\bx, \by; S; \arrowrho_\boat; \vec\sigma)$ described in Section \ref{borderedmoduli}, and the effect of fixing the height to be zero is to instead consider the quotient moduli spaces $\tildemod^B(\bx, \by; S; \arrowrho_\boat; \vec\sigma)$. In general, we denote the relevant moduli space of curves with greatest height of an east puncture fixed to be zero by replacing any instances of $\freemod$ with $\tildemod$ in the above definitions.

There is a corresponding transversality result: 
\begin{prop}\label{borderedpolytransversality} There is an admissible, compatible family of almost-complex structures $\bJ$ for $\cD$, such that for every $l$-fold subdiagram of $\cD$, every source, every sequence of generators for the subdiagram, every set of sequences of Reeb chords $\arrowrho_{\boat}$ and every corresponding splicing $\vec\sigma$ the corresponding moduli spaces  $\tildemod^B(\vec\bx; S; \arrowrho_{\boat}; \vec\sigma)$ are all transversally cut out by the $\delbar$ equations. In this instance, we shall say that $\bJ$ is \emph{generic}.
\end{prop}
\begin{proof}This is a minor modification of \cite[Proposition 4.18]{LOT:SSII}.
\end{proof}

In fact, the proof of Proposition \ref{borderedpolytransversality} contains a little more information. Namely:
\begin{prop}\label{varyingsmallfamily}Let $\mathcal{U}$ denote the space of all admissible, compatible families almost-complex structures, and suppose that $\bJ$ is a generic family of admissible, compatible almost-complex structures. For any path $\bJ_t: [0,1] \rightarrow \mathcal{U}$ which satisfies that $\bJ_0 = \bJ$, there is some $\varepsilon$ such that $\bJ_t$ is generic for all $t \in [0, \varepsilon)$.
\end{prop}
This should be compared with, for instance, \cite[Remark 3.2.8]{McDS:jhol}.

The common generalisation of Theorems \ref{indbord} and \ref{indpoly} is below, and proved in a very similar manner: cf. \cite[Lemma 4.7]{LOT:SSII}.

\begin{prop}\label{bordedpolyindex}Suppose that $(u,j) \in \tildemod^B(\vec\bx; S; \arrowrho_{\boat}, \vec\sigma)$ is a bordered holomorphic polygon. Then the index of the $D\delbar$ operator at $u$ is given by
$$\ind(B,S,\arrowrho_\boat, \sigma) := \frac{3-m}{2} g - \chi(S) + 2e(B) + |\arrowrho_{\boat, 1}| + \cdots + |\arrowrho_{\boat,h}| - \col(\sigma).$$
\end{prop}

And, much like Proposition \ref{embeddedeuler}, we have the following:
\begin{prop}\cite[Proposition 4.9]{LOT:SSII} \label{embeddedeulerpoly}Whether or not a holomorphic curve in $\tildemod^B(\vec\bx; S; \arrowrho_{\boat}, \vec\sigma)$ is embedded is determined solely by the Euler characteristic of the source $S$.
\end{prop}

\subsubsection{Compactifying at east infinity}\label{eastcompactifying}

We now consider the effect of compactifying a toroidal bordered diagram at east infinity upon the corresponding moduli spaces.
If $\cD = (\Sigma; \bgamma^1, \dots, \bgamma^m; \bp; \bb)$ is a toroidal bordered multi-diagram, recall that there is a corresponding compactified diagram $\cD_{\bar{e}}$, where $\partial \Sigma_j$ is replaced by a single point $e^\infty_j$, and completing the arcs $\bgamma^{1,a}$ to curves $\bgamma^{1,a}_{\bar{e}}$. There is a corresponding complex structure $j_{\Sigma_{\bar{e}}}$ (resp. symplectic form) upon $\Sigma_{\bar{e}}$ which agrees with $j_\Sigma$ (resp. $\omega_\Sigma$) outside of arbitrarily small neighbourhoods of $e_j^\infty$, and is such that the compactified $\gamma^{1,a}_i$ are smooth curves (i.e. the angles between them are correct). Similarly, to every family of (generic, admissible, consistent) almost-complex structures $\bJ = \bigcup_m \bJ^m$ there is a corresponding family $\bJ_{\bar{e}} = \bigcup_m \bJ_{\bar{e}}^m$, compatible with the corresponding split symplectic form $\omega_{\Sigma_{\bar{e}}}  \times \omega_{D_n}$. Any sequence of generators $\vec\bx \in \cG(\cD)$ has an obvious counterpart in $\cD_{\bar{e}}$, which we call $\vec{\bx}'$ here but usually abuse notation and also call $\vec\bx$. Similarly each domain $B \in \pi_2(\vec\bx)$ has a corresponding domain $B_{\bar{e}} \in \pi_2(\vec\bx')$.

To any map $u: S \rightarrow X_{\cD, m} \times D_m$ there is a corresponding map $u': S \rightarrow X_{\cD_{\bar{e}}} \times D_m$. We can do slightly better, though: because of the conformal structure upon $X_{\cD_{\bar{e}}}$ near east infinity, punctures corresponding to non-jumping chords can be removed completely from $S$. To any bordered polygonal source $S$ we may associate the source $S_{\bar{e}}$ given by compactifying $S$ at its east punctures labelled by $\rho_{12}$ or $\rho_{23}$.  We then consider the space of pairs $(u, J_{\bar{e}})$ where $J_{\bar{e}} \in \bJ_{\bar{e}}$ and $u$ is a $J_{\bar{e}}$-holomorphic curve satisfying analogues of properties (M-1)--(M-8) and (BOR-1)--(BOR-3).

We denote the moduli space of such pairs in homology class $B_{\bar{e}}$ with asymptotics described by $\vec\bx = (\bx^{0,1}, \dots, \bx^{m-1,0})$ by ${\freemod^{B_{\bar{e}}}}_{\bar{e}}(\bx^{0,1}, \dots, \bx^{m-1, 0}; S_{\bar{e}}; \bJ_{\bar{e}})$.

This is closely related to the moduli space $\freemod^B(\vec\bx; S;  \bJ)$. In particular, if $u$ is a holomorphic curve in $\freemod^B(\vec\bx; S; \bJ),$ there is a corresponding map $u_{\bar{e}} \in {\freemod}^{B_{\bar{e}}}_{\bar{e}}(\vec\bx; S_{\bar{e}}; \bJ_{\bar{e}})$ which can be constructed as follows.

In a small neighbourhood of an east puncture of $S$ the map $u$ takes the form $u = (u_{\Sigma}, u_{D_n})$, as the almost-complex structure $J$ is split near east infinity. By the Schwarz reflection principle and Riemann removable singularities theorem, the map $u_\Sigma$ can be extended holomorphically across the puncture. By the uniqueness in Schwarz and Riemann's theorems, this identifies the moduli space $\freemod^B(\vec\bx; S; \bJ)$ with ${\freemod}^{B_{\bar{e}}}_{\bar{e}}(\vec\bx; S_{\bar{e}}; \bJ_{\bar{e}})$, and we denote the image of $\freemod^B(\vec\bx; S; \arrowrho_{\boat}, \vec\sigma; \bJ)$ under this identification by $${\freemod}^{B_{\bar{e}}}_{\bar{e}}(\vec\bx; S_{\bar{e}}; \arrowrho_{\boat}, \vec\sigma; \bJ_{\bar{e}}).$$

If $u \in {\freemod}^{B_{\bar{e}}}_{\bar{e}}(\vec\bx; S_{\bar{e}}; \bJ_{\bar{e}})$, then for $j = 1, \dots, h$, there is a sequence of partitions $P_j$ of $\bar{E_j}(S_{\bar{e}}): =(\pi_\Sigma \circ u)^{-1}(e^{\infty}_j)$, determined by dictating that $q, q'$ belong to the same part if and only if $\ev(q) = \ev(q')$, which may be ordered as $\arrowP_j$, where $P_j < P_j'$ if and only if for some $q \in P_j$ and $q' \in P'_j$
 ordered by $q< q'$ if and only if $\ev_j(q) < \ev(q')$.
 
 If $q \in \bar{E_j}(S_{\bar{e}})$, then there is an associated Reeb chord $\rho(q)$ defined by first choosing a small disc neighbourhood $M_j$ of $e^\infty_j$ such that $S_q : =(\pi_\Sigma \circ u)^{-1}(M_j)$ is a small half-disc neighbourhood of $q \in S$ which is disjoint from all other elements of $\bar{E_j}(S_{\bar{e}})$. The homology class of $(\pi_\Sigma \circ u)(S_q)$ is given by $B(\rho(q))$ for $\rho(q)$ some element of the torus algebra, as described in Section \ref{splayingdomains}. It follows that to $u$ is a corresponding set of sequences of algebra elements $\arrowrho(u) : =\arrowrho_1, \dots, \arrowrho_h$, given by replacing each $q \in \arrowP_j$ by $\rho(q)$.

The subspace 
$${\freemod}^{B_{\bar{e}}}_{\bar{e}}(\vec\bx; S_{\bar{e}}; \arrowrho_{\boat}, \vec\sigma; \bJ_{\bar{e}}) \subset {\freemod}^{B_{\bar{e}}}_{\bar{e}}(\vec\bx; S_{\bar{e}}; \bJ_{\bar{e}})$$
can be characterised as the set of $u$ such that $\arrowrho(u) = \arrowrho_{\boat}$, and such that the ordered partition of the punctures of $S_{\bar{e}}$ (which correspond with jumping chords in $\arrowrho_{\boat}$) induced by the evaluation maps agrees with the splaying $\vec\sigma$.

As with the moduli spaces $\freemod^B(\vec\bx; S; \arrowrho_\boat; \vec\sigma)$, we may assume that the highest height of an east puncture of $S$ is $0$ --- calling the corresponding moduli spaces by $ {\tildemod}^{B_{\bar{e}}}_{\bar{e}}(\vec\bx; S_{\bar{e}}; \bJ_{\bar{e}})$ and ${\tildemod}^{B_{\bar{e}}}_{\bar{e}}(\vec\bx; S_{\bar{e}}; \arrowrho_{\boat}, \vec\sigma; \bJ_{\bar{e}})$.
%The following remark will come in useful later:
%\begin{rmk}\label{biggereastinformation}
%If $\sigma$ is an interleaved $k$-splicing, consider the non-empty part $P^k_j \in \sigma$, with corresponding Reeb chord $\rho^i_j \in \arrowrho_j$. Denote by $\arrowrho^i_j$ the corresponding sequence $\rho^i_j, \rho^{i+1}_j, \dots, \rho^{|\arrowrho_j|}_j$, and let $\Xi:=\xi^1_i, \dots, \xi^{|\arrowrho_j|}_j$ denote the corresponding sequence of elements of $(\pi_\Sigma \circ u)^{-1}(e^{\infty, i}_j)$. If $N_j$ denotes the neighbourhood of $e^\infty_j$ as used in Proposition ???, then we note that the domain of each component $S'$ of $(\pi_\Sigma \circ u)^{-1}(N_j)$ containing an element of $\Xi$ may be characterised as $B(\arrowrho_{S'})$, where $\arrowrho_{S'}$ is the $123$-reduction of the sequence of Reeb chords $(\rho(\xi): \xi \in S')$, ordered by height of the corresponding $\xi$.
%\end{rmk}

If $\arrowrho_\boat$ contains no jumping chords, then along the edge $\Sigma_{\bar{e}}\times e^1$, $u(\partial S_{\bar{e}})$ is contained in a set of $g+h-1$ cylinders which are a subset of $C^1_{\bar{e}}:= \bgamma^1_{\bar{e}} \times e^1$. It follows that:
\begin{prop}\label{eastcompactifytoclosed} If $\arrowrho_\boat$ has no jumping chords, then the moduli space $\tildemod^{B_{\bar{e}}}_{\bar{e}}(\vec\bx; S_{\bar{e}}; \arrowrho_{\boat}; \emptyset; \bJ_{\bar{e}})$ defined by counting (bordered) polygons associated with the partially splayed diagram $\cD^\partial(\varepsilon, \arrowiota)$ is identified with the moduli space $\tildemod^{B_{\bar{e}}}(\vec\bx; S_{\bar{e}} ; \bJ_{\bar{e}})$, defined by counting polygons with respect to the (not partially) splayed Heegaard multi-diagram $\cD(\varepsilon, \arrowiota)$.
\end{prop}

\subsubsection{Fixing cut behaviour}\label{fixingcutbehaviour}

Now we consider the case where $\cD$ is a splayed or partially splayed diagram diagram  $\cD'(\varepsilon, \arrowiota)$ or  $\cD^\partial(\varepsilon, \arrowiota)$, and suppose that $B\in \dom(\cD)$ is a domain where $\vec\sigma(B) = \vec\sigma^1, \dots, \vec\sigma^k$.

If $\arrowrho(B) = ([\arrowrho_1(B)], \dots, [\arrowrho_h(B)])$, let $\widetilde{\rho_j} = (\widetilde{\rho}_j^1, \dots, \widetilde{\rho_j}^{i_j})$ be the $123$-reduced representative of $[\arrowrho_j(B)]$, and denote by $\widetilde{\rho_\anchor} := (\widetilde{\rho_1}, \dots, \widetilde{\rho_h})$  . If $S$ is a decorated $k$-source, we denote by $\rho_{123}(S, u)$ the set of punctures of $S$ which are mapped by $u$ to a point $e^{\infty, i}_j$ such that the $i$-th chord $\widetilde{\rho}^{k_i}_j$ in $\jump({\widetilde{\rho_j}})$ satisfies $\widetilde{\rho}^{k_i}_j= \rho_{123}$.

We will impose extra conditions upon the moduli spaces $\tildemod^B(\vec\bx; S; \arrowrho_{\boat}; \vec\sigma)$ based upon the behaviour of $u$ near any $q \in \rho_{123}(S, u)$. Namely, fix such a $q$ which is mapped to $e_j^{\infty,i}$. We shall say that $u$ has a \emph{cut} at $q$ if $\pi_\Sigma \circ u$ has a boundary branch point along one of the arcs $\gamma^i_j$  or $\gamma^{i+1}_j$. There are three cases:
\begin{enumerate}\item if there is a boundary branch point upon $\gamma^{i, i+1}$, we will say that $q$ is of $\rho_{12}, \rho_3$-type;
\item when there is a boundary branch point upon $\gamma^{i, i+1}$, we will say that $q$ is of $\rho_1, \rho_{23}$-type; and, finally,
\item when there is no boundary branch point we shall say that $q$ is of $\rho_{123}$-type.
\end{enumerate}

This has an equivalent description: if there is a branch point along $\gamma^{i}_j$ then there is some small disc neighbourhood $S(q)$ of $e_j^{\infty, i}$ such that when restricted to the component $S(q)$ of $u|_{S(q)}$ containing $q$, the homology class of $\pi_\Sigma \circ u(S(q))$ is modelled upon $B(\rho_3)$. Similarly, if the branch point appears upon $\gamma^{i+1}_j$, there is a corresponding $S(q)$ such that the homology class of  $\pi_\Sigma \circ u(S(q))$ is modelled upon $B(\rho_1)$. Finally, if there is no cut at $q$, every sufficiently small disc neighbourhood $M_j^i$ of $q$ satisfies that, restricted the the component $S(q)$ containing $q$, the map $\pi_\Sigma \circ u$ is modelled upon $B(\rho_{123})$.

We would like to impose conditions upon our moduli spaces that dictates which of these cases holds for each $q$. We first choose sensible notation, and then re-write the conditions in a way for which transversality and index formulae are more apparent.
\begin{defn}Fix $B \in \pi_2(\vec\bx)$ with associated sets of sequences of Reeb chords $\widetilde{\rho_\anchor}$, and fix also a splayed source $S$ with corresponding partition of $\rho_{123}(S)$ into $Q = (Q_{123}, Q_{1,23}, Q_{12,3})$.

We shall write $\tildemod^B(\vec\bx; S; \arrowrho_{\boat}, Q_{123};\vec\sigma)$ for the subset of $u \in \tildemod^B(\vec\bx; S; \arrowrho_{\boat}; \vec\sigma)$ where $\rho_{123}(S, u) = \rho_{123}(S)$ and each $q \in Q_{123}$ is of $\rho_{123}$-type (with respect to $u$). Similarly, we shall write $\cM^B(\vec\bx; S; \arrowrho_{\boat}, Q; \vec\sigma)$ for the subset of $\tildemod^B(\vec\bx; S; \arrowrho_{\boat}, Q_{123}; \vec\sigma)$ such that each $q \in Q_{1,23}$ is of $\rho_1, \rho_{23}$-type, and each $q \in Q_{12,3}$ is of $\rho_{12}, \rho_3$-type. 
\end{defn}
We often use a shorthand for these moduli spaces:
\begin{defn}Write ${\arrowrho_\anchor}(Q,u)$ for the representative for the $123$-equivalence class $[\arrowrho(B)]$ given by replacing each $\rho \in \widetilde{\rho_\anchor}$ which is equal to $\rho_{123}$ by $\rho_{12}, \rho_3$ if $q(\rho)\in Q_{12,3}$; by $\rho_1, \rho_{23}$ if $q(\rho) \in Q_{1,23}$; and, finally, leaving it as $\rho_{123}$ if $q(\rho) \in Q_{123}$.

If $\arrowrho_\anchor$ is a representative of $[\arrowrho(B)]$, define by
$$\cM^B(\bx, S, \arrowrho_\boat, \arrowrho_\anchor; \vec\sigma)$$
the moduli space
$$\cM^B(\vec\bx; S; \arrowrho_{\boat},Q; \vec\sigma),$$
where $\arrowrho_\anchor = \arrowrho_\anchor(Q,u)$. 

We call an element of $\cM^B(\bx, S, \arrowrho_\boat, \arrowrho_\anchor; \vec\sigma)$ a \emph{splayed} bordered holomorphic polygon.
\end{defn}

To ensure the corresponding moduli spaces are transversally cut out, we rephrase this condition in a way more compatible with the language of almost-complex structures (cf. \cite[Lemma 9.4]{Ozsvath-Szabo:2001}, \cite[Chapter 6]{ahlforsconformal}). First, a definition:
\begin{defn}Let $u \in \cM^B(\vec\bx; S; \arrowrho_{\boat}, Q; \vec\sigma)$, and let $q \in Q_{1,23}$ or $Q_{12, 3}$.  Denote the corresponding boundary branch point by $b$, and the arc of $\partial S$ connecting $q$ and $(\pi_\Sigma \circ u)^{-1}(b)$ by $\lambda_{q,b}$. We define the \emph{length of the cut at $q$} to be the length of the arc $\bar{\lambda}_{q,b}: = (\pi_\Sigma \circ u)(\lambda_{q,b})$ with respect to the metric induced upon $\Sigma$ by the uniformisation theorem.
\end{defn}

We focus initially upon curves where the cut behaviour near each $q \in \rho_{123}(S)$ is fixed so that either there is no cut, or such that the length of the cut at $q$ is short enough that the corresponding boundary branch point $b$ is contained in a disc neighbourhood $N^i_j$ of $e^{\infty,i}_j$ which is holomorphically identified with the unit disc in $\CC$. We also require that this disc is small enough that for each $q$, the component $S(q)$ of $(\pi_\Sigma \circ u)^{-1}(N^i_j)$ containing $q$ is a half-disc with $q$ its only boundary puncture. Denote by $(u^i_j, j')$ the restriction of $(u, j')$ to $(\pi_\Sigma \circ u)^{-1}(N^i_j)$ (we recall that we have been writing $u$ for the pair $(u, j')$ where $j' \in \conf(D_m)$).

We can modify the complex structure upon $N^i_j$ so that the angle between the arcs in $\partial B^i_j(\rho_{123})$ is $\pi$ (by composing with the map $\varphi: z \mapsto z^{2/3}$, say), and modify the (split) almost-complex structure $J|_{N^i_j} \times j'$ over $N^i_j \times D_m$ accordingly, to a corresponding almost-complex structure $J_{N^i_j}^0 \times j'$. To $u^i_j$ we may associate a corresponding $v^i_j$ with respect to this new almost-complex structure: moreover, if $u^i_j$ had no boundary branch point then this new map does not either, and so there is a corresponding holomorphic map from the half-disk $S(q)$ with boundary puncture $q$ removed to the manifold $X_{\cD, m}$ (by the removable singularities theorem and Schwarz reflection principles).

Similarly, if $u^i_j$ had a boundary branch point then so does $v^i_j$. We further modify the complex structure by composing the chart on $N^i_j$ with the Riemann mapping which takes the open disc with slit of length $t$ along $\varphi(\gamma^{i+1}_j)$ (resp. slit of length $-t$ along $\varphi(\gamma^i_j)$) to the open half-disc. These Riemann mappings yield a family of  almost-complex structures $J_{N^i_j}^t$ upon $N^i_j$ for each $t \in (-1,1)$, such that for every $J_{N^i_j}\times j'$-holomorphic curve $u^i_j$ with a cut of length $t$ along $\gamma^{i+1}_j$ (resp. length $-t$ along $\gamma^i_j$), the corresponding curve (which we also denote by $v^i_j$) with respect to $J_{N^i_j}^t \times j'$ has no cut, and thus can be completed to a holomorphic map with source given by the half-disk.

We would like to modify the almost-complex structures $J_{j'}$ to agree with this $J_{N^i_j}^t \times j'$ in a neighbourhood of $u(q)$ and then apply the transversality result of Proposition \ref{borderedpolytransversality}, but this does not make sense with respect to other components of $u^{-1}(N^i_j)$. The solution, as in \cite[Proposition 6.36]{LOT} (cf. also \cite{seidel:dehn}) is to instead consider, for $\vec{t} = \{t_q\}_{q \in \rho_{123}(S)}$, related moduli spaces of $\bJ^{\vec{t}}$- holomorphic sections $u: S \rightarrow S \times \Sigma_{\bar{e}} \times D_n$, where each $\bJ^{\vec{t}}$ is an almost-complex structure on $S \times \Sigma_{\bar{e}} \times D_n$ equal upon each $(s, z, w)$ to some $j_S \times \bJ^{\vec{t}}_s$. We choose $\bJ^{\vec{t}}_s$ so that for $s$ near any puncture $q \in \rho_{123}(S)$, the almost-complex structure $\bJ^{\vec{t}}_s|_{N_j \times D_m}$ agrees with the split almost-complex structure $J^{t_q}_{N^i_j} \times j'$ described in the previous paragraph, and away from this agrees with $\bJ$.

The proof of transversality for such families follows closely along the lines of \cite[Chapter 3]{Lipshitz:cylindrical}. As such, the moduli spaces $\tildemod^B(\vec\bx; S; \arrowrho_{\boat}, Q_{123}; \vec\sigma)$ are transversally cut out. Moreover the moduli space
$$\tildemod^B(\vec\bx; S; \arrowrho_\boat, \rho_{123}(S); \vec\sigma) \subset \tildemod^B(\vec\bx; S; \arrowrho_\boat, Q_{123}; \vec\sigma)$$
of curves with no cut at any member of $\rho_{123}(S)$ is also transversally cut out and, along the lines of Proposition \ref{varyingsmallfamily}, for some $\varepsilon$, the moduli spaces defined with respect to $\bJ^{\vec{t}}$ for $\vec{t} \in (-\varepsilon, \varepsilon)^{|\rho_{123}(S) - Q_{123}|}$ are also transversally cut out --- thus each curve $u \in \tildemod^B(\vec\bx; S; \arrowrho_\boat, \rho_{123}(S), \vec\sigma)$ has a neighbourhood in $\tildemod^B(\vec\bx; S; \arrowrho_\boat, Q_{123}; \vec\sigma)$ homeomorphic to $\tildemod^B(\vec\bx; S; \arrowrho_\boat, Q_{123}; \vec\sigma) \times (-\varepsilon, \varepsilon)^{\rho_{123}(S) - Q_{123}}$.

Suppose that $u, u' \in \tildemod^B(\vec\bx; S; \arrowrho_\boat, Q_{123}; \vec\sigma_\boat)$, and that with respect to $u$ a puncture $q \in \rho_{123}(S) - Q_{123}$ is of $\rho_{12}, \rho_3$-type, and with respect to $u'$ the same puncture is of $\rho_{1}, \rho_{23}$-type. The topology upon the moduli spaces described in Section \ref{defconvergence}  is metrizable (cf. \cite[Appendix B]{EBHWZ}), and for each $q \in \rho_{123}(S) - Q_{123}$ the length of the curve $\bar{\lambda}_{q,b}$ is continuous with respect to this metric. As such (by the intermediate value theorem), the curves $u, u'$ belong to different path-components of 
$$\tildemod^B(\vec\bx; S; \arrowrho_\boat, Q_{123}; \vec\sigma) - \tildemod^B(\vec\bx; S; \arrowrho_\boat, \rho_{123}(S); \vec\sigma).$$

It follows that:

\begin{prop}\label{indextransversalitycut}There is an admissible, compatible family of almost-complex structures $\bJ$ for $\cD$, such that for all choices of argument, the moduli spaces $\cM^B(\vec\bx; S; P, Q; \vec\sigma_\boat)$  are all transversally cut out, and of expected dimension
$$\ind(B,S, P, Q, \sigma): = \ind(B, S, P, \sigma) - \#_{123}(\arrowrho_\anchor).$$
\end{prop}
\begin{proof}Transversality is explained in the lead-up to the statement of the proposition. The index formula follows along similar lines to \cite[Section 4.1]{Lipshitz:cylindrical}: we do not count `corners' stemming from punctures of $\rho_{123}$-type, as with respect to the relevant $\bJ^{\vec{t}}$ they are not corners at all.
\end{proof}

\begin{convention}In the course of this paper, we will be considering a number of different partial splayings of a bordered Heegaard diagram, which has the potential to become confusing. By Definition \ref{splicingtoidemps}, a splicing $\vec\sigma$ of a set of sequences of Reeb chords $\arrowrho_\anchor$ determines a sequence of sets of idempotents $\iota_{\vec\sigma}$: as such, we may indicate the diagram being considered for a specific moduli space by using the notation
$$\cM^B(\bx; S; \arrowrho_\boat, \arrowrho_\anchor; \vec\sigma_\boat, \vec\sigma_\anchor)$$
to refer to the moduli space
$$\cM^B(\bx; S; \arrowrho_\boat, \arrowrho_\anchor; \vec\sigma_\boat)$$
defined using the diagram $\cD^\partial(\varepsilon, \iota_{\vec\sigma_\anchor})$ or $\cD(\varepsilon, \iota_{\vec\sigma_\anchor})$.
\end{convention}

\subsubsection{Compactifications}\label{combs}
As in \cite{LOT}, the aforementioned moduli spaces admit natural compactifications in terms of \emph{holomorphic combs}.

To discuss these, we need the notion of a holomorphic tooth: this is a map $$v: (T, \partial T) \rightarrow (\RR \times (\partial \Sigma - \bb) \times D_n, \RR \times \ba \times e^1)$$ where $T$ is a tooth source, and the codomain is endowed with the obvious split symplectic form and the split complex structure $J = j_\Sigma \times j_{D_n}$. We denote by $\cN(T)$ the moduli space of such maps satisfying:
\begin{enumerate}\item $v$ is $(j, J)$-holomorphic for some almost complex structure $j$ upon $T$
\item $v$ is proper
\item $\pi_{D} \circ v$ maps $\partial T$ to $e^1$.
\item Near each west puncture $q \in W_j(T)$ labelled by a Reeb chord $\rho$, $\lim_{z \rightarrow q}\pi_\Sigma \circ v(z)$ is $\rho \subset \{-\infty\} \times \partial\Sigma_j$.
\item Near each east puncture $q \in E_j(T)$ labelled by a Reeb chord $\rho$, $\lim_{z \rightarrow q}\pi_\Sigma \circ v(z)$ is $\rho \subset \{+\infty\} \times \partial\Sigma_j$.
\end{enumerate}

These moduli spaces can once-more be refined by requiring the evaluation maps $\ev_q := \lim_{z \rightarrow q} t \circ v(z)$ partition the east and west punctures of $T$ in a prescribed manner. Explicitly, if $P_w := (P_{w,1}, \dots, P_{w,h})$ and $P_e := (P_{e,1}, \dots, P_{e,h})$ are partitions of the sets $W(T) = (W_1(T), \dots, W_h(T))$ and $E(T) = (E_1(T), \dots, E_h(T))$ respectively, denoting
\begin{align*}\ev_{w,j} &:= \prod_{q \in W_j(T)} \ev_q: \cN(T) \rightarrow (e^1)^{|W_j(T)|} \\
\ev_{e,j} &:= \prod_{q \in E_j(T)} \ev_q: \cN(T) \rightarrow (e^1)^{|E_j(T)|};
\end{align*}
and $\ev_w := \ev_{w,1} \times \cdots \times \ev_{w,h}$ and $\ev_w := \ev_{w,1} \times \cdots \times \ev_{w,h}$, we write $\cN(T; P_w, P_e)$ for the moduli space $$(\ev_w \times \ev_w)^{-1}(\Delta_{P_{w,1}} \times \cdots \times \Delta_{P_{w,h}} \times \Delta_{P_{e,1}} \times \cdots \times \Delta_{P_{e,h}})$$ --- omitting $P_w$ from the notation if it is the discrete partition. If we are given a splicing $\sigma_e$ of the corresponding sequence of Reeb chords given by replacing each $q \in P_e$ by the Reeb chord labelling it, we can restrict these moduli spaces even further, writing $\cN(T;P_w, P_e; \sigma_e)$ for the subset  $\ev^{-1}(\Delta_{\sigma}^1, \dots, \Delta_{\sigma}^h)$ for the evaluation maps from $\cN(T;P_w, P_e; \sigma_e) \rightarrow e^1$.

Transversality is known in a simple case:
\begin{prop}\cite[Proposition 5.16]{LOT} If all components of $T$ are topological disks, then $\cN(T)$ is transversally cut out by the $\delbar$-equations for any choice of split almost-complex structure $J$.
\end{prop}

This will allow us to describe the compactifications of our moduli spaces, as in \cite[Section 5.4]{LOT}. We firstly consider small constituent pieces:
\begin{defn}A simple holomorphic $m$-comb is a pair $(u,v)$ of a bordered holomorphic polygon $(u,j) \in \freemod^B(\vec\bx; S)$ and holomorphic teeth $v \in \cN(T)$ for some bordered polygon source $S$ and tooth source $T$, together with a one-to-one correspondence between $E(S)$ and $W(T)$ that preserves labelling by Reeb chords, such that $\ev_w(v) = \ev(u)$.
\end{defn}

In more generality, we must consider curves with even more degenerate behaviour at east infinity:

\begin{defn}A holomorphic $m$-story is a sequence $U = (u, v_1, \dots, v_k)$ for some $k \ge 0$, where, for some $B$, bordered polygonal source $S$ and collection of tooth sources $T_1, \dots, T_k$:
\begin{itemize}\item $u \in \freemod^B(\vec\bx; S)$ ;
\item $v_i \in \cN(T_i)$;
\item $(u, v_1)$ is a simple holomorphic $m$-comb;
\item there is a correspondence between $E(T_i)$ and $W(T_{i+1})$ for $i = 1, \dots, k-1$ which preserves labelling by Reeb chords; and
\item $\ev_e(v_i) = \ev_w(v_{i+1})$ for every $i$.
\end{itemize}
If $k = 0$, we say that $U$ is \emph{toothless}.
\end{defn}

In full generality, we need to consider trees of holomorphic stories.

\begin{defn}A \emph{tree of holomorphic combs} is:
\begin{enumerate}\item A metric tree $(T, \rho)$ with corresponding reduced tree $\bar{T}$.
\item A generator assignment $\vec\bx(\bar{T})$ for $\bar{T}$; a tree of $\vec\bx(\bar{T})$-compatible domains $B(\bar{T})$; a comb-tree of sources $S(\bar{T})$; and a tree of admissible, generic, compatible tree of almost-complex structures $\bJ_T \subset \bJ$.
\item A set of holomorphic stories $\{U_V\}_{V \in V(\bar{T})}$, such that each $U_V$ is a $J(V)$-holomorphic $|V|$-story in the moduli space $\freemod^{B(V)}(\vec\bx(V); S(V))$, and for every non-comb vertex $V \in \bar{T}$, the story $U_V$ is toothless.
\end{enumerate}
If $|V(\bar{T})| = k$, we often say that $U$ has \emph{height $k$}, and we often refer to a heigh $k$ tree of holomorphic combs by its vertices: i.e.  $T = (u_1, \dots, u_k)$.
\end{defn}

A priori, we must allow holomorphic curves to have nodal sources to ensure compactness holds. As in \cite[Section 5.4]{LOT}, we refer to a tree of elementary holomorphic combs as \emph{nodal} or \emph{smooth} if we want to explicitly allow or forbid nodes.

Recall that to each spinal vertex $V_1, \dots, V_k$ of $\bar{T}$ we may associate its storied tree $T(V)$. Each of these has an associated domain $B(V)$, equal to the sum of the domains $\{B(V')\}_{V' \in T(V)}$, and so we may associate to any $T$-tree $U$ a domain $B$, given by the sum of all the domains of $V \in S(T)$. We can also associate to $U$ a set of well-defined Reeb chords which is given by the Reeb chords of the right-most tooth of each story in the spine.

As in \cite[Section 5.4]{LOT}, we call a component $C$ of a $T$-tree $U$ \emph{$\Sigma$-stable} (resp. $D$-stable) if $(\pi_\Sigma \circ U)|_C$ (resp. $(\pi_D \circ U)|_C$) is a stable map, and $\Sigma$-unstable (resp. $D$-unstable) otherwise. We modify the definition of convergence from \cite{LOT} to general $T$-trees of elementary combs:

\begin{defn}\label{defconvergence}We shall say that a sequence $(u_n, j_n)$ of holomorphic $n$-gons in $\Sigma_{\bar{e}} \times D_m$ \emph{converges} to a holomorphic tree $U$ if:
\begin{enumerate}\item Let $S_\Sigma$ denote the result of collapsing all $\Sigma$-unstable components of $U$. Then the map $\{\pi_\Sigma \circ u_n\}$ converges to $\pi_\Sigma \circ U|_{S_\Sigma}$, in the sense of \cite[Section 7.3]{EBHWZ}.
\item The complex structures $j_n$ on $D_m$ converge to the tree of almost complex structures $\{j_v\}_{v \in T}$ over $\{\Sigma_{\bar{e}} \times D_{|v|}\}_{v \in T}$.
\item Let $S_D$ be the result of collapsing all $D$-unstable components of $U$. Then $\{\pi_D \circ u_n\}$ converges to $\pi_D \circ U|_{S_D}$ in the sense of \cite[Section 7.3]{EBHWZ}.

\item For $n$ sufficiently large, all $u_n$ represent the same homology class $B = [U]$.\end{enumerate}
\end{defn}

As is tradition, to avoid drowning in notation, we leave the more general definition of convergence of $T$-trees of elementary holomorphic combs to the reader.

With this in hand, we make the following definition:

\begin{defn}\label{modulicompactifications}Denote by:\begin{itemize}\item $\overline{\overline{\cM}}^B(\vec\bx; S)$ the space of all trees of elementary holomorphic combs with preglued source $S$ in homology class $B$, asymptotic to the sequence of generators $\vec\bx = (\bx^{0,1}, \dots, \bx^{m-1, 0})$, compatible with the almost complex structures $\bJ$.
\item $\overline{\cM}^B(\vec\bx; S)$  the closure of the moduli space $\cM^B(\vec\bx; S)$ in $\overline{\overline{\cM}}^B(\vec\bx; S)$.

\item  $\overline{\overline{\cM}}^B(\vec\bx; S; P; \sigma)$ the space of all relevant holomorphic trees respecting the partition $P$.
\item $\overline{\cM}^B(\vec\bx; S; P)$ is the closure of $\cM^B(\vec\bx; S; P;\sigma)$ in  $\overline{\overline{\cM}}^B(\vec\bx; S; P)$.
\end{itemize}
To save on ink and patience, we observe the convention that for all other moduli spaces $\cM^B(*)$, we denote by $\overline{\cM}^B(*)$ the closure of  $\cM^B(*)$ in the corresponding moduli space $\overline{\cM}^B(\vec\bx; S; P; \sigma)$.
\end{defn}

Indeed, these are the required notions for compactifying the moduli spaces:

\begin{prop}The spaces above are compact: i.e. if $\{U_n\}$ is a tree of holomorphic polygon combs in a fixed homology class with fixed preglued topological source $S$, then $\{U_n\}$ has a subsequence which converges to a tree of holomorphic polygon combs $U$ in the same homology class.
\end{prop}
\begin{proof}The argument from \cite[Prop. 5.24]{LOT} carries over readily to this more general case. In particular,  \cite[Prop. 5.27]{LOT} applies to a general symmetric cylindrical almost-complex manifold, and doubling both source and $X_{\cD, m}$ is still a manifold of this type.
\end{proof}

\subsection{Holomorphic anchors and perturbations}\label{sub:anchorsperturbations}
We discuss modifying the Lagrangian cylinders used in the definition of the moduli spaces $\cM^B(\vec\bx; S; \arrowrho_\boat, \arrowrho_\anchor; \vec\sigma_\boat, \vec\sigma_\anchor)$.

\subsubsection{Interpolating cylinders}
Fix a partially splayed bordered diagram $\cD' =\cD^\partial(\varepsilon, \arrowiota)$ with attaching curves $\bgamma^0, \bgamma^1, \dots, \bgamma^m$,  and recall the east-compactified curves $\bgamma^1_{\bar{e}} = \bgamma^{1,a}_{\bar{e}} \cup \bgamma^{1,c}_{\bar{e}}$ with corresponding cylinders in the east-compactification of $X_{\cD', m}$ given by $$C^1_j: = \gamma^1_{\bar{e}} \times e^1 \subset \Sigma_{\bar{e}} \times e^1.$$ 

We can perturb these cylinders using the Hamiltonian $H$ as described in Section \ref{diagrams}. Namely, we put
$$C^1(\varepsilon) = \left\{\begin{array}{c c} H(\bgamma^{1}_{\bar{e}}, 0) \times t &\mbox{ for } t \in (-\infty, -\varepsilon/2] \\
H(\bgamma^{1}_{\bar{e}}, \frac{2t + \varepsilon}{2(m+1)})\times t & \mbox{ for } t \in [-\varepsilon/2, \varepsilon/2] \\
H(\bgamma^{1}_{\bar{e}}, \frac{\varepsilon}{m+1})\times t & \mbox{ for } t \in [\varepsilon/2, \infty) \end{array} \right.$$
(recall that we fixed an identification of a neighbourhood of $e^1$ with $\RR \times [0, \varepsilon)$).
For $t \le -\varepsilon/2$, these cylinders coincide with the cylinders $C^1 = \bgamma^1_{\bar{e}} \times e^1$, and for $t \ge \varepsilon/2$, they coincide with the cylinders $(C^1)' = H(\bgamma^1_{\bar{e}}, \frac{\varepsilon}{m+1}) \times e^1$. They are Lagrangian with respect to a perturbed symplectic form upon the east compactification of $X_{\cD', m}$, as described in \cite[Equation 3.25]{LOT:SSII}.
\subsubsection{Moduli spaces with perturbed cylinders}\label{perturbedcylinders}
Consider again a partially splayed diagram $\cD^\partial(\varepsilon, \arrowiota)$, with attaching curves labelled $\bgamma^0, \dots, \bgamma^m$. We consider the moduli space of holomorphic curves  $$\cM^{B_{\bar{e}}}_{\bar{e},\varepsilon}(\vec\bx; S_{\bar{e}}; \arrowrho_{\boat}, \arrowrho_\anchor; \vec\sigma; \bJ_{\bar{e}}),$$
which are defined by replacing the unperturbed cylinders $C^i_{\bar{e}}= C^i(\bgamma^i_{\bar{e}})$ in the definition of the east-compactified moduli space $$\cM^{B_{\bar{e}}}_{\bar{e}}(\vec\bx; S_{\bar{e}}; \arrowrho_{\boat}, \arrowrho_\anchor; \vec\sigma; \bJ_{\bar{e}})$$
with cylinders such that:

\begin{itemize}\item $C^0(\varepsilon)$ are the cylinders $\bgamma_{\bar{e}}^0 \times e^0$;
\item $C^1(\varepsilon)$ are the perturbed cylinders $C^1(\varepsilon)$; and
\item $C^i(\varepsilon)$ are the cylinders $H(\bgamma^i_{\bar{e}}, \frac{\varepsilon(m-i)}{m(m+1)}) \times e^i$ for all $i \ge 2$.
\end{itemize}

Note that if $\iota = (\iota_1, \dots, \iota_h)$ is a set of idempotents, then for the corresponding partially splayed diagram $\cD^\partial(\varepsilon, \iota \cup \arrowiota)$ with attaching curves $\bdelta^0, \bdelta^1 \dots, \bdelta^{m+1}$, we have that:
\begin{itemize}\item $C^0(\varepsilon)$ agree with $\bdelta^0_{\bar{e}} \times e^0$;
\item $C^1(\varepsilon)$ agrees with $\bdelta^1_{\bar{e}}$ for $t \le -\frac{\varepsilon}{2}$ and with $\bdelta^2_{\bar{e}}$ for $t \ge \frac{\varepsilon}{2}$; and
\item for every $i \ge 2$, $C^i(\varepsilon)$ agrees with the cylinders $\bdelta^{i+1} \times e^i$.
\end{itemize}

\begin{prop}\label{perturbingidentification}For small enough $\varepsilon$, all of the moduli spaces of the form $$\cM(\bJ_{\bar{e}}): = \cM^{B_{\bar{e}}}_{\bar{e}}(\vec\bx; S_{\bar{e}}; \arrowrho_{\boat}, \arrowrho_\anchor; \vec\sigma; \bJ_{\bar{e}})$$ are identified with the corresponding moduli space $$\cM_{\varepsilon}(\bJ_{\bar{e}}):=\cM^{B_{\bar{e}}}_{\bar{e},\varepsilon}(\vec\bx; S_{\bar{e}}; \arrowrho_{\boat}, \arrowrho_\anchor; \vec\sigma; \bJ_{\bar{e}}).$$
\end{prop}
\begin{proof}For every $m \ge 2$, consider a Lagrangian isotopy $\phi_t = \phi(x, t): X_{\cD, m} \times [0,1] \rightarrow X_{\cD, m}$ which is supported in a $\varepsilon$-neighbourhood of each $\Sigma_{\bar{e}} \times e^i$, and satisfies that $\phi_t(C^i_{\bar{e}}) = C^i_{\bar{e}}(\varepsilon/t)$ for every $t \in [0,1]$.

A holomorphic curve in the moduli space $\cM_{t\varepsilon}(\bJ_{\bar{e}})$ is the same as a holomorphic curve in the moduli space $\cM(\bJ_{\bar{e}}^t)$, where $\bJ_{\bar{e}}^t$ is the family of perturbed almost-complex structures $\phi_t^*\bJ$. Thus we may identify these two moduli spaces. Moreover, the family of almost-complex structures $\bJ_{\bar{e}}^t$ constitutes a smooth path in the space of families of admissible almost-complex structures. As such, by Proposition \ref{varyingsmallfamily}, there is some $T \in (0,1)$ such that
$$\bigcup_{t \in [0, T]}\cM_{t\varepsilon}(\bJ_{\bar{e}}) = \bigcup_{t \in [0,T]}\cM(\bJ_{\bar{e}}^t)$$
provides a cobordism between $\cM(\bJ_{\bar{e}})=\cM_{0}(\bJ_{\bar{e}})$ and $\cM_{T\varepsilon}(\bJ_{\bar{e}})$, thus choosing $\varepsilon = T\varepsilon$ provides the result.
\end{proof}

\begin{rmk}
There is a similar construction for subdiagrams: if $\cD' := \cD^\partial(\varepsilon, \arrowiota)|_S$ is a subdiagram of $\cD^\partial(\varepsilon, \arrowiota)$, and $\vec\bx'$ is a set of generators for $\cD'$, there is a corresponding perturbed moduli space $$\cM^{B_{\bar{e}}}_{\bar{e}, \varepsilon}(\vec\bx'; S'_{\bar{e}}; \arrowrho_{\boat}; \bJ_{\bar{e}})$$
formed by replacing the cylinders $\{C^i: i \in S\}$ used in the definition of
$$\cM^{B_{\bar{e}}}_{\bar{e}}(\vec\bx'; S'_{\bar{e}}; \arrowrho_{\boat}; \bJ_{\bar{e}})$$
with the corresponding perturbed cylinders $\{C^{i}(\varepsilon): i \in S\}$. The proof of Proposition \ref{perturbingidentification} carries over to this case to identify these two moduli spaces.
\end{rmk}
Similarly, for regularisations of diagrams (Definition \ref{regularisation}) we also have:

\begin{prop}\label{regularisationidentification} Let $\cD|_S$ be a subdiagram of $\cD$, and $(\cD|_S)_{\text{reg}}$ be its regularisation. Then the moduli space $\cM_{\text{reg}}: =\cM_{\text{reg}}^B(\vec\bx; S; \arrowrho_\boat, \arrowrho_\anchor; \vec\sigma)$ defined using the diagram $(\cD|_S)_{\text{reg}}$ agrees with the moduli space $\cM|_S : =\cM(\vec\bx; S; \arrowrho_\boat, \arrowrho_\anchor; \vec\sigma)$ defined using $\cD|_S$.
\end{prop}
\begin{proof}This follows closely along the lines of the proof of Proposition \ref{perturbingidentification} --- one perturbs the cylinders for $\cD|_S$ to agree with those for $(\cD|_S)_{\text{\reg}}$. When $\bgamma^{i}, \bgamma^{i+1}$ are small approximations of one another, this identifies $\cM|_S$ with a moduli space of curves with boundary contained in Lagrangians $\bgamma^i \times e^i$ and $\bgamma^i \times e^{i+1}$: applying the removable singularities theorem to extend any $u'$ in this moduli space to a bona fide map in $\cM_{\text{reg}}$ yields the result.
\end{proof}

\subsubsection{Holomorphic anchors and stretching the neck}
Later on, we will stretch the neck of the almost-complex structures $\bJ$ along a hypersurface, forcing degenerations of moduli spaces. In this section, we fix notation in preparation for this: although we later largely restrict possible degenerations, we begin here in full generality.

Fix an almost-complex structure $J_\anchor$ upon $\Sigma_{\bar{e}} \times \HH$, where $\HH$ denotes the upper-half plane $\{z \in \CC: \Im(z) \ge 0\}$. Consider the moduli space of maps $$u_\anchor: (S_\anchor, \partial S_\anchor) \rightarrow (\Sigma_{\bar{e}} \times \HH, C^1(\varepsilon) \times \RR),$$
where $S_\anchor$ is an anchor source.

We require these to satisfy:
\begin{enumerate}\item $u_\anchor$ is $(j, J_\anchor)$-holomorphic for some almost complex structure $j$ upon $S_\anchor$;
\item $u_\anchor$ is proper;
\item $u_\anchor$ has finite energy;
\item $u_\anchor$ is disjoint from $\bp, \bb$;
\item the map $\pi_{\HH} \circ u_\anchor$ is a branched covering;
\item $u_\anchor$ is an embedding;
\item At each puncture $q$ of $S_\anchor$ labelled by $\infty$, $\lim_{z \rightarrow q}(\pi_D \circ u)(z) = \infty$.
\item At each puncture $q$ of $S_\anchor$ labelled by a Reeb chord $\rho$, we have that upon some neighbourhood $\eta(q)$ of $q$, $(\pi_\Sigma \circ u)(\eta)$ is modelled upon $B(\rho)$.
\item For every $t \in \RR$ and each $j = 1, \dots, h$, we have that $u^{-1}(C^1_j|_{\Sigma \times \{t\}})$ consists of at most one point, and for each $j = h+1, \dots, g+h-1$, the set $u^{-1}(C^1_j|_{\Sigma \times \{t\}})$ consists of exactly one point.

\end{enumerate}
From these conditions, it follows that near each puncture $q$ of $S_\anchor$ labelled by $\infty$, the map $u_\anchor$ is asymptotic to a chord of the form $\theta^{1, 2}(q) \times \infty$, such that $\btheta^{1, 2} = \cup_q \theta^{1, 2}(q)$ is a generator for the pair $(\bgamma^1_{\bar{e}},H(\bgamma^{1}_{\bar{e}}, \frac{\varepsilon}{m+1}))$, or, more transparently, for the pair $(\bdelta^1, \bdelta^2)$ described in the previous subsection. Each such map has an associated homology class $B_\anchor$ in the relative homology group
$H_2(\Sigma_{\bar{e}} \times \HH, C^1(\varepsilon))$. We call such curves \emph{holomorphic anchors} and collect them into moduli spaces $\cM^{B_\anchor}_\anchor(\btheta^{1,2}; S_\anchor)$. We shall later see that in fact, for the moduli spaces to be non-empty, the generator $\btheta^{1,2}$ must be the distinguished generator $\btheta^{1,2}_+$. 

As with all of our other moduli spaces, we can cut these down according to various partitions of the east punctures of $S_\anchor$. If $P_\anchor$ is a set of partitions $P_\anchor = (P_{\anchor,1}, \dots, P_{\anchor,h})$, where each $P_{\anchor,j}$ is a partition of $E_j(S_\anchor)$, there is a corresponding evaluation map
$$\ev := \prod_{q \in E(S_\anchor)} \ev_q: \cM^{B_\anchor}_\anchor(\btheta^{1,2}; S_\anchor)) \rightarrow (e^1)^{|E(S_\anchor)|} ,$$

and we write $$\cM^{B_\anchor}_\anchor(\btheta^{1,2}; S_\anchor; P_\anchor)$$ for the moduli space $(\ev)^{-1}(\Delta_{P_\anchor})$. If we are further provided with a splicing $\sigma_\anchor$ of the set of punctures of $S_\anchor$ labelled with a jumping chord, we can define the moduli spaces $\cM^{B_\anchor}_\anchor(\btheta^{1,2}; S_\anchor; P_\anchor; \sigma_\anchor)$ for the subset  $\ev^{-1}(\Delta_{\sigma_\anchor}^1, \dots, \Delta_{\sigma_\anchor}^h)$, where $\ev:\cM^{B_\anchor}_\anchor(\btheta^{1,2}; S_\anchor; P_\anchor) \rightarrow e^1$.

It transpires that for particularly simple sources $S_\anchor$, the moduli spaces $\cM^{B_\anchor}_\anchor(\btheta^{1,2}; S)$ are always transversally cut out:
\begin{prop}\label{transversalityanchor}Suppose $S_\anchor$ is a disjoint union of topological discs. Then the moduli spaces $\cM_\anchor(\btheta^{1,2}, S_\anchor)$ are always transversally cut out by the $\delbar$-equations, for any choice of almost-complex structure $J_\anchor$.
\end{prop}
\begin{proof}This follows along the lines of \cite[Prop 5.16]{LOT} --- namely, any $u = (u_{\Sigma}, u_{\HH})$ doubles to a pair of maps from a punctured $CP^1$ to a punctured $CP^1$. The first Chern class of these doubled maps is equal to the Maslov index of $u_{\Sigma}$ and $u_{\HH}$, which is $-1$ in both cases (it is straightforward to see that each is a small anti-clockwise disc). Thus by \cite[Lemma 3.3.1]{McDS:jhol}, $D\delbar$ is surjective for this doubled map, which implies surjectivity for the original map --- see, for instance, \cite[p.158]{hoferlizansikorav}.
\end{proof}

These anchors will arise as the result of `stretching the neck' of bordered holomorphic polygons. Explicitly, for every $m$ fix a $\varepsilon$-radius half-disc neighbourhood $\eta_{\varepsilon}^m$  of the point $\{0\} \in e^1 \subset \partial D_m$. We consider the family $\bJ^t$ obtained from $\bJ$ by inserting a neck of length $t$ along the boundary of each $\eta_{\varepsilon}^m$, as described in \cite[Section A.2]{Lipshitz:cylindrical}.

Suppose that $\ind(B, S, \arrowrho_\boat, \arrowrho_\anchor, \sigma_{\boat}) = 0$, and let $$\pmb{\cM}_{\ge T}^B(\bx; S; \arrowrho_{\boat}, \arrowrho_\anchor; \vec\sigma_\boat, \vec\sigma_\anchor; \bJ)$$ denote the set $\bigcup_{t \ge T} \cM^B(\vec\bx; S; \arrowrho_{\boat}, \arrowrho_\anchor; \vec\sigma_\boat, \vec\sigma_\anchor; \bJ^t)$. By a standard transversality argument, for generic choice of $\bJ$ and generic $T$, the moduli spaces $${\pmb{\cM}}^B_{\ge T}: = {\pmb{\cM}}^B_{\ge T}(\vec\bx; S; \arrowrho_{\boat}, \arrowrho_\anchor; \vec\sigma_\boat, \vec\sigma_\anchor; \bJ^t)$$ are one-dimensional manifolds. Moreover, for small enough $\varepsilon$, the moduli spaces
$${\pmb{\cM}}^{B_{\bar{e}}}_{\ge T, \bar{e}, \varepsilon}(\vec\bx; S_{\bar{e}}; \arrowrho_{\boat}, \arrowrho_\anchor; \vec\sigma_\boat, \vec\sigma_\anchor; \bJ^t)$$
defined to be equal to $\bigcup_{t \ge T} \cM^{B_{\bar{e}}}_{\bar{e}, \varepsilon}(\vec\bx; S_{\bar{e}}; \arrowrho_{\boat}, \arrowrho_\anchor; \vec\sigma_\boat, \vec\sigma_\anchor; \bJ^t)$
are also transversally cut out, and form one-manifolds homeomorphic to ${\pmb{\cM}}^B_{\ge T}$ (cf. Proposition \ref{perturbingidentification}). By Gromov compactness, in both instances, there is some $T$ such that for every $t>T$, each constituent moduli space $\cM^B(\vec\bx; S; \arrowrho_{\boat}, \arrowrho_\anchor; \vec\sigma_\boat, \vec\sigma_\anchor; \bJ^t)$ (resp. $ \cM^{B_{\bar{e}}}_{\bar{e}, \varepsilon}(\vec\bx; S_{\bar{e}}; \arrowrho_{\boat}, \arrowrho_\anchor; \vec\sigma_\boat, \vec\sigma_\anchor; \bJ^t)$) is transversally cut out.

The perturbed moduli spaces admit compactifications similar to the aforementioned trees of holomorphic polygons: but there is one more possible type of degeneration which we force in the neck-stretching process. This is an \emph{anchored holomorphic curve}:

\begin{defn}An \emph{anchored holomorphic polygon} is a pair $(u, u_\anchor)$ where
$$u \in \cM^B(\vec\bx,\btheta^{0,1}, \btheta^{1,2}, \dots, \btheta^{m-1,0}, \by; S; \arrowrho_\boat, \arrowrho_\anchor; \sigma_\boat, \sigma_\anchor; \bJ)$$
is a bordered holomorphic polygon and $u_\anchor \in \cM_\anchor(\btheta^{1,2}; S_\anchor; P; \sigma_\anchor)$ is a holomorphic anchor. (See Figure \ref{fig:anchoredcurve} for a schematic.)
\end{defn}

\begin{figure}[h]
\centering{
\makebox[\textwidth]{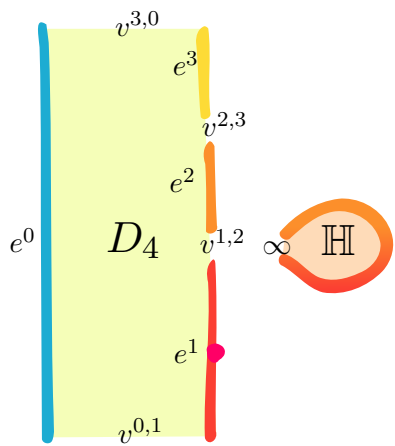}
\caption{A schematic of an anchored holomorphic polygon. The pink dot indicates $\pi_D \circ u$ of an east puncture of the source $S$.}
\label{fig:anchoredcurve}
}
\end{figure}

We write $$\overline{{\pmb{\cM}}}^{B_{\bar{e}}}_{\ge T, \bar{e}, \varepsilon}(\vec\bx; S_{\bar{e}}; \arrowrho_{\boat}, \arrowrho_\anchor; \vec\sigma_\boat, \vec\sigma_\anchor; \bJ^t)$$
for the compactification of$${\pmb{\cM}}^{B_{\bar{e}}}_{\ge T, \bar{e}, \varepsilon}(\vec\bx; S_{\bar{e}}; \arrowrho_{\boat}, \arrowrho_\anchor; \vec\sigma_\boat, \vec\sigma_\anchor; \bJ^t),$$ which can be defined in much the same way as in Defintion \ref{modulicompactifications}, but allowing for the bordered polygons in each definition to be anchored holomorphic polygons.
 
\subsection{Gluing for trees of elementary holomorphic combs and anchors}
We now quote the relevant gluing results, and then explore their consequences.

\begin{prop}\label{gluingpolys}Let $U = (u_1, u_2)$ be a height 2 tree of (bordered) holomorphic polygons where $u_1 \in \tildemod^{B_1}(\vec\bx(V_1), S_1; P_1, Q_1;\sigma_1)$ and $u_2 \in \tildemod^{B_2}(\vec\bx(V_2), S_2; P_2, Q_2; \sigma_2)$ are spinal.

Let $B = B_1 * B_2$, $S$ denote the splayed source $S_T= S_1\natural S_2$, with induced distinguished set of non-easterly punctures $Q = Q_1 \cup Q_2$, and finally let $\sigma = \sigma_1 \star \sigma_2$.

For some sufficiently small open neighbourhoods $U_1$ of $u_1$ and $U_2$ of $u_2$, there is an open neighbourhood of $(u_1, u_2)$ in $\overline{{\cM}}^{B}(\vec\bx(T), S; P_1 \star P_2, Q; \sigma)$ which is homeomorphic to $U_1 \times U_2 \times [0,1)$.\end{prop}

We also require gluing for two-story combs with one easterly vertex:
\begin{prop}\label{gluingeast}Let $U = (u_1, u_2)$ be a height 2 tree of (bordered) holomorphic polygons where $u_1 \in \tildemod^{B_1}(\vec\bx(V_1), S_1; P_1, Q;\sigma_1)$ is spinal and $u_2 \in \tildemod^{B_2}(\vec\bx(V_2), S_2; P_2; \sigma_2)$ is easterly.

Let $B = B_1 * B_2$, $S$ denote the splayed source $S_T=S_1\natural S_2$, with induced distinguished set of non-easterly punctures $Q'$, and finally let $\sigma = \sigma_1 \star \sigma_2$.

For some sufficiently small open neighbourhoods $U_1$ of $u_1$ and $U_2$ of $u_2$, there is an open neighbourhood of $(u_1, u_2)$ in $\overline{{\cM}}^{B}(\vec\bx(T), S; P_1 \star P_2, Q; \sigma)$ which is homeomorphic to $U_1 \times U_2 \times [0,1)$.\end{prop}

We also require gluing for simple combs:
\begin{prop}\label{gluingcomb}Let $(u,v)$ be a simple $m$-comb, with $u \in  \tildemod^B(\vec\bx; S; Q; J)$, and $v \in \cN(T, P_e; \sigma)$, and suppose the corresponding moduli spaces are transversally cut out at $u$ and $v$ respectively, and the evaluation maps $\ev: \tildemod^B(\vec\bx; S; Q; J) \rightarrow e^1$ and $\ev_w: \cN(T; P_e; \sigma) \rightarrow e^1$ are transverse at $(u,v)$.

If we denote by $S' = S \natural T$, and $Q'$ the corresponding set of distinguished non-easterly punctures induced by $S$, then for sufficiently small open neighbourhoods $U_u$ of $u$ and $U_v$ of $v$, there is an open neighbourhood of $(u,v)$ in $\overline{\cM}^B(\vec\bx; S'; P, Q; \sigma)$ homeomorphic to $(U_u \times_{\ev} U_v) \times [0,1)$.\end{prop}

Finally, we require the following result for anchored holomorphic curves:
\begin{prop}\label{gluinganchored}Let $u'$ be a splayed holomorphic polygon in $$\cM_{\bar{e}}^{B'}(\bx, \vec\btheta_+(k), \by; S'; P', Q'; \sigma'; \bJ),$$ and $u_\anchor \in  \cM^{B_\anchor}_\anchor(\btheta^{1,2}; S_\anchor; P_\anchor; \sigma_\anchor; J_\anchor)$ a holomorphic anchor, and suppose that these moduli spaces are transversally cut out at $u, u_\anchor$ respectively.

Denote by $B = B' * B_\anchor$, $S$ the surface formed by gluing $S'$ and $S_\anchor$ at punctures labelled by $v^{1,2}$, $P = P' \star P_\anchor$, $\sigma = \sigma' \star \sigma_\anchor$, and $Q$ the set of distinguished non-east punctures of $S$ induced by $Q'$. Then, for sufficiently small open neighbourhoods $U$ and $U_\anchor$ containing $u$ and $u_\anchor$ respectively, there is an open neighbourhood of $(u, u_\anchor)$ in the moduli space $\overline{{\pmb{\cM}}}_{\ge T}(\bx, \vec\btheta_+(k-1), \by; S; P, Q; \sigma)$ homeomorphic to $U \times U_\anchor \times [0,1)$.
\end{prop}

\begin{proof}[Proof of propositions \ref{gluingpolys}, \ref{gluingeast} \ref{gluingcomb}, and \ref{gluinganchored}]A straightforward combination of \cite[Proposition 5.30]{LOT} and \cite[Proposition 4.21]{LOT:SSII} gives the first three of these results when dropping the relevant instances of the letter $Q$ from the statements, and the last follows similarly. In each case, as the punctures of each relevant source are mapped either to east infinity or to some boundary puncture of the relevant disc, the fact that the manifolds $X_{\cD, m}$ have two types of infinity is irrelevant, so the result follows by standard gluing techniques (the provision of splicings behaves no differently than the other partitions of east punctures). To include the provision of cut behaviour, one sees that the nearby families of holomorphic curves constructed by the implicit function theorem in the proof of \cite[Proposition 5.30]{LOT} and \cite[Proposition 4.21]{LOT:SSII} must have small cuts near the relevant punctures themselves --- thus, by the discussion in Section \ref{fixingcutbehaviour} there are nearby families of curves with no cut at these punctures.

For example, for Proposition \ref{gluingpolys}, one constructs pre-glued maps $u_1 \natural_R u_2  $ with sources $S_1 \natural S_2$ which converge to the two-story building. One shows that $D\delbar(u_1 \natural_R u_2) \rightarrow 0$ as $R \rightarrow \infty$, and applies the implicit function theorem to find a neighbourhood $N_\delta = N_\delta(u_1 \natural_R u_2)$ such that the set of zeroes of the $\delbar$-equations in $N_\delta$ form a smooth manifold of the correct dimension. As such, choosing $R$ sufficiently large finds families of solutions to the $\delbar$-equations near to $u_1 \natural_R u_2$.

If $S_1$ and $S_2$ have distinguished sets of punctures $Q_1$ and $Q_2$ respectively, where $Q_1 \cup Q_2 = Q$, it follows that the pre-glued maps $u_1 \natural_R u_2$ have no cut at the corresponding punctures of $S_1 \natural S_2$. As such, any holomorphic curve in the neighbourhood $N_\delta$ has (a priori) at least a very small cut at the relevant punctures. By the discussion in Section \ref{fixingcutbehaviour}, it follows that the set $\delbar^{-1}(0) \cap N_\delta$ can be parameterised by cut length as $\tildemod^B(\vec\bx; S; \arrowrho_\boat, Q; \vec\sigma) \times (-\varepsilon, \varepsilon)^{k}$ for some $k$. It follows by the triangle inequality that the set $\tildemod^B(\vec\bx; S; \arrowrho_\boat, Q; \vec\sigma)$ intersects $B_\delta$, and thus provides a family of holomorphic curves with the correct cut behaviour near to $u_1 \natural_R u_2$ and thus near to $(u_1, u_2)$, as in \cite[Section A.4]{Lipshitz:cylindrical} (cf. also \cite[Section 5.3.3]{bourgeois:morse}).
\end{proof}

\subsection{Restrictions upon easterly curves and the boundary of one-dimensional moduli spaces}\label{easterlyrestrictions}
In this section, we prove the main result which we shall use to show that maps we define in the next section satisfy the partial $\cA_\infty$ relations. We will need to understand the modulo-two count of the ends of moduli spaces of expected dimension 1 in more detail. As in \cite[Definition 5.58]{LOT}, we give names to the types of ends that occur. Some of these are illustrated in Figures \ref{fig:spinal}, \ref{fig:mixed} and \ref{fig:compcoll} --- recall that we distinguished some special types of vertex of a tree in Definition \ref{specialvertices}.

\begin{figure}[h]
\centering{
\makebox[\textwidth]{\scalebox{0.5}{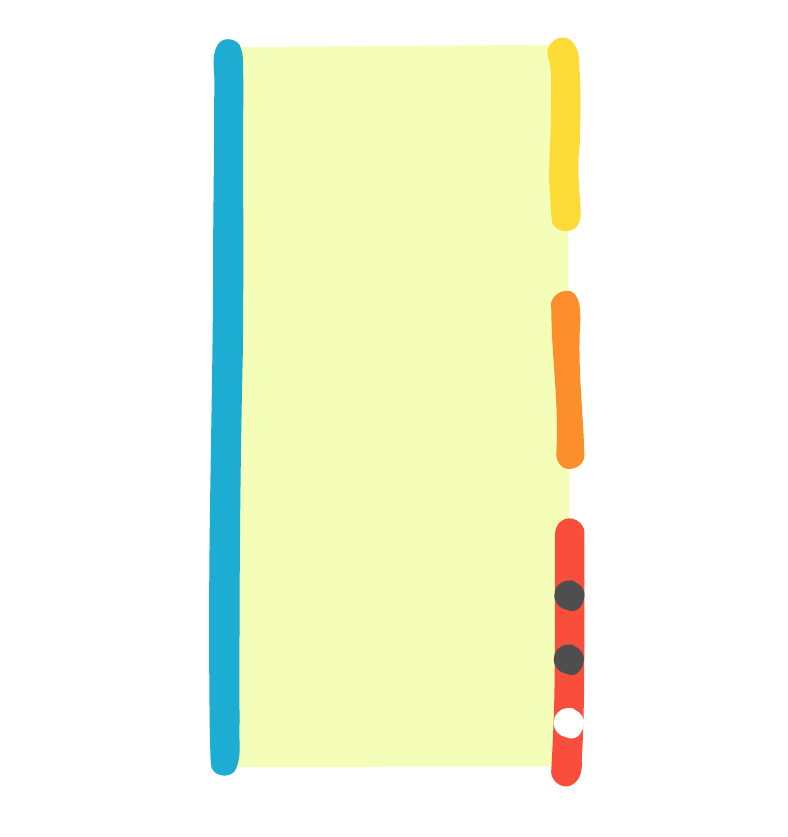 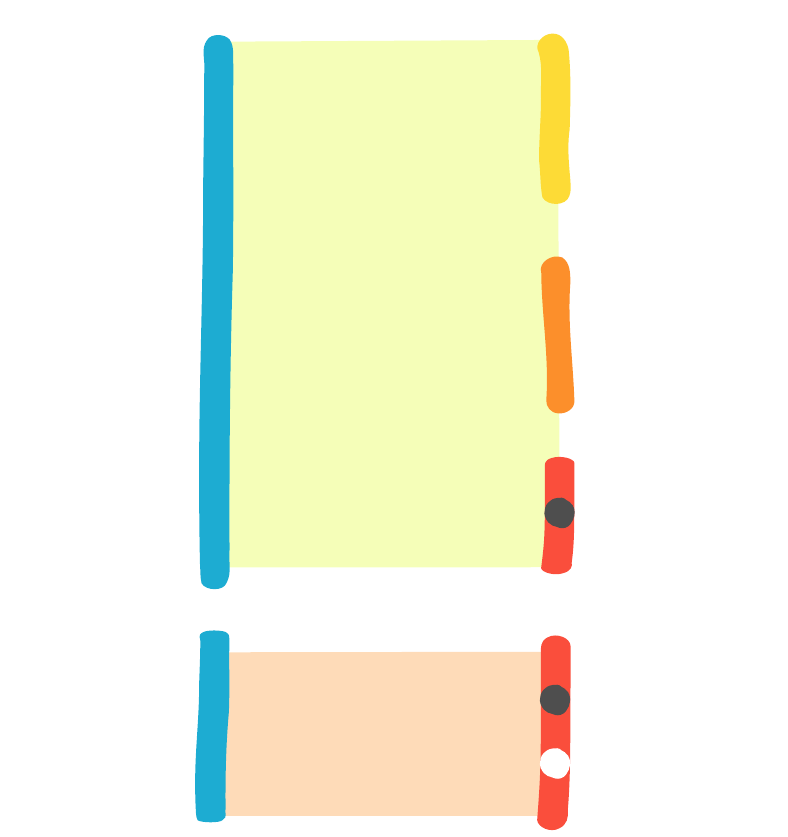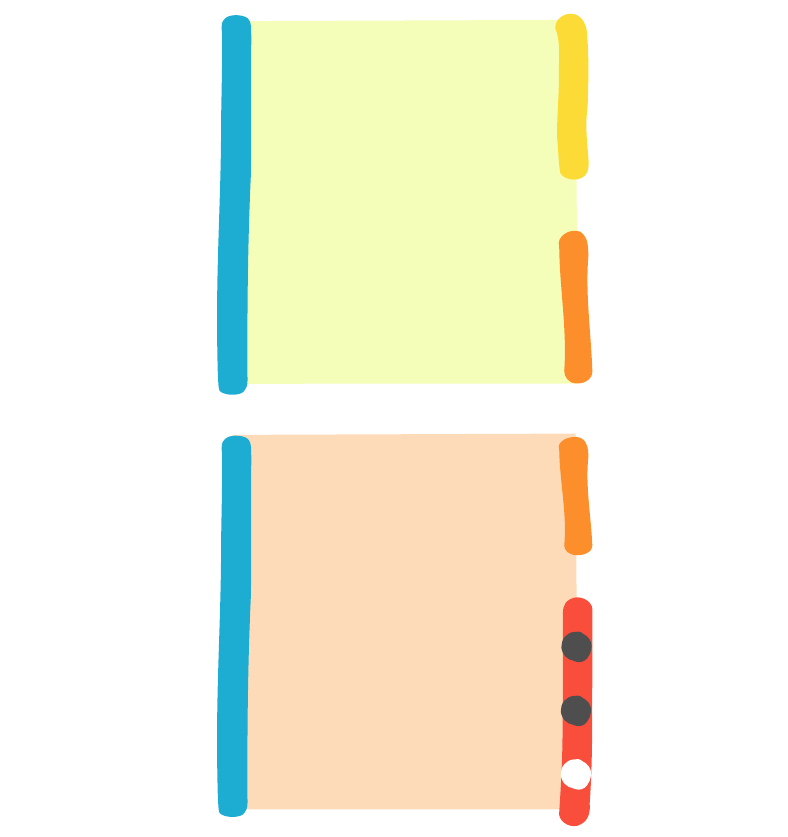}}
\caption{A schematic picture of a one-dimensional moduli space (left), with two types of spinal ends shown (centre and right). The breaks in the boundary indicate punctures of $D_n$, whereas the dots indicate east punctures --- with punctures corresponding to one east infinity shown in grey, and a puncture corresponding to another shown in white.}
\label{fig:spinal}
}
\end{figure}

\begin{figure}[h]
\centering{
\makebox[\textwidth]{\scalebox{0.5}{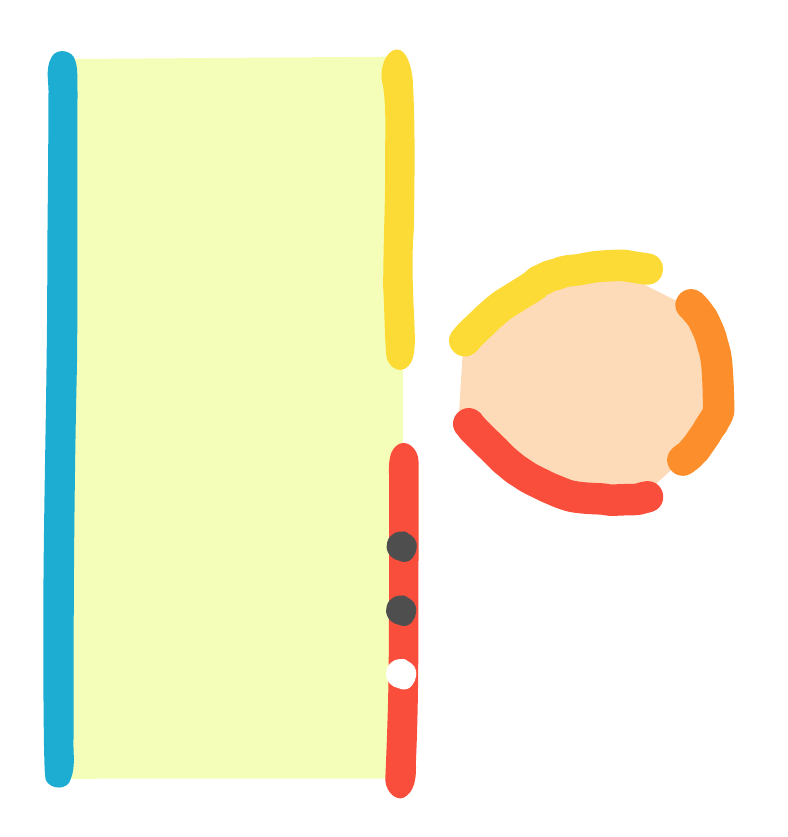 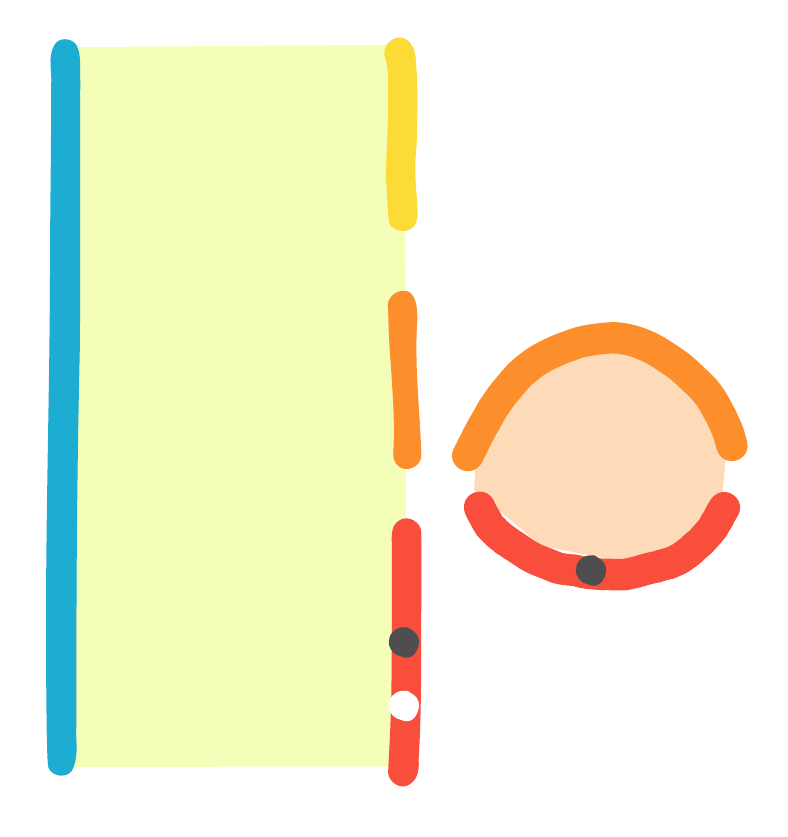}}
\caption{A schematic picture of two types of mixed ends of the one-dimensional moduli space on the left of Figure \ref{fig:spinal}.}
\label{fig:mixed}
}
\end{figure}

\begin{figure}[h]
\centering{
\makebox[\textwidth]{\scalebox{0.5}{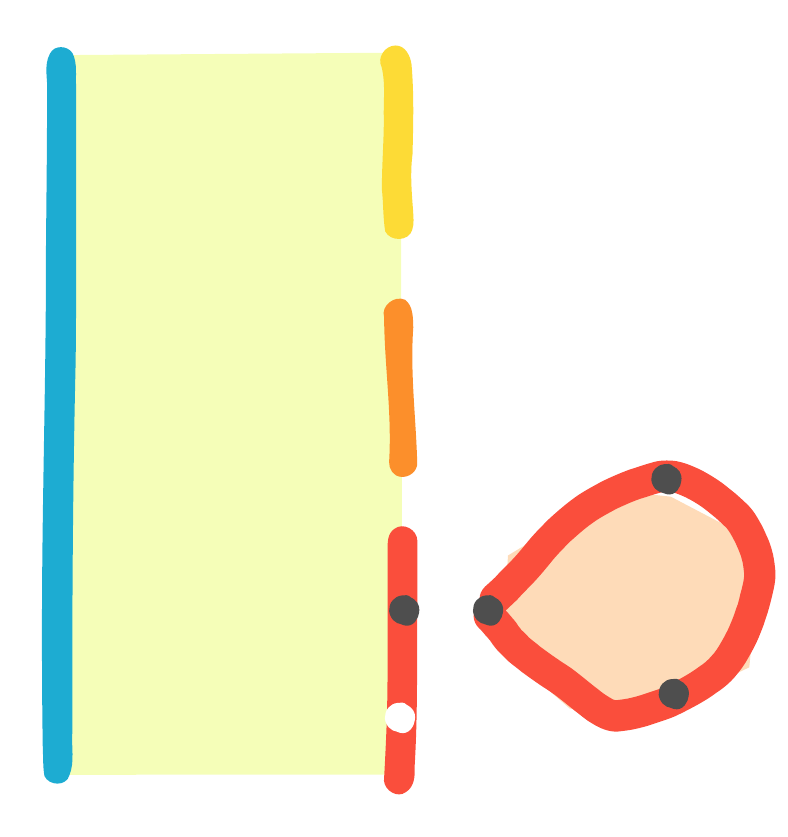 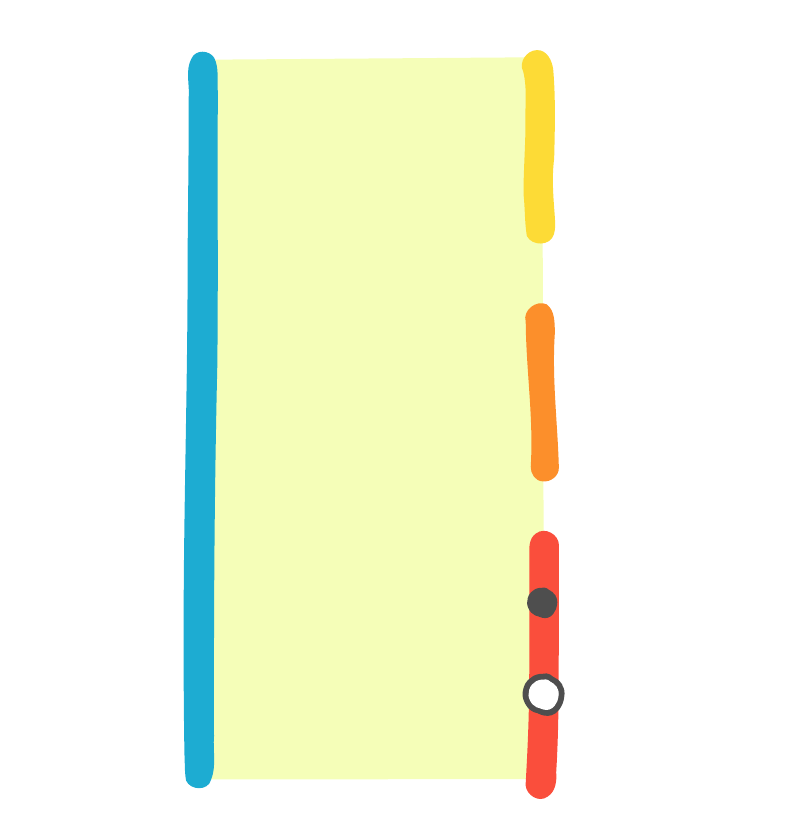}}
\caption{A schematic picture of a composition end (left) and a collision end (right) of the moduli space on the left of Figure \ref{fig:spinal}.}
\label{fig:compcoll}
}
\end{figure}

\begin{defn}\label{possibleends}Let $\cM : = \cM^B(\bx, \vec\btheta_+(k), \by; S; \arrowrho_\boat, \arrowrho_\anchor; \sigma_\boat, \sigma_\anchor)$ be a one-dimensional moduli space. Then:
\begin{enumerate}\item A \emph{spinal two-story end} of $\cM$ is a height two spinal tree of prescribed-cut polygons $(u, v)$, where $$u \in \cM^{B_1}(\bx, \vec\btheta_+, \bw; S_1; \arrowrho_\boat^1, \arrowrho_\anchor^1; \sigma_\boat^1, \sigma_\anchor^1)$$ and $$v \in \cM^{B_2}(\bw, \vec\btheta_+, \by; S_2; \arrowrho_\boat^2, \arrowrho_\anchor^2; \sigma_\boat^2, \sigma_\anchor^2)$$ are spinal, and such that $B_1 \in \pi_2^{\vec\btheta_+}(\bx, \bw)$ and $B_2 \in \pi_2^{\vec\btheta_+}(\bw, \bz)$ are such that $B_1 * B_2 = B$; $S_1 \natural S_2 = S$;  $\arrowrho^1_\boat \star \arrowrho_\boat^2 = \arrowrho_\boat$; $\arrowrho_\anchor^1 \star \arrowrho_\boat^2 = \arrowrho_\anchor$; $\sigma_\boat^1 \star \sigma_\boat^2 = \sigma_\boat$; and $\sigma_\anchor^1 \star \sigma_\anchor^2 = \sigma_\anchor$.

\item A \emph{mixed two-story end} of $\cM$ is a height two tree $(u,v)$ such that $$u \in \cM^{B'}(\bx, \btheta_+^{1,2}, \dots, \btheta_+^{\kappa-1, \kappa}, \btheta^{\kappa, \kappa+\lambda}, \btheta_+^{\kappa+\lambda, \kappa + \lambda+1}, \dots, \btheta_+^{k-1, k}, \by; S'; \arrowrho'_\boat, \arrowrho_\anchor; \sigma'_\boat, \sigma_\anchor)$$ is spinal and $$v \in \cM^{B_E}(\btheta_+^{\kappa, \kappa+1}, \dots, \btheta_+^{\kappa+\lambda-1, \kappa+\lambda}, \btheta^{\kappa+\lambda, \kappa}; S_E; \arrowrho_{E})$$ is easterly, and such that $B' * B_E = B$ and $S' \natural S_E = S$.
\item An \emph{$(i,j)$-composition end} is an element of $\cM^B(\bx, \vec\btheta_+(k), \by; S'; \arrowrho_\boat', \arrowrho_\anchor; \sigma_\boat', \sigma_\anchor)$, where $\arrowrho_{\boat}'$ is given by $\bar{\mu}^i_j(\arrowrho_\boat)$ for some $i, j$; $\sigma'_\boat$ is given by $\sigma^i_j$; and $S'$ is obtained from $S$ by contracting the arc of $\partial S$ connecting the punctures corresponding with $\rho^i_j$ and $\rho^{i+1}_j$, and labelling the corresponding puncture by $\rho^i_j\cdot \rho^{i+1}_j$.

\item A \emph{$\kappa$-collision end} is an element of $\cM^B(\bx, \vec\btheta_+(k), \by; S; \arrowrho_\boat, \arrowrho_\anchor; \sigma_\boat(\kappa), \sigma_\anchor)$.

\item A \emph{cut-vanishing end} of $\cM$ is a prescribed-cut polygon $u$ in the moduli space $\cM^B(\bx, \vec\btheta_+(k), \by; S'; \arrowrho'_\boat, \arrowrho'_\anchor; \sigma'_\boat, \sigma'_\anchor)$, such that $\arrowrho'_\boat \star \arrowrho'_\anchor$ is a $123$-reduction of $\arrowrho_\boat \star \arrowrho_\anchor$ formed by replacing an adjacent pair $\rho_1, \rho_{23}$ or $\rho_{12}, \rho_3$ with $\rho_{123}$, where $\rho_1$ (resp. $\rho_3$) is an element of $\arrowrho_{\anchor}$. Here, $S'$ is given by contracting the arc in $\partial S$ between the punctures corresponding with the $\rho_1, \rho_{23}$ (resp. $\rho_{12}, \rho_3$) and labelling the resultant puncture by $\rho_{123}$.
\end{enumerate}
\end{defn}
(The condition on a cut-vanishing end essentially is the same as $\arrowrho'_\anchor$ being a $123$-reduction of $\arrowrho_\anchor$, except there is a possibility for the first chord of $\arrowrho_{\anchor, j}$ to have been $\rho_3$ and the final chord of $\arrowrho_{\boat, j}$ to have been $\rho_{12}$.)

With this in mind, we state the first main theorem of this section:
\begin{thm}\label{evenends}Suppose that $$\cM = \cM^B(\bx, \vec\btheta_+(k), \by; S; \arrowrho_\boat, \arrowrho_\anchor; \sigma_\boat, \sigma_\anchor)$$ is one-dimensional. Then the total number of spinal two-story ends of $\cM$, mixed two-story ends of $\cM$, $(i, j)$-composition ends of $\cM$, $\kappa$-collision ends of $\cM$, and cut-vanishing ends of $\cM$ is even.
\end{thm}
\begin{proof}This follows readily along the same lines as the proof of \cite[Theorem 5.61]{LOT}. In particular, trees of height greater than two are ruled out in a very similar way to \cite[Proposition 5.43]{LOT}. In our situation, however, there are two types of two-story ends. Height one combs still degenerate, but due to the simplicity of the diagram near the boundary, join or shuffle curves can not appear, just sequences of split curves (corresponding with composition and collision ends of $\cM$). \emph{A priori}, nodal curves can still degenerate, but are ruled out along the lines of \cite[Section 5.6.3]{LOT}.

One may truncate $\overline{\cM}$ near composition and collision ends, as in the proof of  \cite[Theorem 5.61]{LOT} to give a moduli space $\overline{\cM}_{\text{cropped}}$ which is a compact 1-manifold with boundary. This has spinal two-story ends if the complex structure on $D_{k+2}$ degenerates along an arc with one end on the edge $e^0$, mixed two-story ends if an arc without an end upon $e^0$ degenerates, a composition or collision end if a holomorphic tooth degenerates at east infinity, and finally a cut-vanishing end if the length of a cut at some corner of $\rho_{123}(S)$ degenerates to zero. Ends where more than one cut degenerates, or a cut degenerates at the same time as another type of degeneration, are ruled out by the index formula in Proposition \ref{indextransversalitycut} --- each degenerated length subtracts one from the expected dimension, thus if more than one cut degenerated the corresponding moduli spaces would have negative dimension and so be empty.
\end{proof}

We now examine the types of mixed two-story ends which may occur. Fix a partially splayed diagram $\cD = \cD'(\varepsilon, \arrowiota)$ with attaching curves $\bgamma^0, \bgamma^1, \dots, \bgamma^m$, where we recall that $\bgamma^1$ splits as $\bgamma^1 = \bgamma^{1,a} \cup \bgamma^{1,c}$. Fix also a subdiagram $\cD|_S$ where $S$ is a contiguous subsequence of $\{1, \dots, m\}$. For clarity's sake, re-label the attaching curves of this subdiagram by $\bdelta^1, \dots, \bdelta^{\lambda}$, and the corresponding sequence of generators by $(\btheta_+^{\kappa, \kappa+1}, \dots, \btheta_+^{\kappa+\lambda-1, \kappa+\lambda}) := \vec\btheta_+(\lambda)$ and $\btheta = \btheta^{\kappa+\lambda, \kappa} \in \cG(\bdelta^\lambda, \bdelta^1)$. We call the resultant generator for $\cD|_S$ by $\vec\btheta(\lambda) : = (\vec\btheta_+(\lambda), \btheta)$.
We first restrict $B_E$ to some fairly manageable possibilities:

\begin{prop}\label{easterlyinapprox}Let $B_E \in \pi_2(\vec\btheta(\lambda)) \subset \dom(\cD|_S)$. Then it is supported within the approximation region for the curves $\bdelta^1, \dots, \bdelta^\lambda$.
\end{prop}
\begin{proof}This follows from the homological linear independence of the bordered diagram $\border\cD'$. The corresponding domain $B_E' \in \dom(\border\cD')$ given by taking multiplicities away from the approximation region has boundary contained solely in $\balpha$ and multiplicity zero at $\bb$ --- contradicting homological linear independence unless all multiplicities away from the approximation region are zero.
\end{proof}

It follows from this that the domain $B_E$ decomposes as a sum $B_E = B_1 + \cdots + B_{g+h-1}$, where each $B_j$ is supported in the approximation region for $\delta^i_j, \delta^{i+1}_j$.

The next step is to separate the curves into the more-widely-studied case where they are each approximations of one another, and the case where they change in a given index. As such, we let
$$\Gamma : = \{j \in \{1, \dots, h\} : \jump(\arrowrho_{E, j}) = \emptyset \text{\ and\ } \delta^1_j \approx \cdots \approx \delta^{l}_j \} \cup \{h+1, \dots, g+h-1\},$$
and $\overline{\Gamma}: = \{1, \dots, h\} - (\Gamma \cap \{1, \dots, h\})$.

We now show that the domains of the less-widely-studied case are supported within the handles.
\begin{prop}Suppose that $B_E \in \pi_2(\vec\btheta(\lambda))$, and let $j \in \overline{\Gamma}$. Then the domain $B_j$ is supported within the handle $\handle_j$.
\end{prop}
\begin{proof}Any domain $B_j$ not supported within $\handle_j$ but supported within the approximation region can only have corners upon points of the form $\theta^{i_k, i_{k+1}}_{+,j}$ where $\delta_j^{i_k} \approx \delta_j^{i_{k+1}}$. As $j \in \overline{\Gamma}$, either one co-ordinate $\theta(\lambda)^i_j$ is equal to some $e^{\infty, i}_j$, or $B_j$ is required to have a corner at $e^\infty_j$ --- neither of which are of this form.
\end{proof}

For each $j \in \Gamma$, then, if $\delta^{1}$ is a set of $2h$ arcs, any map $u \in \cM^{B_E}(\btheta(\lambda); S_E; \arrowrho_E)$ has boundary in just one of the cylinders $\delta^{1,m}_j \times e^1$ or $\delta^{1,l}_j \times e^1$. We label the corresponding curve $\delta^1_{j}$. Let $\dot{\Sigma_\Gamma}$ denote the surface given by removing the handles $\cup_{j \in \overline{\Gamma}} \handle_j$ from the east-compactified splayed surface $\Sigma(\varepsilon, \arrowiota)_{\bar{e}}$, and $\Sigma_\Gamma$ the result of collapsing the remaining boundary components $\handle_{L,j}$ and $\handle_{R, j}$ in $\dot{\Sigma_\Gamma}$ to $2|\overline{\Gamma}|$ single points $\bp_{\overline{\Gamma}}$. This contains the curves $\bdelta^i_{\Gamma} = \{\delta^i_{\bar{e},j}: j \in \Gamma\},$ which approximate one another and base-points $\bp_{\Gamma}: = \{p_j: j \in \Gamma\}$ and thus we form the closed multi-diagram $\cD_{{\Gamma}}$ given by the tuple $(\Sigma_\Gamma; \bdelta^1_{\Gamma}, \dots, \bdelta^l_{\Gamma}; \bp_{\Gamma} \cup \bp_{\overline{\Gamma}})$.

It follows from the above proposition that for $j \in \Gamma$ the domains $B_{j}$ are contained in $\dot\Sigma_\Gamma$, and thus have a corresponding domain $B_\Gamma \in \dom(\cD_{\Gamma})$, given by the sum of their images in $\Sigma_\Gamma$. Moreover, the domain $B_\Gamma$ lives in the homology group $\pi_2(\vec\btheta(\lambda)|_{\Gamma})$, where $\vec\btheta(\lambda)|_{\Gamma}$ is the restriction of $\vec\btheta(\lambda)$ to the curves $\bdelta^1_{\Gamma}, \dots, \bdelta^\lambda_{\Gamma}$.

 Considering the surface $\dot{\Sigma_\Gamma}$ as a subset of both surfaces $\Sigma(\varepsilon, \arrowiota)_{\bar{e}}$ and $\Sigma_\Gamma$, we choose a family of almost-complex structures $\bJ'$ upon $\Sigma_\Gamma \times D_l$ such that their restriction to $\dot{\Sigma}_\Gamma \times D_l$ agrees with the restriction of $\bJ$.

For every $j \in \overline{\Gamma}$, we identify the handle $\handle_j$ with the punctured torus $\Sigma_j  : = T^2 - e^\infty_j$ formed by gluing the ends $\handle_{L,j}$ and $\handle_{R,j}$ (more properly, we identify the handle with its end circles removed with an open subset of the torus). The curves and arcs $\delta^1_j \cap \handle_j, \dots, \delta^\lambda_j \cap \handle_j$ have natural images in $T^2$, for which we abuse notation and call $\delta^1_j, \dots, \delta^\lambda_j$. Together with the image of the base-point $b_j$, the tuple $\cD_j : = (\Sigma_j; \delta^1_j, \dots, \delta^\lambda_j; b_j)$ forms a toroidal bordered Heegaard multi-diagram. The domain $B_j$ has a natural image in $\dom(\cD_j)$ which we also denote by $B_j$, and lives in the group $\pi_2(\theta(\lambda)^1_j, \dots, \theta(\lambda)^\lambda_j)$. We choose a family of almost-complex structures $\bJ_j$ upon $\Sigma_j \times D_\lambda$ which agree with the almost-complex structure $\bJ$ upon $\handle_j \times D_\lambda$.

It follows that the moduli space $\cM^{B_E}(\vec\btheta(\lambda); S_E; \arrowrho_\boat)$ splits as a fibred product
$$\cM^{B_E}(\vec\btheta(\lambda); S_E; \arrowrho_E) = \cM_\Gamma \times_{\conf(D_\lambda)} \cM_1 \times_{\conf(D_\lambda)} \dots \times_{\conf(D_\lambda)} \cM_{|\overline{\Gamma}|},$$
where $\cM_\Gamma = \cM^{B_\Gamma}(\vec\btheta(\lambda)|_{\Gamma}; S_\Gamma; \bJ')$ is the moduli space defined by counting holomorphic polygons according to $\cD_{\Gamma}$ as defined in Section \ref{polygonmodulispaces}, and each $\cM_j$ is the moduli space $\cM^{B_j}(\theta(\lambda)^{1,2}_j, \dots, \theta(\lambda)^{\lambda,1}_j; S_j; \arrowrho_{\boat, j}; \bJ_j)$ defined by counting bordered polygons according to $\cD_j$, as defined in Section \ref{borderedpolygonmodulispaces}.

\begin{prop}\label{projectionmodulis}There is a unique homology class $D_\Gamma \in \dom(\cD_\Gamma)$ which carries a holomorphic representative in the moduli space $\cM_\Gamma$. This lives in the homology group $\pi_2(\vec\btheta_+(\lambda)|_\Gamma)$. Moreover, the projection of the moduli space $\cM_\Gamma$ onto $\conf(D_\lambda)$ has degree one for this homology class: thus $\cM^{B_E}(\btheta(\lambda); S_E; \arrowrho_\boat)$ is non-empty only if the corresponding $B_{\Gamma} = D_\Gamma$, and $\vec\btheta(\lambda)|_\Gamma = \vec\btheta_+(\lambda)|_\Gamma$. In this case, $\cM^{B_E}(\vec\btheta(\lambda); S_E; \arrowrho_E)$ has the same number of elements as $$\cM_1 \times_{\conf(D_\lambda)} \dots \times_{\conf(D_\lambda)} \cM_{|\overline{\Gamma}|}.$$
\end{prop}
\begin{proof}This is well-known: see, for example, \cite[Lemma 3.50]{LOT:SSII} or \cite[Lemma 6.17]{Manolescu:linksurgeries}.\end{proof}

\begin{prop}\label{triangledomains}Let $B_j, S_j, \arrowrho_{\boat, j}$ be such that $\ind(B_j, S_j, \arrowrho_{\boat,j}) = 2-\lambda$. For each $j$, the moduli space $\cM^{B_j}(\theta(\lambda)^{1,2}_j, \dots, \theta(\lambda)^{\lambda,1}_j; S_j; \arrowrho_{\boat, j}; \bJ_j)$ (of expected dimension zero) is nonzero only if either:\begin{itemize} \item $\lambda = 2$ and $\arrowrho_{\boat,j}$ contains a single jumping chord, or \item$\lambda = 3$ and $\arrowrho_{\boat, j}$ is empty. \end{itemize} In these cases there is only a single domain for which the moduli space is non-empty, and in these cases it contains a unique element.
\end{prop}
\begin{proof}
We assume that $\arrowrho_{\boat,j}$ is empty --- the case where it is non-empty follows similarly, requiring only the minor modification of changing which index formulae are used. Suppose initially that $l>3$.

If, for some (non-cyclically) consecutive $i, i+1$, the curves $\delta_j^i, \delta_j^{i+1}$ are approximations of one another, then by Propositions \ref{gluinganchored} and \cite[Lemma 11.8]{Lipshitz:cylindrical} (cf. Proposition \ref{anchorsawayfromedge} below), there is a suitable holomorphic anchor which may be glued to any curve $u \in \cM^{B_j}(\theta(\lambda)^{1,2}_j, \dots, \theta(\lambda)^{\lambda,1}_j; S_j; \arrowrho_{\boat, j})$ of index $2-\lambda$ to yield an element $u'$ of the moduli space $$\cM' = \cM^{B'_j}(\theta(\lambda)^{1,2}_j, \dots, \theta(\lambda)^{i-1,i}, \theta(l)^{i+1, i+2}, \dots, \theta(\lambda)^{\lambda,1}_j; S'_j; \arrowrho_{\boat, j})$$ where the domain $B_j'$ has Euler measure $e(B_j') = e(B_j) - 1/4$. It follows from the index formula that $$\ind(B_j', S_j', \arrowrho_{\boat,j}) = \frac{3-(\lambda-1)}{2} - \chi(S_j') + 2e(B_j') = \ind(B_j,S_j, \arrowrho_{\boat,j}),$$
which implies that the moduli space $\cM'$ has negative expected dimension. As we are working with a generic family of almost-complex structures, this implies that such a $u'$ (and thus such a $u$) does not exist --- thus if any (non-cyclically) adjacent $\delta_j^i, \delta_j^{i+1}$ approximate one another, the corresponding moduli space is empty.

So we may assume that the family $\delta_j^1, \dots, \delta_j^{\lambda}$ satisfies:
\begin{enumerate}
\item For every $i = 1, \dots, \lambda-1$, $\delta_j^i$ and $\delta_j^{i+1}$ intersect transversally in a single point.
\item For every $i = 1, \dots, \lambda-2$, the curves $\delta_j^i, \delta_j^{i+2}$ are small approximations of one another.
\item If $\lambda$ is even, then the curves $\delta^\lambda_j, \delta^1_j$ intersect transversally in a single point; if $\lambda$ is odd, then the curves $\delta^\lambda_j, \delta^1_j$ are small approximations of one another.
\end{enumerate}

\begin{figure}[h]
\centering{
\makebox[\textwidth]{\scalebox{0.7}{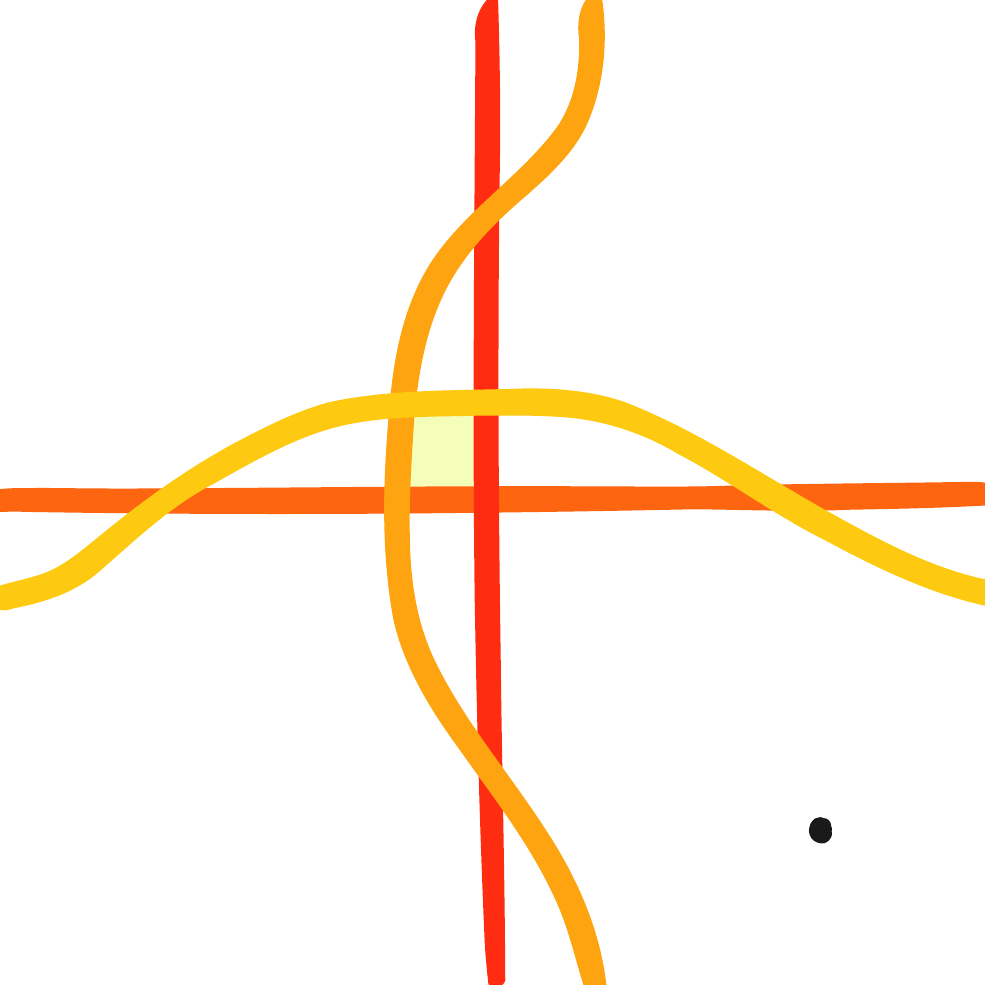}\quad \quad \scalebox{0.7}{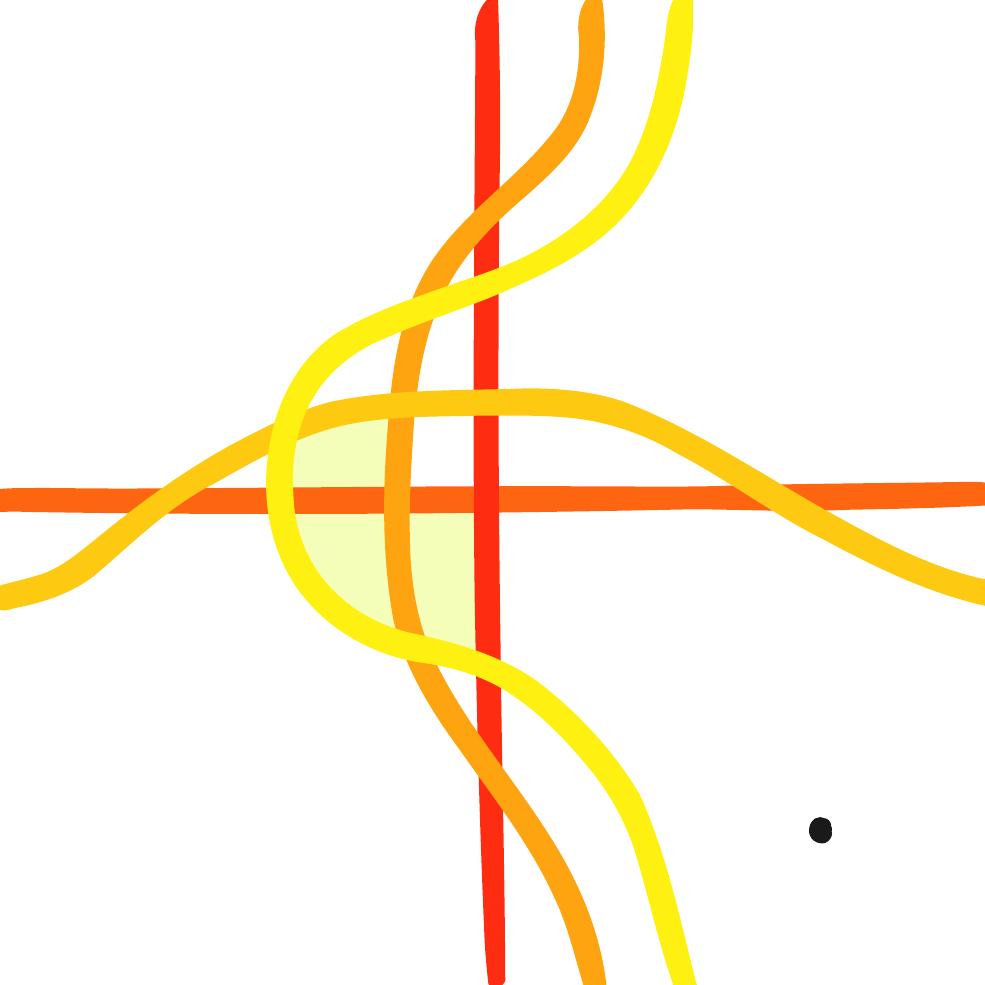}} \caption{Ruling out $\lambda > 3$. On the left is shown the case where $l=4$, and the only possible domain connecting the intersection points shown has the incorrect orientation. On the right is shown the case where $\lambda = 5$: the domain shown has index $-3$, thus the corresponding moduli space is of negative expected dimension.}
\label{fig:rulingouteasterly}
}
\end{figure}

If $\lambda$ is even, then it follows that $\theta_j \in \delta^l_j \cap \delta^1_j$ is uniquely determined and it is straightforward to see that the corresponding homology group is empty (see Figure \ref{fig:rulingouteasterly} for a picture when $l = 4$).

If $l > 3$ is odd, then $\theta_j$ is also uniquely determined for non-empty moduli spaces:  if $\theta_j = \theta_{j,+}$ then a similar gluing argument to above shows that the corresponding moduli space is empty. We can see that for general odd $l$, there is a unique element of $\pi_2(\theta(\lambda)^{1,2}_j, \dots, \theta(\lambda)^{\lambda,1}_j)$ --- an embedded $\lambda$-gon with $\frac{\lambda+3}{2}$ acute corners and $\frac{\lambda - 3}{2}$ obtuse cotners. It follows from the index formula that the expected dimension of $\cM_{\bar{e}}^{B_j}((\theta(\lambda)^{1,2}_j, \dots, \theta(l)^{\lambda,1}_j; S_j)$ is equal to $\frac{3-\lambda}{2}$. As such, this domain has index $3-\lambda$ if and only if $\lambda = 3$, as claimed in the proposition. (See Figure \ref{fig:rulingouteasterly} for the case where $\lambda = 5$.)

There are six different possible configurations for the curves, with corresponding unique domain --- as shown in Figure \ref{fig:eastpieces}.

\begin{figure}[h]
\centering{
\makebox[\textwidth]{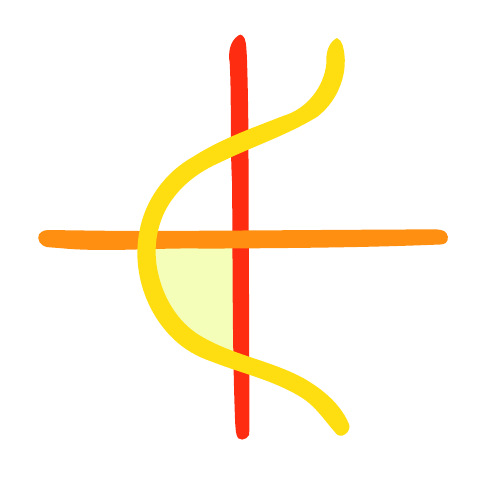\quad 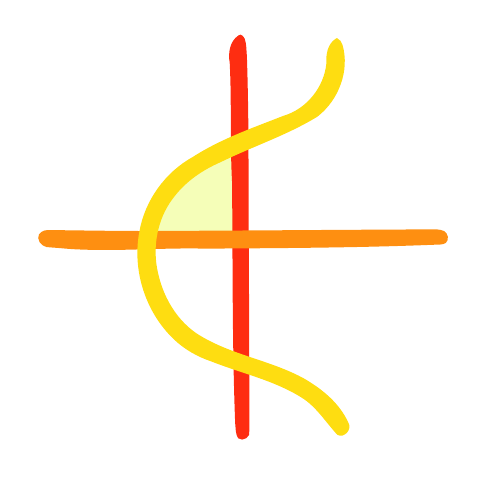 \quad 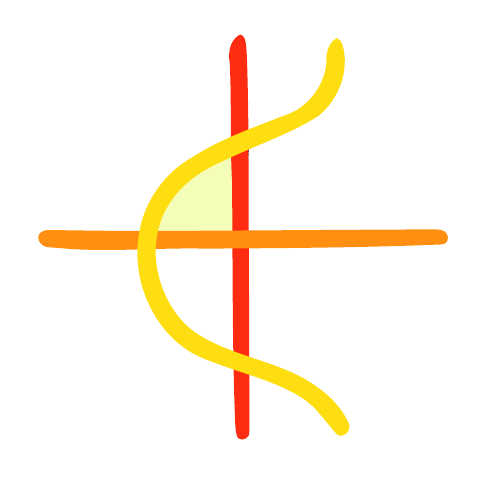 } \caption{Three possible configurations for the curves, and the corresponding domains are shown (assuming that $l = 3$) --- the other three are obtained by rotating these diagrams by $\pi/2$ clockwise. On the left is a \emph{$(\rho_1, \rho_2)$-composition piece}, and on the right are two \emph{collision pieces}, a \emph{pre-collision} and \emph{post-collision} respectively. Rotating these gives a \emph{$(\rho_2, \rho_3)$-composition piece} and two collision pieces respectively.}
\label{fig:eastpieces}
}
\end{figure}

\end{proof}

We give things names:

\begin{defn}We separate the possible domains near $e^\infty_j$ described in Proposition \ref{triangledomains} into types, as indicated in Figure \ref{fig:eastpieces}: \emph{$(\rho_1, \rho_2)$-composition pieces},  \emph{$(\rho_2, \rho_3)$-composition pieces} and \emph{collision pieces}.

Suppose that $\lambda = 3$ and $B_E \in \pi_2(\vec\btheta(\lambda))$ is a domain with no Reeb chords. We say $B_E$ is a \emph{$(\rho_1, \rho_2, j)$-composition domain} (resp. \emph{$(\rho_2, \rho_3, j)$-composition domain}) if the set $\overline{\Gamma}$ has two elements $j, j'$, and $B_j$ is a $(\rho_1, \rho_2)$-composition piece (resp. $(\rho_2, \rho_3)$-composition piece) and $B_{j'}$ is a collision piece.

Similarly, we say $B_E$ is a \emph{collision domain} if $\overline{\Gamma}$ has two elements $j, j'$ and $B_j$ and $B_{j'}$ are both collision pieces.

A similar definition may be made for the case $\lambda=2$, where the corresponding domain has at most one jumping Reeb chord upon each boundary component.
\end{defn}
It follows almost immediately that these are all of the possibilities for easterly domains:
\begin{prop}\label{easterlyoptions}Suppose $B_E$ is such that $\cM^{B_E}(\vec\btheta(\lambda); S; \arrowrho_\boat)$ is non-empty. Then $\lambda = 3$ or $\lambda = 2$, and either:
\begin{itemize}\item $B_E$ is a $(\rho_1, \rho_2, j)$-composition domain for some $j$, and $\theta^{\kappa, \kappa+\lambda}_j = \theta^{\kappa,\kappa+\lambda}_{-, j}$ and $\theta^{\kappa,\kappa+\lambda}_{j+1} = \theta^{\kappa,\kappa+\lambda}_{+, j+1}$, or
\item$B_E$ is a $(\rho_2, \rho_3, j)$-composition domain for some $j$, and $\theta^{\kappa,\kappa+\lambda}_j = \theta^{\kappa,\kappa+\lambda}_{-, j}$ and $\theta^{\kappa,\kappa+\lambda}_{j+1} = \theta^{\kappa,\kappa+\lambda}_{+, j+1}$, or
\item$B_E$ is a collision domain, and we have that $\theta^{\kappa,\kappa+\lambda}_j = \theta^{\kappa,\kappa+\lambda}_{+, j}$ and $\theta^{\kappa,\kappa+\lambda}_{j+1} = \theta^{\kappa,\kappa+\lambda}_{+, j+1}$ for some $j$.
\end{itemize}
(and similar for $l=2$).
In these cases, $\cM^{B_E}(\vec\btheta(\lambda); S; \arrowrho_\boat)$ contains a unique holomorphic curve.
\end{prop}
\begin{proof}By Proposition \ref{projectionmodulis}, the domain is uniquely determined by its restriction to handles $\handle_j$ where $j \in \overline{\Gamma}$, and by Proposition \ref{triangledomains} we know that $\lambda = 3$ or $\lambda = 2$.

As we assumed $\sigma_\boat \star \sigma_\anchor$ was interleaved, the set $\overline{\Gamma}$ having more than two elements is impossible: only one curve in $\bdelta^2$ is not an approximation of one in $\bdelta^1$, and only one curve in $\bdelta^3$ is not an approximation of one in $\bdelta^2$. As such, we only need check the possibilities for $B_j$ near each element $j, j'$ of $\overline{\Gamma}$. There are two cases which may occur for this: either (without loss of generality), $B_j$ is a composition piece and $B_{j'}$ is forced to be a collision piece (else too many curves from $\bdelta^1$ are not approximations of curves from $\bdelta^2$); or $B_j$ is a pre-collision piece, forcing $B_{j'}$ to be a post-collision piece (for the same reason). The case where one piece is a composition piece splits depending upon whether $\delta^1_j$ is meridional or longitudinal.

A similar argument applies for the case where $\lambda=2$.
\end{proof}

We use this to give names to types of mixed two-story trees of holomorphic curves:
\begin{defn}\label{mixednames}Let $$\cM : = \cM^B(\bx, \vec\btheta_+(k), \by; S; \arrowrho_\boat, \arrowrho_\anchor; \sigma_\boat, \sigma_\anchor)$$ be a one-dimensional moduli space, and $U = (u,v)$ be a height two tree of prescribed-cut polygons in a mixed end of the boundary of $\overline{\cM}$, where $$u \in \cM^{B'}(\bx, \vec\btheta_+(k-\lambda), \bw; S'; \arrowrho_\boat, \arrowrho'_\anchor; \sigma_\boat, \sigma'_\anchor)$$ is spinal and $v \in \cM^{B_E}(\vec\btheta(
\lambda); S_E; \arrowrho_{E})$ is easterly.

We shall say that $U$ is:
\begin{itemize}\item an $(\kappa,j)$-composition end if $B_E$ is a $(\rho_1, \rho_2, j)$-composition or $(\rho_2, \rho_3,j)$-composition domain; or
\item a $\kappa$-collision end if $B_E$ is a collision domain.
\end{itemize}
\end{defn}
By Proposition \ref{easterlyoptions}, any mixed $U \in \partial M$ is either an $(i,j)$-composition end or a $k$-collision end.

We conclude this section by identifying composition and collision ends with recognisable moduli spaces:
\begin{prop}\label{compcollidentifications}Suppose that
\begin{align*}\cM'  &:= \cM^{B'}(\bx, \btheta_+^{1,2}, \dots, \btheta_+^{\kappa-1, \kappa}, \btheta^{\kappa, \kappa+3}, \btheta_+^{\kappa+3, \kappa + 4}, \dots, \btheta_+^{k-1, k}, \by; S'; \arrowrho_\boat, \arrowrho'_\anchor; \sigma_\boat, \sigma'_\anchor)\\
\cM_E  &:=\cM^{B_E}(\btheta_+^{\kappa, \kappa+1}, \btheta_+^{\kappa+1, \kappa+2}, \btheta^{\kappa+2, \kappa}; S_E; \arrowrho_{E}),\end{align*}
 are such that $\cM' \times \cM_E \subset \partial \overline{\cM}$. 

If $B_E$ is a $(\rho_1, \rho_2, j)$-composition domain (resp. $(\rho_2, \rho_3,j)$-composition domain), then $\cM' \times \cM_E$ contains the same number of elements modulo two as the moduli space  $$\cM_{\text{comp}}: =\cM^{B'}(\bx, \vec\btheta_+(k-1), \by; S'; \arrowrho_\boat, \bar{\mu}^\kappa_j(\arrowrho_\anchor); \sigma_\boat; \sigma_\anchor, (\sigma_\boat)^\kappa_j).$$

Similarly, if $B_E$ is a collision domain, then the moduli space $\cM' \times \cM_E$ contains the same number of elements modulo two as
$$\cM_{\text{coll}}:=\cM^{B'}(\bx, \vec\btheta_+(k-1), \by; S'; \arrowrho_\boat, \arrowrho_\anchor; \sigma_\boat, \sigma_\anchor(\kappa)).$$
\end{prop}
\begin{proof}By Proposition \ref{easterlyoptions}, the projection of each $\cM' \times \cM_E$ onto $\cM'$ has degree one, hence we only have to show is that, in each case, $\cM'$ has the same number of elements as the moduli spaces in question.

By Proposition \ref{regularisationidentification}, as the diagram $\cD_{\text{comp}}$ used for the definition of $\cM_{\text{comp}}$ (resp. $\cD_{\text{coll}}$ and $\cM_{\text{coll}}$) is the regularisation of the diagram $\cD'$ used in the definition of $\cM'$, we know that if $B_E$ is a composition domain, $\cM'$ is identified with a subset of the moduli space 
$$\cM^{B'}(\bx, \btheta_+^{1,2}, \dots, \btheta_+^{\kappa-1, \kappa}, \btheta_+^{\kappa+3, \kappa + 4}, \dots, \btheta_+^{k-1, k}, \by; S'; \arrowrho_\boat, Q; (\sigma_\boat)^\kappa_j),$$
 and it remains to be seen that this subset is precisely  $\cM_{\text{comp}}$.
 
 Assume first that $B_E$ is a $(\rho_1, \rho_2,j)$-composition domain, and suppose that for each $j' = 1, \dots, h$, the sequences
 $$R_{j'}:= (D_{j'}^1,B_{j'}^1,D_{j'}^2,B_{j'}^2, \dots,D_{j'}^k, B_{j'}^k , D_{j'}^{k+1})$$
 form a sequence of Reeb domains for $B$. For every $j' \ne j$, the corresponding domain $B'_{j'}$ (as discussed in Section \ref{splayingdomains}) is equal to $B_{j'}$ --- it follows that in this case $R_{j'}$ is also a sequence of Reeb domains for $B_j$. For $j' = j$, we have that $B_{j}^\kappa$ decomposes as the sum $C_j^{\kappa} + \hat{B}_j^\kappa$, where $C_j^\kappa$ is a $(\rho_1, \rho_2)$-composition domain. It follows that, by definition, $B_j^\kappa = B_j^\kappa(\rho_1)$, and that $\hat{B}_j^\kappa =  B_j^\kappa(\rho_1) - C_j^\kappa$ has multiplicity one in the south-west region adjacent to $e^{\infty, \kappa+1}_j$.
 
The domain $B'$ has no corner at the intersection between $\gamma^{\kappa}_j$ and $\gamma^{\kappa+1}_j$ --- it has the same multiplicity upon either side of $\gamma_j^{\kappa+1}$. As such, $\hat{B}_j^\kappa + B_j^{\kappa+1}$ has no corners at $e^{\infty, \kappa+1}_j$ and so $ B_j^{\kappa+1} = B_j^{\kappa+1}(\rho_2)$ --- and we have that $D_j^{\kappa+1} = D_j^{\kappa+1}(\emptyset)$. It follows (by comparing multiplicities away from the approximation region between $\gamma^{\kappa}_j$ and $\gamma^{\kappa+2}_j$) that if $D^i_j := n_i B^i(\rho_i)$, $B^i_j : = B^i(\rho'_i)$ and we denote by
$$E^i_j: = \left\{\begin{array}{r l}n_i B^i(\rho_i), & \text{ for } i = 1, \dots, \kappa-1 \\
(n_{\kappa} + n_{\kappa+2} + 1) B^i(\rho_i), & \text{ for } i = \kappa\\
n_{i+2}  B^i(\rho_{i+2}), & \text{ for } i = \kappa + 1, \dots, k-1 \end{array} \right. $$
and
$$F^i_j: = \left\{\begin{array}{r l}B^i(\rho'_i), & \text{ for } i = 1, \dots, \kappa-1 \\
 B^i(\rho'_{i+2}), & \text{ for } i = \kappa, \dots, k-2 \end{array} \right. ,$$ 
then
$$(E_j^1,F_j^1,E_j^2, \dots, E_j^{k-2}, F_j^{k-2}, E_j^{k-1})$$
is a sequence of Reeb domains for $B'_j$ and the result follows. Similar holds for $B_E$ a $(\rho_2, \rho_3, j)$-composition domain.

If $B_E$ is a collision domain then it is completely supported within the approximation regions between $\bgamma^l, \bgamma_k^{l+1}$. As such, subtracting the domain $B_E$ has no effect upon the sequence $\arrowrho_\anchor(B)$, and thus $\arrowrho_\anchor(B'') = \arrowrho_\anchor(B)$. The jumping chords previously in partitions $P^{\kappa}_j$ and $P^{\kappa+1}_{j'}$ are now determined by corners at the copies of east infinity $e^{\kappa}_j, e^{\kappa}_{j'}$, and thus the corresponding partition is given by $\sigma_\anchor(\kappa)$, as required.
\end{proof}

\subsection{Behaviour of holomorphic anchors and stretching the neck}

In this section, we will try and understand the moduli spaces $\cM^{B_\anchor}(\btheta^{1,2}; S_\anchor)$.
\begin{prop}\label{anchorsaresimple}Suppose that $(u_n) \in  {\pmb{\cM}}^{B_{\bar{e}}}_{\ge T, \bar{e}, \varepsilon}(\vec\bx; S_{\bar{e}}; P, Q ; \sigma_\boat, \sigma_\anchor; \bJ^t),$ is a sequence of splayed holomorphic polygons, converging to some two-story building $$(u', u_\anchor) \in \cM^{B'}(\bx, \btheta^{1,2}, \dots, \btheta^{m-1,0}, \by; S'; P', Q', \sigma_\boat', \vec\sigma_\anchor'; \bJ) \times \cM^{B_\anchor}_\anchor(\vec\btheta(m_\anchor); S_\anchor; P_\anchor, \sigma).$$
Then $u_\anchor$ is a holomorphic anchor.
\end{prop}
\begin{proof}For notational convenience, we let $g' = g+h-1$. We note that $\chi(S') + \chi(S_\anchor) - g' = \chi(S)$ and $m_\anchor + m' - 1 = m$, so $5-m = 6 - (m_\anchor + m')$. Denoting by $e(B_\anchor) := e(B) - e(B')$, $\col_\anchor:= \col(\sigma_\boat) - \col(\sigma_\boat')$ and $\#_\anchor := |Q| - |Q'|$, we have that

\begin{align*}\ind(B,S,P, Q, \sigma_\boat) =& \frac{3-m}{2}g' - \chi(S) + 2e(B) + |P| - \col(\sigma_\boat) - |Q|\\
=& \frac{3-m}{2}g' -(\chi(S') + \chi(S_\anchor) - g') + 2e(B) + |P| - \col(\sigma_\boat) - |Q|\\
=& \frac{5-m}{2}g' -\chi(S') - \chi(S_\anchor) + 2e(B) + |P| - \col(\sigma_\boat) - |Q|\\
=& \frac{3-m'}{2}g' -\chi(S') + 2e(B') + |P'| - \col(\sigma_\boat') -|Q'|\\
&+ \frac{3-m_\anchor}{2}g' - \chi(S_\anchor) + 2e(B_\anchor) + |P_\anchor| - \col_\anchor - \#_\anchor\\
=&\ind(B', S', P', Q', \sigma_\boat')\\
&+ \frac{3-m_\anchor}{2}g' - \chi(S_\anchor) + 2e(B_\anchor) + |P_\anchor| - \col_\anchor - \#_\anchor
\end{align*}
Denote by $\ind'$ the first term in this final sum, and $\ind_\anchor$ the remainder. When $u_\anchor$ is a non-simple anchor, by Proposition \ref{perturbingidentification} there is a corresponding bordered holomorphic polygon $u_\anchor'$ which lives in a moduli space of  curves of index $\ind_\anchor$.

It follows that $2-m - \ind' = \ind_\anchor$. Noting that $\ind' \ge 2-m'$, it follows that $\ind_\anchor \le 2-m - (2-m') = m' - m = 1 - m_\anchor$, and so $u_\anchor'$ lives in a moduli space of negative expected dimension, so thus does not exist --- and nor does $u_\anchor$. Therefore, $u_\anchor$ is a simple holomorphic anchor as required.\end{proof}

\begin{prop}
The domain $B_\anchor$ is supported within the approximation region for the curves, and $B'$ splays $B$.
\end{prop}
\begin{proof}The first half of the statement follows along similar lines to Proposition \ref{easterlyinapprox}. If we fix points away from the approximation region between $\bgamma^1$ and $\bgamma^2$ in the diagram corresponding with the moduli space for $u'$, then each curve $u_n$ has the same multiplicity at these points and thus so does $u'$ --- so $B'$ has the same multiplicity as $B$ everywhere away from this approximation region, which is precisely the definition of $B'$ splaying $B$.\end{proof}

We need a lemma about the acute corners of anchored curves:
\begin{lem}Suppose that  $$(u', u_\anchor) \in \cM^{B'}(\bx, \btheta^{1,2}, \dots, \btheta^{m-1,0}, \by; S'; P', Q', \sigma_\boat', \vec\sigma_\anchor'; \bJ) \times \cM^{B_\anchor}_\anchor(\btheta^{1,2}; S_\anchor; P_\anchor, \sigma)$$ is an anchored holomorphic curve in the boundary of the moduli space $$ {\pmb{\cM}}^{B_{\bar{e}}}_{\ge T, \bar{e}, \varepsilon}(\vec\bx; S_{\bar{e}}; P, Q ; \sigma_\boat, \sigma_\anchor).$$ If $A$ denotes the number of acute corners of $B'$ at the points $e^{\infty, 1}$, then the difference between the number of acute corners of $B$ and $B'$ satisfies $$\ac(B') - \ac(B)  \ge A - 2(|P| - |P'|),$$
with equality holding only if every component of $S_\anchor$ has at most one east puncture.
\end{lem}
\begin{proof}The acute corners of $B$ and $B'$ consist of two types: \emph{east acute} corners and \emph{main acute} corners. Here, a main acute corner is an acute corner corresponding with a non-east puncture of $S$ (resp. $S'$), and east acute corners correspond with intersections of arcs in $\bgamma^1$ with $\partial \Sigma$. The main corners of $B'$ split into two sets: those $A$ at $e^{\infty,1}$ and all others. As $B'$ approximates $B$, it follows that the number of main acute corners of $B'$ not at $e^{\infty,1}$ is precisely the number of main acute corners of $B$. As such, $\ac(B') - \ac(B) = A + A_E' - A_E$, where $A'_E$ and $A_E$ are the number of easterly acute corners of $B'$ and $B$ respectively.

The quantities $A_E$ and $A'_E$ may be read from the multiplicity of $B$ and $B'$ in a neighbourhood of each $e^{\infty}_j$. Indeed, each puncture in $E_j(S)$ (resp. $S'$) contributes two acute corners (at the start and end of the corresponding Reeb chord) unless there is some other puncture labelled with the subsequent chord. If we choose a small neighbourhood $N_j$ of east infinity and four points $p_1, p_2, p_3$ in the corresponding regions of $N_j - \bgamma^1_j$, this translates as
\begin{align*}A_E &= n_{p_1}(B) + |n_{p_1}(B) - n_{p_2}(B)| + |n_{p_2}(B) - n_{p_3}(B)| + n_{p_3}(B)\\
A'_E &=  n_{p_1}(B') + |n_{p_1}(B') - n_{p_2}(B')| + |n_{p_2}(B') - n_{p_3}(B')| + n_{p_3}(B').
\end{align*}
We extend this to $B_\anchor$, setting
$$A_{\anchor,E} :=  n_{p_1}(B_\anchor) + |n_{p_1}(B_\anchor) - n_{p_2}(B_\anchor)| + |n_{p_2}(B_\anchor) - n_{p_3}(B_\anchor)| + n_{p_3}(B_\anchor).$$
This final quantity is less than or equal to $2\cdot |E_j(S_\anchor)|$, with equality holding only if $|E_j(S_\anchor)| = 1$. By the triangle inequality, $A'_E + A_{\anchor, E} \le A_E$ and hence $A_E - A'_E \le A_{\anchor,E} \le 2\cdot |E_j(S_\anchor)|$, with equality holding only if $|E_j(S_\anchor)|=1$.

Summing over all $j$, it follows that
$$\ac(B') - \ac(B) = A + A'_E - A_E \ge A - 2(|P| - |P'|)$$
as required: with equality only if $|E_j(S_\anchor)| = 1$ for each $j$.

\end{proof}

\begin{prop}\label{simpleanchorssimpler}Suppose that $u_n \in  {\pmb{\cM}}^B_{\ge T}(\bx; S; P, Q; \sigma_\boat, \vec\sigma_\anchor)$ is a sequence of splayed holomorphic polygons, converging to some anchored holomorphic curve  $$(u', u_\anchor) \in \cM^{B'}(\bx, \btheta^{1,2}, \dots, \btheta^{m-1,0}, \by; S'; P', Q', \sigma_\boat', \sigma_\anchor'; \bJ) \times \cM^{B_\anchor}_\anchor(\btheta^{1,2}; S_\anchor; P_\anchor, \sigma).$$
Then $S_\anchor$ a disjoint union of topological discs, where $|E_j(S_\anchor)| \le 1$ for every $j$, and either $A = 1$ and $Q'$ consists of the punctures of $S'$ induced by $Q$, or $A = 0$ and $Q'$ consists of the punctures of $S'$ induced by $Q$, together with a single puncture labelled by $v^{1,2}$.
\end{prop}
\begin{proof}Let $g' = g+h-1$, $e(B_\anchor) := e(B) - e(B')$, $\col_\anchor:= \col(\sigma_\boat) - \col(\sigma_\boat')$ and $\#_\anchor := |Q| - |Q'|$ as in the proof of Proposition \ref{anchorsaresimple}, and observe that the quantities $\col_\anchor$ and $\#_\anchor$ are negative. Note also that $\chi(S_\anchor) \le g'$, with equality if and only if $S_\anchor$ is a disjoint union of topological disks.

The argument in Proposition \ref{anchorsaresimple} tells us that
$\ind(B,S,P, Q, \sigma_\boat) - \ind'$ is equal to 
\begin{equation}\label{equationindexanchor}\frac{3-m_\anchor}{2}g' - \chi(S_\anchor) + 2e(B_\anchor) + |P_\anchor| - \col_\anchor - \#_\anchor.\end{equation}

As $u_\anchor$ is an anchor, it follows that $\ind' \ge 2-(m+1) = 1-m$. By assumption, $\ind(B,S,P, Q, \sigma_\boat) = 2-m$ and so the formula (\ref{equationindexanchor}) is less than or equal to $1$.

We examine the quantity $2(e(B) - e(B')) = 2e(B_\anchor)$. It is clear that $\chi(B) = \chi(B')$, and thus $2e(B_\anchor) = \frac{1}{2}\left(\ob(B) - \ob(B') -\ac(B) + \ac(B')\right)$. If we let $A$ denote the number of acute corners at the vertices $e^{\infty,1}_1, \dots, e^{\infty,1}_h$, then as no corners at east infinity are obtuse we can see that $\ob(B) - \ob(B') = (1-g + (A - h)) = -g' + A$. By the previous lemma, we know that $\ac(B') - \ac(B) \ge A - 2(|P| - |P'|)$ and as such $2(e(B) - e(B')) \ge -g' + A - |P_\anchor|$. Hence
$$1 \ge \frac{3-m_\anchor}{2}g' - \chi(S_\anchor) + 2e(B_\anchor) + |P_\anchor| - \col_\anchor - \#_\anchor \ge 0$$
and at most one one of the terms $A, g' - \chi(S_\anchor), \col_\anchor$ or $\#_\anchor$ is positive (and equal to one).

Assume first that $\chi(S_\anchor) = g'$, so that $S_\anchor$ is a disjoint union of topological disks. There are a series of cases for this:
\begin{itemize}\item If $A = 1$, then $\col_\anchor = \#_\anchor = 0$ (so $|Q| = |Q'|$) and each component of $S_\anchor$ has at most one east puncture.
\item If $A = 0$, then either:\begin{itemize}\item $
\#_\anchor = 1$, so that $|Q| = |Q'| - 1$, $\col_\anchor = 0$, and each component of $S_\anchor$ has at most one east puncture; or
\item $\col_\anchor = 1$,  $\#_\anchor = 0$ and each component of $S_\anchor$ has at most one east puncture: this does not occur as gluing $(u', u_\anchor)$ gives a nearby curve in some ${\pmb{\cM}}^B(\vec\bx; S; P, Q; \sigma'_\boat, \vec\sigma_\anchor; \bJ^t)$, where $\col(\sigma'_\boat) > 0$ --- such moduli spaces have negative expected dimension.
\item Or, finally, $A_E - A_{E'} = 2(|P| - |P|) -1$: this does not occur as the left hand side of this equation is even.
\end{itemize}
\end{itemize}
If $\chi(S_\anchor) = g'-1$, then one component of $S_\anchor$ is a topological annulus without an `$\infty$' puncture on one boundary component. Any nearby curve $u$ in ${\pmb{\cM}}^B(\bx; S; P, Q; \sigma'_\boat, \vec\sigma_\anchor; \bJ^t)$ has a boundary component with no non-east puncture upon it, so that $\pi_D \circ u$ is constant on this component or violates the maximum modulus principle --- hence this case does not occur (compare \cite[Proposition 5.43]{LOT}). \end{proof}

Proposition \ref{anchorsaresimple} tells us that if $(u',u_\anchor) \in \partial \overline{{\pmb{\cM}}}^B_{\ge T}(\bx; S; \arrowrho_{\boat}, \arrowrho_\anchor; \vec\sigma_\boat, \vec\sigma_\anchor, \bJ^t)$, each component of the source $S_\anchor$ has at most one puncture which is mapped to $e^\infty_j$ for each $j$. By definition, each of these punctures is labelled by some element of the torus algebra $\rho_j$, and we shall write $\cM^{B_\anchor}(\btheta^{1,2}; S_\anchor; \arrowrho, \sigma)$ to refer to the moduli space of curves in $\cM^{B_\anchor}_\anchor(\btheta^{1,2}; S_\anchor); P_\anchor, \sigma)$  where $P_\anchor$ is the discrete partition and $\arrowrho = \{\rho_j : \rho_j \in E_j(S_\anchor)\}$.

In a similar manner to Section \ref{easterlyrestrictions}, we can define
$$\overline{\Gamma} = \{j \in 1, \dots, h: E_j(S_\anchor) \ne \emptyset\}$$
and show that the moduli space $\cM^{B_\anchor}(\btheta^{1,2}; S_\anchor; \arrowrho, \sigma)$ splits as a product
$$\cM^{B_\anchor}(\btheta^{1,2}; S_\anchor; \arrowrho, \sigma) = \cM_\anchor \times \cM_{\anchor,1} \times \cdots \times \cM_{\anchor, |\overline{\Gamma}|}.$$
Here $\cM_\anchor$ is a moduli space of one-gons in a closed Heegaard multi-diagram $\cD'_\anchor$, and each $\cM_{\anchor, j}$ is a moduli space of one-gons defined using a bordered multi-diagram $\cD_{\anchor, j} = (\Sigma_{\anchor, j}; \gamma^1_j; b_j)$, where $\Sigma_{\anchor, j}$ is a torus.

\begin{prop}\label{anchorsawayfromedge}There is a unique homology class $B_{\anchor}'$ which carries a holomorphic representative in the moduli space $\cM_\anchor$. Moreover, $B_\anchor' \in \pi_2(\btheta^{1,2}_{\text{res},+})$, where $\btheta^{1,2}_{\text{res},+}$ is the restriction of the distinguished generator $\btheta^{1,2}_+$ to the diagram $\cD'_\anchor$. As such, the moduli space $\cM^{B_\anchor}(\btheta^{1,2}; S_\anchor; \arrowrho, \sigma)$ is non-empty only if $B_\anchor = B_\anchor' + \sum_{j \in \overline{\Gamma}} B_{\anchor, j}$, where each $B_{\anchor,j}$ is supported within the handle $\handle_j$, and $\btheta^{1,2} = \btheta^{1,2}_+$, in which case it has the same number of elements as $\sum_{j \in \overline{\Gamma}}\cM_{\anchor, j}$.
\end{prop}
\begin{proof}This is considered in \cite[Lemma 11.8]{Lipshitz:cylindrical}, though as stated there the proof is incorrect. See \cite{Lipshitz:errata2} for the correction to this.
\end{proof}

As such, we need only to understand the moduli spaces $\cM_{\anchor, j}$.

\begin{prop}\label{howmanyanchors}Let $S_{\anchor,j}$ be an anchor source with underlying surface a disc, and with unique east puncture, $q$. For the evaluation map $\ev_q: \cM^{B_{\anchor, j}}(\theta^{1,2}_j; S_{\anchor, j}; \rho_j) \rightarrow \RR$, the moduli space $\cM_{\anchor,j}(0) = \ev^{-1}(0)$  contains a unique element for every $\rho_j$.
\end{prop}
\begin{proof}This follows by a model computation for each choice of $\rho_j$, and Proposition \ref{gluinganchored}. For instance, consider the bordered Heegaard diagram and Heegaard multi-diagrams as shown in Figure \ref{fig:modelcase}.
\begin{figure}[h]
\centering{
\makebox[\textwidth]{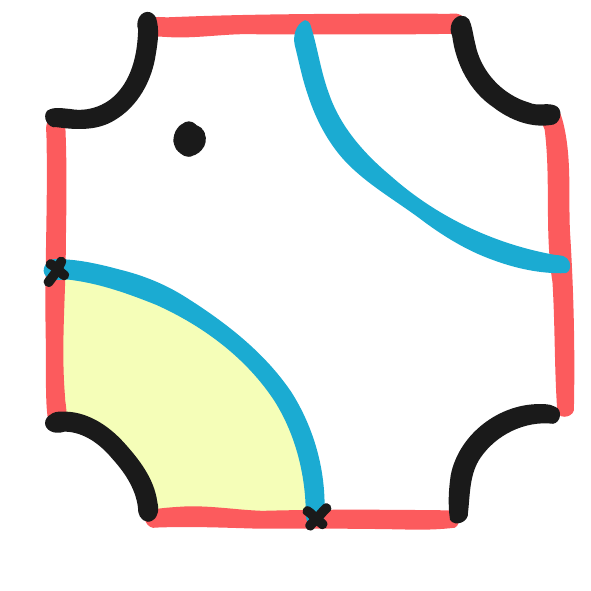\quad 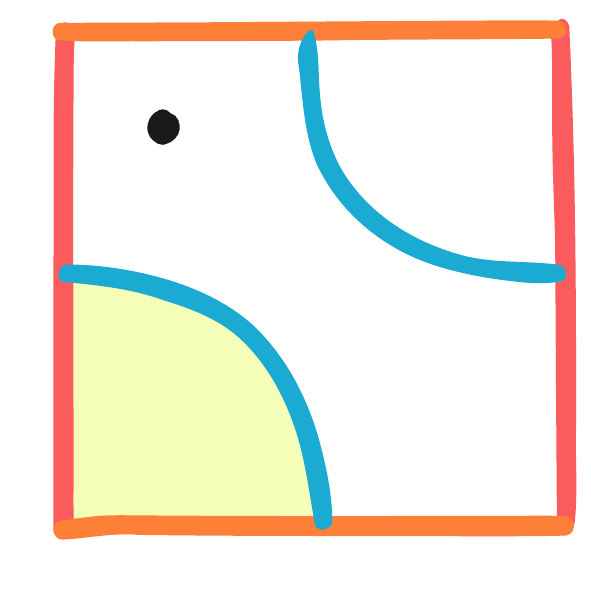} \caption{A model computation for Proposition \ref{howmanyanchors}.}
\label{fig:modelcase}
}
\end{figure}

The moduli space $$\cM^B(x, y; S;\rho_3, \emptyset; \{\rho_3\}, \emptyset)$$ clearly contains a unique element, as does the moduli space $$\cM^{B'}(x, \theta^{1,2}_+, y; S'; \emptyset, \rho_3; \emptyset, \{\rho_{3}\}).$$ Therefore, by Proposition \ref{gluinganchored},
$$|\cM^B(x, y; S;\rho_3, \emptyset; \{\rho_3\}, \emptyset)| = |\cM^{B'}(x, \theta^{1,2}_+, y; S'; \emptyset, \rho_3; \emptyset, \{\rho_{3}\})| \cdot |\cM_{\anchor, j}(0)|,$$ 
hence $|\cM_{\anchor, j}(0)| = 1$, as required.
\end{proof}

This lets us rule out anchored curves with holomorphic anchors that have more than one easterly puncture:
\begin{prop}Suppose that  $$(u', u_\anchor) \in \cM^{B'}(\bx, \btheta^{1,2}, \dots, \btheta^{m-1,0}, \by; S'; P', Q', \sigma_\boat', \vec\sigma_\anchor'; \bJ) \times \cM^{B_\anchor}_\anchor(\btheta^{1,2}; S_\anchor; \arrowrho, \vec\sigma)$$ is an anchored holomorphic curve in the boundary of the moduli space $${\pmb{\cM}}^B_{\ge T}(\bx; S; P, Q; \sigma_\boat, \vec\sigma_\anchor).$$
Then $\vec\sigma_\anchor'$ is interleaved, and hence $S_\anchor$ consists of a disjoint union of topological disks, only one of which has an east puncture.
\end{prop}
\begin{proof}We already know that each component of $S_\anchor$ is a disk, with at most one east puncture. Suppose that $\vec\sigma_\anchor'$ were not interleaved: by Proposition \ref{howmanyanchors}, there is a (unique) element $v_\anchor$ of  $\cM^{B_\anchor}_\anchor(\btheta^{1,2}; S_\anchor; \arrowrho, \vec\sigma)$ where the height of each east puncture is $0$. As such, by \ref{gluinganchored}, we could glue $u'$ to $v_\anchor$ give an element of a moduli space ${\cM}^B_{\ge T}(\bx; S; \arrowrho_{\boat}, \arrowrho_\anchor; \sigma_\boat; \bJ^t)$ with $\col(\sigma_\boat) > 0$: but this moduli space has negative expected dimension.
\end{proof}

The final result we need is to understand the implied asymptotics of curves in the boundary:
\begin{prop}\label{correctdataboundary}Suppose that  $$(u', u_\anchor) \in \cM^{B'}(\bx, \btheta^{1,2}, \dots, \btheta^{m-1,0}, \by; S'; P', Q', \sigma_\boat', \vec\sigma_\anchor'; \bJ) \times \cM^{B_\anchor}_\anchor(\btheta^{1,2}; S_\anchor; P_\anchor, \sigma)$$ is an anchored holomorphic curve in the boundary of the moduli space $${\pmb{\cM}}^B_{\ge T}(\vec\bx; S; \arrowrho_\boat, \arrowrho_\anchor; \vec\sigma_\boat, \vec\sigma_\anchor).$$

 Then $u' \in \cM^{B'}(\bx, \btheta^{1,2}, \dots, \btheta^{m-1,0}, \by; S'; \arrowrho_\boat', \arrowrho_\anchor', \vec\sigma_\boat', \vec\sigma_\anchor'; \bJ)$, where the pair $(\arrowrho'_\boat, \sigma'_\boat)$ and $(\arrowrho'_\anchor, \sigma'_\anchor)$ splay $(\arrowrho_\boat, \sigma_\boat)$ and $(\arrowrho_\anchor, \sigma_\anchor)$.
\end{prop}
\begin{proof}
Consider a small disc neighbourhood $M^1_j$ of $e^{\infty,1}_j$ which satisfies that each component of $\pi_\Sigma \circ u^{-1}(M^1_j)$ contains a single point of the set $\Xi : = \pi_\Sigma \circ u^{-1}(e^{\infty,1}_j)$ (we include the puncture $q$ which is mapped to $e^{\infty, 1}_j$ in this set) --- label the components by $S(\xi)$ so that $\xi \in \Xi$ is contained in $S(\xi)$. The set $\Xi$ splits into three sets: $q$, the set $\Xi_L$ of points mapped to $e^2$ by $\pi_D \circ u$, and the set $\Xi_R$ of points mapped to $e^1$ by $\pi_D \circ u$. There are corresponding restricted sources $S_L := \cup_{\xi \in \Xi_L}S(\xi)$ and $S_R : = \cup_{\xi \in \Xi_R}S(\xi)$.

The homology class of $\pi_\Sigma \circ u|_{S_L}$ is modelled upon $n \cdot B(\rho)$ for some non-jumping $\rho$ (which is equal to $\rho_{12}$ if $\gamma^2_j$ is meridional, and $\rho_{23}$ if $\gamma^2_j$ is longitudinal), and the homology class of $\pi_\Sigma \circ u|_{S(q)}$ is modelled upon $B(\rho')$ for some jumping $\rho'$. It follows (by considering multiplicities away from the approximation region between $\bgamma^1$ and $\bgamma^2$) that if
$$(D_j^1,B_j^1,D_j^2,B_j^2, \dots,D_j^k, B_j^k , D_j^{k+1})$$
were a sequence of Reeb domains for $B_j$, the sequence obtained by replacing each $D_j^i = n_i B^i(\rho_i)$ with $n_i B^{i+1}(\rho_i)$ and each $B_j^i = B^i(\rho'_i)$ with $B^{i+1}(\rho'_i)$, 
 adjoining $B^0_j = B^1_j(\rho')$ and replacing $D_j^1 = B^i(\rho_1)$ by $n \cdot B^2(\rho) + n_1 B^2(\rho_1)$ is a sequence of Reeb domains for $B_j' - B(\arrowrho_{\boat,j}')$.
 
It remains to show that the sequence $\arrowrho = (\rho', \underbrace{\rho, \dots, \rho}_{n \text{\ times}})$ satisfies that $\arrowrho_{\boat,j}' \star \arrowrho = \arrowrho_{\boat, j}$. To see this, we note that the domain $B(\arrowrho_\anchor)$ must agree with $B^1(\arrowrho_\anchor') + B^1(\arrowrho)$ away from the approximation region. As such, $B(\arrowrho) = B(\arrowrho_\anchor) - B(\arrowrho_\anchor')$. This latter quantity, however, may be written in a more useful way. 

Let $\mathcal{Q} = \pi_\Sigma \circ u^{-1}(e^\infty_j)$, and let $\mathcal{Q}'$ denote the subset of $\mathcal{Q}$ given by $\{q \in \mathcal{Q}: \ev_j(q) < 0\}$. If we fix a small neighbourhood $M^\infty_j$ of $e^\infty_j$ such that each component of $\pi_\Sigma \circ u^{-1}(M^\infty_j)$ contains a unique point in $\mathcal{Q}$, the domains $B(\arrowrho_\anchor)$ and $B(\arrowrho_\anchor')$ can be written as a sum over points in $\mathcal{Q}$ or $\mathcal{Q}'$ respectively of the quantities $B(\rho(q))$, where $\rho(q)$ is such that $(\pi_\Sigma \circ u)(S(q))$ is modelled upon $B(\rho(q)) \subset M_j^\infty$ --- and $S(q)$ denotes the component of $ \pi_\Sigma \circ u^{-1}(M^\infty_j)$ containing $q$. It follows by considering the multiplicity of $u$ at base-points in each of the four quadrants of $M_j^\infty$ that $B(\arrowrho_\anchor) - B(\arrowrho_\anchor')$ coincides with the sum over points $q$ in $\mathcal{Q} - \mathcal{Q}'$ of $B(q)$, and thus is equal to $B(\underline{\arrowrho})$, where $\underline{\arrowrho}$ satisfies $\arrowrho_{\boat,j}' \star \underline{\arrowrho} = \arrowrho_{\boat, j}$. As such, $B(\arrowrho) = B(\underline{\arrowrho})$ and so $\arrowrho$ must lie in the same $123$-equivalence class as $\underline{\arrowrho}$. The cardinality of the sets $\mathcal{Q}$ and $\mathcal{Q}' \cup \Xi_L \cup \{q\}$ agree, so in fact $\arrowrho = \underline{\arrowrho}$.
\end{proof}

We finish this section by utilising the results above to understand properly the behaviour of moduli spaces under the neck-stretching operation.

\begin{prop}\label{neckstretch}Let $B, S, \arrowrho_\boat, \arrowrho_\anchor$ and $\vec\sigma_\boat$ be such that
$${{\pmb{\cM}}}_{\ge T} := {\pmb{\cM}}^{B_{\bar{e}}}_{\ge T, \bar{e}, \varepsilon}(\vec\bx; S_{\bar{e}}; \arrowrho_\boat, \arrowrho_\anchor ; \sigma_\boat, \sigma_\anchor)$$
is one-dimensional, and $\vec\sigma_\boat$ is interleaved.

Say, as shorthand, that $(B'; \arrowrho'_\boat, \sigma'_\boat; \arrowrho'_\anchor, \sigma'_\anchor)$ splays $(B; \arrowrho_\boat, \sigma_\boat; \arrowrho_\anchor, \sigma_\anchor)$ if $B'$ splays $B$, and $(\arrowrho'_\boat, \sigma'_\boat)$ and $(\arrowrho'_\anchor, \sigma'_\anchor)$ splay $(\arrowrho_\boat, \sigma_\boat)$ and $(\arrowrho_\anchor, \sigma_\anchor)$.

Denote by
$$\cM_{\text{splayed}}: = \bigcup_{\substack{(B'; \arrowrho'_\boat, \sigma'_\boat; \arrowrho'_\anchor, \sigma'_\anchor)\\ \text{ splays }\\ (B; \arrowrho_\boat, \sigma_\boat; \arrowrho_\anchor, \sigma_\anchor)}} \cM^{B'}(\bx, \btheta_+(k-1), \by; S'; \arrowrho'_\boat, \arrowrho'_\anchor; \sigma'_\anchor, \sigma'_\boat; \bJ).
$$

Then $\overline{{\pmb{\cM}}}_{\ge T}$  is a compact one-manifold with boundary given by the disjoint union of $\cM^{B_{\bar{e}}}_{\bar{e},\varepsilon}(\bx, \btheta_+(k-2), \by; S; \arrowrho_\boat, \arrowrho_\anchor; \vec\sigma_\boat, \vec\sigma_\anchor)$ and $$\cM_{\text{splayed}} \times \cM^{B_\anchor}(\btheta^{1,2}, S_\anchor, \rho),$$
where $\rho$ is the unique jumping chord in the final part of $\vec\sigma_\boat$.

Therefore, the moduli spaces $$\cM^B(\bx, \btheta_+(k-2), \by; S; \arrowrho_\boat, \arrowrho_\anchor; \vec\sigma_\boat, \vec\sigma_\anchor; \bJ^T)$$ and $\cM_{\text{splayed}}$  contain the same number of elements modulo two.
\end{prop}
\begin{proof}By Proposition \ref{correctdataboundary}, we know that any curve in $\partial\overline{{\pmb{\cM}}}_{\ge T}$ lives in the moduli space $$\cM_{\text{splayed}} \times  \cM^{B_\anchor}(\btheta^{1,2}, S_\anchor, \rho).$$
By Proposition \ref{anchorsaresimple}, we know that $S_\anchor$ is a disjoint union of topological disks, and thus by Proposition \ref{gluinganchored},  $\partial\overline{{\pmb{\cM}}}_{\ge T}$ is a compact one-manifold with boundary given.

The second part follows from the identification of $$\cM^{B_{\bar{e}}}_{\bar{e},\varepsilon}(\bx, \btheta_+(k-2), \by; S; \arrowrho_\boat, \arrowrho_\anchor; \vec\sigma_\boat, \vec\sigma_\anchor)$$ with $$\cM^B(\bx, \btheta_+(k-2), \by; S; \arrowrho_\boat, \arrowrho_\anchor; \vec\sigma_\boat, \vec\sigma_\anchor)$$ given in Propositions \ref{eastcompactifying} and \ref{perturbedcylinders}, and with the fact that  $\cM^{B_\anchor}(\btheta^{1,2}, S_\anchor; \rho)$ has one element modulo two by Proposition \ref{howmanyanchors}.
\end{proof}

\section{Statement and proof of theorems}
After all this formality, we are at long last ready to state some theorems.

Let $\border\cD$ be a bordered diagram with $h$ boundary components. We shall define a right $\cA_\infty$ multi-module over the torus algebra, which we call $\poly(\border\cD)$. We first define a series of auxiliary modules $\poly_k(\border\cD)$.

Consider the module generated over $\FF_2$ by the set of generators $\cG(\border\cD)$. This has an obvious right action by the ring of idempotents of $\cT^h$, defined by setting
 $$\bx \cdot \iota = \left\{\begin{array}{c c} \bx & \mbox{ if } \iota(\bx) = \iota\\
0 & \mbox{ otherwise}.\end{array}\right.$$

We define an $\cA_\infty$ module structure upon $\poly_k(\border\cD)$ by first defining maps
$$n^k_{i_1, \dots, i_h}: \poly_k(\border\cD) \otimes_{\cI} \cT_1^{\otimes i_1} \otimes \cdots \otimes  \cT_h^{\otimes i_h} \times \sigma(i_1, \dots, i_h) \rightarrow \poly_k(\border\cD),$$
where $\sigma(i_1, \dots, i_h)$ denotes the set of all splicings of all sets of sequences of Reeb chords in  $\cT_1^{\otimes i_1} \otimes \cdots \otimes  \cT_h^{\otimes i_h}$.

Suppose that $\arrowrho = \arrowrho_1, \dots, \arrowrho_h$ is a set of sequences of Reeb multi-chords in $\cT^h$, where $|\arrowrho_j| = i_j$. For every interleaved splicing $\sigma$ of $\arrowrho$, we define the coefficient 
$$a^k(\bx, \by, \arrowrho, \sigma) : =\sum_{\substack{(\arrowrho_\boat, \sigma_\boat), (\arrowrho_\anchor, \sigma_\anchor)\\ \text{is a\ }k-\text{shipping}\\\text{ of\ }(\arrowrho, \sigma)}} \sum_{\substack{B\in \pi_2^{\btheta_+}(\bx, \by)\\ \chi(S) = \chi_\emb(S)\\  \ind(B, S \arrowrho_\boat, \arrowrho_\anchor, \sigma_\boat) = 2 - k}} \# \cM^B(\bx, \btheta_+(k-2), \by'; S; \arrowrho_\boat, \arrowrho_\anchor; \sigma_{\boat}, \sigma_{\anchor})$$
where the nearest-point map takes $\by'$ to $\by$, and hence put

$$n^k_{i_1, \dots, i_h}(\bx, \arrowrho, \sigma) := \sum_{\by \in \cG(\border\cD)} a^k(\bx, \by, \arrowrho, \sigma) \cdot \by.$$

We define:
$$m^k_{i_1, \dots, i_h}: \poly_k(\border\cD) \otimes_{\cI} \cA_1^{\otimes i_1} \otimes \cdots \otimes  \cA_h^{\otimes i_h} \rightarrow \poly_k(\border\cD)$$
by
 $$m^k_{i_1, \dots, i_h}(x, \arrowrho) : = \sum_{\sigma \text{\ an interleaving for\ }\arrowrho} n^k_{i_1, \dots, i_h}(x, \arrowrho, \sigma).$$

\begin{prop}\label{kiszeroisbsa}For sufficiently generic choice of $\bJ$, the definition of $\poly_0(\border \cD)$ above agrees with the definition of the multi-module $BSA(\border\cD)$ described in \cite[Section 8.3]{Zarev}.
\end{prop}
\begin{proof}In the case where $k = 0$, the above moduli spaces are formed of bi-gons in the bordered diagram $\border\cD$, with constraints upon evaluation maps given by the interleaving $\sigma$. To see that these are precisely the bi-gons counted in the definition of $BSA(\border\cD)$, we note that for any $(B, S, \arrowrho_\anchor)$ such that the moduli space $\cM(\bx, \by; S; \arrowrho_\anchor)$ is zero-dimensional (as counted in the definition of the $BSA$ multiplication maps), the subspace of $\cM(\bx, \by; S; \arrowrho_\anchor)$ where the heights of any punctures $q, q'$ labelled by a jumping chord coincide is given by $\cM(\bx, \by; S; \arrowrho_\anchor; \sigma)$ for $\sigma$ some non-interleaved splicing of $\arrowrho_\anchor$. By the formula at the end of Section \ref{splicingbordered} (or Proposition \ref{bordedpolyindex}), $\ind(B,S, \arrowrho_\anchor, \sigma)< 1$ and so the corresponding moduli space is empty: i.e. the multiplication maps in $BSA(\border\cD)$ do not count curves where heights coincide. (This should be compared with \cite[Section 2.4.3]{Lipshitz:bimodules}.)

As such, the moduli space $\cM(\bx, \by; S; \arrowrho_\anchor)$  splits as a union over all interleaved splicings of moduli spaces $\cM(\bx, \by; S; \arrowrho_\anchor; \sigma)$, which is precisely what the maps $m^0_{i_1,\dots,i_h}$ count.
\end{proof}

\begin{prop}\label{kappabound} (Compare \cite[Lemma 7.7]{LOT}, \cite[Theorem 7.8.]{Zarev})Suppose that the diagram $\border \cD$ is provincially admissible. For any fixed generators $\bx$ and $\by$, and a set of non-empty sets of Reeb chords, there are at most finitely many $B \in \pi_2^{\btheta_+}(\bx, \by)$ such that $\cM^B(\bx, \btheta_+(k-2), \by'; S; \arrowrho_\boat, \arrowrho_\anchor)$ are non-empty --- i.e. the sums in the definition of the maps $m_n$ are finite.

Moreover, if $\border \cD$ is admissible, then there is a constant $\kappa$ depending only upon $\border \cD$ for which every map $m^k_i$ with $i > \kappa$ is zero, independent of $k$.
\end{prop}
\begin{proof}We proceed by induction. In the case $k = 0$, by Proposition \ref{kiszeroisbsa} the proposition coincides with \cite[Theorem 7.8.]{Zarev}. 

For higher $k$, we prove the first half of the proposition by noting that by Proposition \ref{gluinganchored}, if there were an infinite number of $B \in \pi_2(\bx, \btheta_+(k), \by)$ with an element $u_B \in \cM^B(\bx, \btheta_+(k-2), \by'; S; \arrowrho_\boat, \arrowrho_\anchor; \sigma_\boat, \sigma_\anchor)$, we can glue $u_B$ to a holomorphic anchor to give a curve $u_{B'} \in \cM^{B'}(\bx, \btheta_+(k-3), \by'; S'; \arrowrho_\boat', \arrowrho_\anchor'; \sigma_\boat', \sigma_\anchor')$, where $(\arrowrho'_\boat, \sigma'_\boat), (\arrowrho_\anchor', \sigma_\anchor')$ is a $k-1$ shipping of $\arrowrho$. By induction hypothesis, this is impossible.

For the second half, we let $\kappa$ be the corresponding constant for $m^0_i$. Again, if some $i > \kappa$ does not satisfy that $m_i^k = 0$, we may glue the corresponding polygon to a holomorphic anchor, implying that $m_i^{k-1}$ is also nonzero. By induction hypothesis, this is, again, impossible.
\end{proof}

We now turn to the main definition and theorem of the paper. Consider the $2^h$ elementary splayings $\cD(\iota(\delta))$, where $\delta \in \{m,l\}^h$. The module $\poly(\cD)$ is generated over $\FF_2$ by the sum $\cG := \bigoplus_{\delta \in \{m,l\}^h}\cG(\cD(\iota(\delta)))$. We endow this module with the obvious right action of the ring of idempotents of $\cT^h$ by setting $$\bx \cdot \iota(\delta) = \left\{\begin{array}{c c} \bx & \mbox{ if } \bx \in \cG(\cD(\iota(\delta)))\\
0 & \mbox{ otherwise}.\end{array}\right.$$

We define an $\cA_\infty$ module structure upon $\poly(\border\cD)$ by first defining maps
$$n_{i_1, \dots, i_h}: \poly(\border\cD) \otimes_{\cI} \cT_1^{\otimes i_1} \otimes \cdots \otimes  \cT_h^{\otimes i_h} \times \sigma(i_1, \dots, i_h) \rightarrow \poly(\border\cD),$$
where $\sigma(i_1, \dots, i_h)$ denotes the set of all splicings of all sets of sequences of Reeb chords in  $\cT_1^{\otimes i_1} \otimes \cdots \otimes  \cT_h^{\otimes i_h}$.

Suppose that $\arrowrho = \arrowrho_1, \dots, \arrowrho_h$ is a set of sequences of Reeb multi-chords in $\cT^h$, where $|\arrowrho_j| = i_j$. For every interleaved splicing $\sigma$ of $\arrowrho$, we define
$$n^k_{i_1, \dots, i_h}(\bx, \arrowrho, \sigma) := \sum_{\by \in \cG(\border\cD)}  \sum_{\substack{B\in \pi_2^{\btheta_+}(\bx, \by)\\ \chi(S) = \chi_\emb(S)\\  \ind(B, S \arrowrho, \sigma) = 2 - k}} \# \cM^B(\bx, \btheta_+(k-2), \by'; S; \arrowrho; \sigma) \cdot \by,$$
where the nearest-point map takes $\by'$ to $\by$.

We define:
$$m_{i_1, \dots, i_h}: \poly(\border\cD) \otimes_{\cI} \cA_1^{\otimes i_1} \otimes \cdots \otimes  \cA_h^{\otimes i_h} \rightarrow \poly(\border\cD)$$
by
 $$m_{i_1, \dots, i_h}(x, \arrowrho) : = \sum_{\sigma \text{\ an interleaving for\ }\arrowrho} n_{i_1, \dots, i_h}(x, \arrowrho, \sigma).$$

We first verify that the modules $\poly_k(\border\cD)$ are $\cA_{\infty}$ modules.
\begin{thm}\label{ainfinity}The maps $\{n_{i_1, \dots, i_h}\}_{i_1, \dots, i_h}$ satisfy the partial $\cA_\infty$ relation. As such, the data $(\poly_k(\border\cD), \{m^k_{i_1, \dots, i_h}\}_{i_1, \dots, i_h}^\infty)$ forms an $\cA_\infty$ module over $\cT^h$.
\end{thm}
\begin{proof}Recall the relation which we need to verify is
\begin{align*}0 & = \sum_{\substack{\arrowlambda \star \arrowdelta = \arrowrho \\ \sigma_1 \star \sigma_2 = \sigma}} n(n(x, \arrowlambda, \sigma_1), \arrowdelta, \sigma_2 )\\
&+ \sum_{i, j \text{\ compatible}} n(x,\bar{\mu}^i_j(\arrowrho), \sigma^i_j) \\
&+ \sum_{k' \text{\ collidable}}^{|\sigma|-1} n(x, \arrowrho, \sigma(k')).\end{align*}

Consider a one-dimensional moduli space $\cM = \cM^B(\bx, \btheta_+(k), \by; S; \arrowrho_\boat, \arrowrho_\anchor; \sigma_\boat, \sigma_\anchor)$, where $\sigma_\boat \star \sigma_\anchor = \sigma$, and $\sigma_\anchor$ is a $k$-splicing. By Proposition \ref{evenends}, the ends of this moduli space are formed of the following.
\begin{itemize}
\item Spinal two-story ends, corresponding with terms in the first sum in the statement of the theorem.
\item Mixed two-story ends, which, by Proposition \ref{easterlyoptions} are either composition ends or collision ends (see Definition \ref{mixednames}). Composition ends correspond (by Proposition \ref{compcollidentifications}) with terms in the second sum in the statement where $i,j$ satisfy that $\rho^i_j \in P^{i'}_j \in \sigma$ for $i'\ge k$, and $\rho^i_j, \rho^{i+1}_j$ are equal to $\rho_1, \rho_2$ or $\rho_2, \rho_3$; collision ends correspond with terms in the third sum in the statement where $k' \ge k$.
\item $(i,j)$-composition ends which correspond with terms in the second sum where $i, j$ satisfy that $\rho^i_j \in P^{i'}_j$ for $i' < k$.
\item  $k'$-collision ends, which correspond with terms in the third sum where $k' < k$.
\item Cut-vanishing ends, corresponding with terms in the third sum in the statement where $i,j$ satisfy that $\rho^i_j \in P^{i'}_j \in \Sigma$ for $i'\ge k$, and $\rho^i_j\cdot \rho^{i+1}_j$ is equal to $\rho_{123}$.
\end{itemize}
By Theorem \ref{evenends}, the number of these is zero modulo two; it corresponds that the sum is zero, as required.
\end{proof}
We will not show directly that $\poly_k(\border\cD)$ is an invariant of $\border\cD$ up to equivalence --- indeed, we have not discussed any notion of equivalence for diagrams. Instead, we just need to understand how $\poly_k(\border\cD)$ depends upon the choice of generic admissible compatible almost-complex structures $\bJ$.

\begin{prop}\label{changeofcomplexstructure}Let $\bJ_0$ and $\bJ_1$ be generic, admissible, compatible families of almost-complex structures. Then the $\cA_\infty$ modules $\poly_k(\border\cD, \bJ_0)$ and $\poly_k(\border\cD, \bJ_1)$ are $\cA_\infty$ homotopy equivalent.
\end{prop}
\begin{proof}This follows by modifying \cite[Section 7.3.1]{LOT} to the case of splayed polygons --- combining it with, say, \cite[Proposition 3.30]{LOT:SSII}, and the proof of Theorem \ref{ainfinity} above.
\end{proof}

\begin{thm}\label{polyequivalence}Let $k \ge 0$ be an integer. Then the $\cA_\infty$ modules $\poly_k(\border\cD, \bJ)$ and $\poly_{k+1}(\border\cD, \bJ)$ are $\cA_\infty$ homotopy equivalent.
\end{thm}
\begin{proof}
By Proposition \ref{neckstretch}, and the finiteness of the zero-dimensional moduli spaces involved in the definition of the $m^k_i$ and $m^{k+1}_i$ there is some $T >0$ such that, for every $\bx, \by, B, S, \arrowrho, \vec\sigma$,the moduli spaces $\cM^B(\bx, \btheta_+(k), \by; S; \arrowrho_\boat, \arrowrho_\anchor; \sigma_\boat, \sigma_\anchor, \bJ^T)$ and
$$\cM_{\text{splayed}}: = \bigcup_{\substack{(B'; \arrowrho'_\boat, \sigma'_\boat; \arrowrho'_\anchor, \sigma'_\anchor)\\ \text{ splays }\\ (B; \arrowrho_\boat, \sigma_\boat; \arrowrho_\anchor, \sigma_\anchor)}} \cM^{B'}(\bx, \btheta_+(k-1), \by; S'; \arrowrho'_\boat, \arrowrho'_\anchor; \sigma'_\anchor, \sigma'_\boat; \bJ).
$$
 are identified.

It follows that if we sum the number of elements of moduli spaces of both types over all $B \in \pi_2^{\btheta_+}(\bx, \by)$, we get the same number modulo 2. These are the coefficients of $\by$ in the maps  $n_{i_1, \dots, i_h}^k(\bx, \arrowrho)$, defined with respect to $\bJ^T$, and of $\by$ in $n_{i_1, \dots, i_h}^{k+1}(\bx, \arrowrho)$, defined with respect to $\bJ$, respectively. It follows that the $\cA_\infty$ modules $\poly_k(\border\cD, \bJ^T)$ and $\poly_k(\border\cD, \bJ)$ are isomorphic as $\cA_\infty$ modules.

By Proposition \ref{changeofcomplexstructure}, $\poly_k(\border\cD, \bJ^T)$ is $\cA_\infty$ homotopy equivalent to $\poly_k(\border\cD, \bJ)$, and so the proposition follows.
\end{proof}

We finally state and prove the main theorem of this paper:
\begin{thm}$\poly(\border \cD)$ is an $\cA_\infty$ module, and $\poly(\border\cD)$ is $\cA_\infty$ homotopy equivalent to $BSA(\border \cD)$.
\end{thm}
\begin{proof}Let $\kappa$ denote the constant defined in Proposition \ref{kappabound}. The only nonzero maps in $\poly_\kappa$ are determined by moduli spaces  $\cM^B(\bx, \btheta_+(k), \by; S; \emptyset, \arrowrho_\anchor; \emptyset, \sigma_\anchor; \bJ)$, which, by Proposition \ref{eastcompactifytoclosed}, are all identified with the moduli spaces used in the definition of the structure maps of $\poly(\border\cD)$. As such, $\poly(\border\cD)$ is an $\cA_\infty$-module which is $\cA_\infty$ isomorphic to $\poly_\kappa(\border\cD)$. Moreover, by Theorem \ref{polyequivalence},
$$BSA(\border\cD) \equiv \poly_0(\border\cD) \cong \poly_1(\border\cD) \cong \cdots \cong \poly_{\kappa -1}(\border\cD) \cong \poly_\kappa(\border\cD) \equiv \poly(\border\cD).$$
\end{proof}

\bibliographystyle{alpha}
\bibliography{library}

\newcommand{\etalchar}[1]{$^{#1}$}
\begin{thebibliography}{LOT14b}

\bibitem[Ahl10]{ahlforsconformal}
Lars~Valerian Ahlfors.
\newblock {\em Conformal invariants: topics in geometric function theory},
  volume 371.
\newblock American Mathematical Soc., 2010.

\bibitem[Aur]{auroux:fukaya}
Denis Auroux.
\newblock {F}ukaya categories of symmetric products and bordered
  {H}eegaard--{F}loer homology, {J}. {G}{\"o}kova {G}eometry {T}opology 4
  (2010), 1--54.
\newblock {\em arXiv preprint arXiv:1001.4323}.

\bibitem[BEH{\etalchar{+}}]{EBHWZ}
F~Bourgeois, Y~Eliashberg, H~Hofer, K~Wysocki, and E~Zehnder.
\newblock Compactness results in symplectic field theory.

\bibitem[Bou02]{bourgeois:morse}
Fr\'{e}d\'{e}ric Bourgeois.
\newblock {\em A {M}orse-{B}ott approach to contact homology}.
\newblock PhD thesis, {S}tanford university, 2002.

\bibitem[Eft16]{eftekhary:splicing}
Eaman Eftekhary.
\newblock Floer homology and splicing knot complements.
\newblock {\em Algebraic \& Geometric Topology}, 15(6):3155--3213, 2016.

\bibitem[Han13]{Hanselman:graph}
Jonathan Hanselman.
\newblock Bordered {H}eegaard {F}loer homology and graph manifolds.
\newblock 10 2013.

\bibitem[HLS97]{hoferlizansikorav}
Helmut Hofer, V{\'e}ronique Lizan, and Jean-Claude Sikorav.
\newblock On genericity for holomorphic curves in four-dimensional
  almost-complex manifolds.
\newblock {\em Journal of Geometric Analysis}, 7(1):149--159, 1997.

\bibitem[Hoc18]{me:satellite}
Thomas Hockenhull.
\newblock A satellite formula for link {F}loer homology.
\newblock {\em In preparation}, 2018.

\bibitem[Lev10]{Levine:doubles}
Adam~Simon Levine.
\newblock Knot doubling operators and bordered {H}eegaard {F}loer homology.
\newblock 08 2010.

\bibitem[Lip06]{Lipshitz:cylindrical}
Robert Lipshitz.
\newblock A cylindrical reformulation of {H}eegaard {F}loer homology.
\newblock {\em Geom. Topol.}, 10:955--1096, 2006.

\bibitem[Lip17]{Lipshitz:errata2}
Robert Lipshitz.
\newblock Further small corrections and explanations for ``{A} cylindrical
  reformulation of {H}eegaard {F}loer homology''.
\newblock {\em \url{http://pages.uoregon.edu/lipshitz/MoreCylErat.pdf}}, 2017.

\bibitem[LOT08]{LOT}
Robert Lipshitz, Peter Ozsv\'{a}th, and Dylan Thurston.
\newblock Bordered {H}eegaard {F}loer homology: Invariance and pairing.
\newblock 10 2008.

\bibitem[LOT10a]{Lipshitz:bimodules}
Robert Lipshitz, Peter~S. Ozsvath, and Dylan~P. Thurston.
\newblock Bimodules in bordered {H}eegaard {F}loer homology.
\newblock 03 2010.

\bibitem[LOT10b]{LOT:mappingclasses}
Robert Lipshitz, Peter~S. Ozsv\'{a}th, and Dylan~P. Thurston.
\newblock Computing $\widehat{HF}$ by factoring mapping classes.
\newblock 10 2010.

\bibitem[LOT14a]{LOT:SS}
Robert Lipshitz, Peter~S Ozsv\'{a}th, and Dylan~P Thurston.
\newblock Bordered {F}loer homology and the spectral sequence of a branched
  double cover {I}.
\newblock {\em Journal of Topology}, 7(4):1155--1199, 2014.

\bibitem[LOT14b]{LOT:SSII}
Robert Lipshitz, Peter~S. Ozsv\'{a}th, and Dylan~P. Thurston.
\newblock Bordered {F}loer homology and the spectral sequence of a branched
  double cover {II}: the spectral sequences agree.
\newblock 04 2014.

\bibitem[MO10]{Manolescu:linksurgeries}
Ciprian Manolescu and Peter Ozsv\'{a}th.
\newblock {H}eegaard {F}loer homology and integer surgeries on links.
\newblock 11 2010.

\bibitem[MS04]{McDS:jhol}
Dusa McDuff and Dietmar Salamon.
\newblock {\em J-holomorphic curves and symplectic topology}, volume~52.
\newblock American Mathematical Society Providence, RI, 2004.

\bibitem[OS01]{Ozsvath-Szabo:2001}
Peter Ozsv\'{a}th and Zolt\'{a}n Szab\'{o}.
\newblock Holomorphic disks and topological invariants for closed
  three-manifolds.
\newblock 2001.

\bibitem[OS03]{Ozsvath-Szabo:2003}
Peter Ozsv\'{a}th and Zolt\'{a}n Szab\'{o}.
\newblock Holomorphic disks and knot invariants.
\newblock 2003.

\bibitem[OS04]{Ozsvath-Szabo:integer}
Peter Ozsv\'{a}th and Zolt\'{a}n Szab\'{o}.
\newblock Knot {F}loer homology and integer surgeries.
\newblock 2004.

\bibitem[OS05]{Ozsvath-Szabo:BDC}
Peter Ozsv{\'a}th and Zolt{\'a}n Szab{\'o}.
\newblock On the heegaard floer homology of branched double-covers.
\newblock {\em Advances in Mathematics}, 194(1):1--33, 2005.

\bibitem[OS11]{Ozsvath-Szabo:rational}
Peter Ozsv\'{a}th and Zolt\'{a}n Szab\'{o}.
\newblock Knot {F}loer homology and rational surgeries.
\newblock {\em Algebr. Geom. Topol.}, 11(1):1--68, 2011.

\bibitem[Ras03]{Rasmussen:03}
Jacob Rasmussen.
\newblock {\em {F}loer homology and knot complements}.
\newblock PhD thesis, Harvard University, 2003.

\bibitem[Sei06]{Seidel:Fukaya}
Paul Seidel.
\newblock {F}ukaya categories and {P}icard-{L}efschetz theory.
\newblock {\em {ETH} Lecture Notes}, 2, 2006.

\bibitem[Sei08]{seidel:dehn}
Paul Seidel.
\newblock Lectures on four-dimensional {D}ehn twists.
\newblock {\em Lecture notes in mathematics}, 1938:231, 2008.

\bibitem[Zar09]{Zarev}
Rumen Zarev.
\newblock Bordered {F}loer homology for sutured manifolds.
\newblock 08 2009.

\end{thebibliography}
\end{document}